\documentclass[11pt]{amsart}
\usepackage{amsmath,amsthm,amssymb,mathrsfs}

\newtheorem{lemma}{Lemma}[section]

\newtheorem{proposition}[lemma]{Proposition}
\newtheorem{corollary}[lemma]{Corollary}
\newtheorem{theorem}[lemma]{Theorem}
\newtheorem{remark}[lemma]{Remark}
\newtheorem{example}[lemma]{Example}

\numberwithin{equation}{section}

\newenvironment{mlist}{\list{}{\listparindent 0pt
\itemsep 0pt \parsep  2pt \topsep 1pt
}}{\endlist}

\newcommand{\ad}{\operatorname{ad}}
\newcommand{\Ad}{\operatorname{Ad}}
\newcommand{\tr}{\operatorname{tr}}
\newcommand{\Ric}{\operatorname{Ric}}

\newcommand{\card}{\operatorname{card}}

\newcommand{\End}{\operatorname{End}}
\newcommand{\Id}{{\operatorname{Id}\kern.4pt}}
\newcommand{\rank}{\operatorname{rank}}

\newcommand{\eqdef}{\stackrel{\mathrm{def}}{=}}
\newcommand{\tith}{{\widetilde\theta}}
\newcommand{\tio}{{\widetilde\omega}}

\newcommand{\tw}{{*}}

\newcommand{\fa}{\mathfrak{a}}
\newcommand{\fb}{\mathfrak{b}}
\newcommand{\fg}{\mathfrak{g}}
\newcommand{\fh}{\mathfrak{h}}
\newcommand{\fk}{\mathfrak{k}}
\newcommand{\fm}{\mathfrak{m}}
\newcommand{\ft}{\mathfrak{t}}
\newcommand{\fn}{\mathfrak{n}}
\newcommand{\fz}{\mathfrak{z}}

\newcommand{\cK}{\mathcal{K}}
\newcommand{\cR}{\mathcal{R}}
\newcommand{\cS}{\mathcal{S}}

\newcommand{\rd}{\mathrm{d}}
\newcommand{\ri}{\mathrm{i}}
\newcommand{\ai}{{i}}

\newcommand{\bbR}{\mathbb{R}}

\newcommand{\bbC}{\mathbb{C}}

\newcommand{\bfg}{\mathbf{g}}

\newcommand{\sF}{\mathscr{F}}
\newcommand{\sH}{\mathscr{H}}
\newcommand{\sK}{\mathscr{K}}

\title[Ricci-flat K\"ahler structures]
{Invariant Ricci-flat K\"ahler metrics on tangent bundles
of compact symmetric spaces}

\thanks{Research supported by the Ministry of Economy, Industry
and Competitiveness, Spain, under Project MTM2016-77093-P.}

\author[P.~M. Gadea]{P.~M. Gadea}
\address{Instituto de F\'\i sica Fundamental, CSIC,
Serrano 113 bis, 28006-Madrid, Spain.}
\email{p.m.gadea@csic.es }
\author[J.~C.~Gonz{\'a}lez-D{\'a}vila]{J.~C.~Gonz{\'a}lez-D{\'a}vila}
\address{Departamento de Matem\'aticas, Estad\'istica e Investigaci\'on
Ope\-ra\-tiva, University of La Laguna, 38200 La Laguna, Tenerife, Spain.}
\email{jcgonza@ull.es}
\author[I.~V. Mykytyuk]{I.~V. Mykytyuk}
\address{Institute of Mathematics, Cracow University of Technology,
Warszawska 24, 31155, Cracow, Poland. \newline
\indent Institute of Applied Problems of Mathematics and Mechanics,
Naukova Str. 3b, 79601, Lviv, Ukraine.}
\email{mykytyuk{\_\,}i@yahoo.com}


\begin{document}
\maketitle

\begin{abstract}
We give a description of  all
$G$-invariant Ricci-flat K\"ahler metrics on the canonical
complexification of any compact Riemannian symmetric space
$G/K$ of arbitrary rank, by using
some special local $(1,0)$ vector fields on
$T(G/K)$. As the simplest application, we obtain the
explicit description of the set of all complete
$\mathrm{SO}(3)$-invariant Ricci-flat K\"ahler metrics on
$T{\mathbb S}^2$, which includes the well-known
Eguchi-Hanson-Stenzel metrics and a new one-parameter family
of metrics.
\end{abstract}

\section{Introduction}
\label{s.1}
As it is well known, the existence of Ricci-flat K\"ahler metrics
on either compact or non-compact K\"ahler manifolds is very different.
Given a compact K\"ahler manifold whose first Chern class
is zero, by Yau's solution of Calabi's conjecture, there
is a unique Ricci-flat K\"ahler metric in the original K\"ahler class.
If the K\"ahler manifold
is not compact, the situation is completely different,
and it could in principle admit many of such metrics, even complete metrics.
There is  not for now a  general existence
theorem for Ricci-flat K\"ahler metrics on
non-compact K\"ahler manifolds.

Over the latest decades there has been considerable interest
in Ricci-flat K\"ahler metrics whose underlying manifold is
diffeomorphic to the tangent bundle
$T(G/K)$ of a rank-one compact Riemannian symmetric space
$G/K$. For instance, a remarkable class of Ricci-flat
K\"ahler manifolds of cohomo\-geneity one was discovered by
M.~Stenzel~\cite{St}. This has originated an extensive
series of papers. To cite but a few: M.~Cveti\v c,
G.\,W.~Gibbons, H.~L\"u and C.\,N.~Pope~\cite{CGLP} studied
certain harmonic forms on these manifolds and found an
explicit formula for the Stenzel metrics in terms of
hypergeometric functions. Earlier, T.\,C.~Lee~\cite{Le} gave
an explicit formula of the Stenzel metrics for classical spaces
$G/K$ but in another vein, using the approach of G.~Patrizio
and P.~Wong~\cite{PW}. J.\,M.~Baptista~\cite{Ba} used the
Stenzel metrics on $\mathrm{SL}(2,{\mathbb C})\cong T(\mathrm{SU}(2))$
for holomorphic quantization of the classical symmetries of
the metrics. A.\,S.~Dancer and I.\,A.\,B.~Strachan~\cite{DS}
gave a much more elementary and concrete treatment in the case that
$G/K$ is the round sphere ${\mathbb S}^n=\mathrm{SO}(n+1)/\mathrm{SO}(n)$,
exploiting the fact that the Stenzel metrics are of
cohomogeneity one with respect to the natural action of the Lie group
$G$ on $T(G/K)$. Remark also that in the case of the
standard sphere ${\mathbb S}^2$, the Stenzel metrics coincide with the
well-known Eguchi-Hanson metrics~\cite{EH}.

The natural question arises on a construction of
$G$-invariant Ricci-flat K\"ahler metrics (all metrics in
as many cases as possible) on the tangent bundles of compact
Riemannian symmetric spaces
$G/K$ of any rank or, equivalently, on the complexification
$G^{\mathbb C}/K^{\mathbb C}$ (for the latter spaces, see
G.\,D.\,Mostow~\cite{Mo1,Mo2}). The most general existence
theorems to date are due to H.~Azad and
R.~Kobayashi~\cite{AK} and R.~Bielawski~\cite{Bl}.
These results are non-constructive in nature and rely on
non-linear analysis. At this moment, explicit expressions
for such metrics have been found only when
$G/K$ is an Hermitian symmetric space (see O.~Biquard and
P.~Gauduchon~\cite{BG1,BG2}, where these metrics are
hyper-K\"ahlerian, thus automatically Ricci-flat). Note that
for the simplest case,
$G/K={\mathbb C}{\mathbf P}^n$,
there is also an explicit formula for these metrics by
E.~Calabi~\cite{Ca} giving the K\"ahler form of
$T(G/K)$ as the sum of the pull-back of the K\"ahler form on
${\mathbb C}{\mathbf P}^n$ and a term given by an explicit
potential.

In the present paper we obtain a
description, reached
in our main theorem (Theorem~\ref{th.5.1}) of such metrics
for compact Riemannian symmetric
spaces of any rank, as we outline with some more detail in
the next paragraph.

Let $G/K$ be a homogeneous manifold,
$G$ being a connected, compact Lie group. The tangent bundle
$T(G/K)$ has a canonical complex structure
$J^K_c$ coming from the $G$-equivariant diffeomorphism
$T(G/K) \to G^\bbC/K^\bbC$. The latter space is the
complexification of
$G/K$ mentioned above. Our approach is based on the explicit
algebraic description of some special local $(1,0)$ vector
fields defined on an open subset
of $T(G/K)$ (see Lemma \ref{le.3.5}). These vector fields
determine, for each $G$-invariant K\"ahler metric ${\mathbf g}$ on
$T(G/K)$ associated to $J^K_c$, a
$G$-{\it invariant\/} function
$\cS \colon T(G/K) \to \bbC$ so that the Ricci form
$\Ric({\mathbf g})$ of
${\mathbf g}$ can be expressed (Proposition \ref{pr.3.6}) as
$
 \Ric({\mathbf g}) = -\ri \,\partial\bar\partial \ln \cS.
$

Then, using the root theory of symmetric spaces,
we can describe, for $G/K$ being
any Riemannian symmetric space of compact type, all
$G$-invariant K\"ahler structures
$({\mathbf g}, J^K_c)$ which moreover are Ricci-flat on an
open dense subset
\linebreak
$T^+(G/K)$ of $T(G/K)$. Here,
$T^+(G/K)$ is the image of $G/H \times W^+$ under
a certain $G$-equivariant diffeomorphism, where
$W^+$ is some Weyl chamber and
$H$ denotes the centralizer of a (regular) element of
$W^+$ in $K$. Such $G$-invariant
K\"ahler and Ricci-flat K\"ahler structures are determined
uniquely by a vector-function ${\mathbf a}\colon W^+ \to \fg_H$
satisfying certain conditions (Theorem \ref{th.5.1}),
$\fg_H$ being the subalgebra of
$\Ad(H)$-fixed points of the Lie algebra of $G$.

We also give (in Section \ref{s.6}) its simplest
application; namely, we describe, in terms of our vector-functions
${\mathbf a}\colon \bbR^+ \to \mathfrak{so}(3)$, the set of all
$G$-invariant Ricci-flat K\"ahler metrics
$(\mathbf{g}, J_c^K)$ on the punctured tangent bundle
$T^+{\mathbb S}^2 = T{\mathbb S}^2 \backslash \{
\mathrm{zero\; section}\}$
of ${\mathbb S}^2=\mathrm{SO}(3)/\mathrm{SO}(2)$ and, among them, those that
extend to smooth complete metrics on the whole tangent
bundle. This family of
$\mathrm{SO}(3)$-invariant Ricci-flat K\"ahler metrics on
$T{\mathbb S}^2$ includes the well-known
$\partial\bar\partial$-exact Eguchi-Hanson
(hyper-K\"ahlerian) metrics~\cite{EH}
reopened by M.~Stenzel~\cite{St} and a new family of metrics
which are not $\partial\bar\partial$-exact.

In our next paper~\cite{GGM} we give an {\it explicit\/}
description, by using the technique of this article and the
main Theorem~\ref{th.5.1}, of the set of all
$G$-invariant Ricci-flat K\"ahler metrics
$(\mathbf{g}, J_c^K)$ on $T(G/K)$, where $G/K$
is a compact rank-one symmetric space. It is also shown that
this set contains a new family of metrics which are not
$\partial\bar\partial$-exact if
$G/K \in \{{\mathbb C}{\mathbf P}^n, n\geqslant 1\}$,
and coincides with the set of Stenzel metrics for any of the
latter spaces $G/K$.

\section{Preliminaries}
\label{s.2}

\subsection{Invariant polarizations}
\label{ss.2.1}

Let $\widetilde{M}$ be a smooth real manifold such that two real
Lie groups
$G$ and $K$ act on it and suppose that these actions commute
and the action of $K$ on
$\widetilde{M}$ is free and proper. Then the orbit space
$M=\widetilde{M}/K$ is a well-defined smooth manifold and
the projection mapping
$\tilde\pi \colon\widetilde{M}\to M$ is a principal
$K$-bundle. Since the actions of $G$ and $K$ on
$\widetilde{M}$ commute, there exists a unique action of
$G$ on $M$ such that the mapping $\tilde\pi$ is
$G$-equivariant.

Let $\sK\subset T\widetilde{M}$ be the kernel of the tangent map
$\tilde\pi_* \colon T \widetilde{M} \to TM$.  Then
$\sK$ is an involutive subbundle (of rank $\dim K$) of the tangent bundle
$T\widetilde{M}$. Since the actions of $G$ and $K$ commute,
the subbundle $\sK$ is $G$-invariant.

Suppose that $(M,\omega)$ is a (smooth) symplectic
manifold with a $G$-invariant symplectic structure $\omega$.
Let $J$ be a $G$-invariant almost complex
structure on $M$ and let
$F(J)\subset T^{\mathbb   C}M$ be its complex subbundle of
$(1,0)$-vectors, that is,
${\Gamma}F(J)=\{Y- \ri JY,Y\in {\Gamma}(TM)\}$.
The pair $(\omega,J)$ is a K\"ahler structure on
$M$ if and only if the subbundle $F(J)$ is a positive-definite
polarization, i.e.\ $(a)$ the complex subbundle $F(J)$ of rank
$\frac 12\dim_\bbR  M$
is involutive; $(b)$ $F(J) \cap \overline{F(J)} = 0$;
$(c)$ $\omega(F(J),F(J))=0$ (it is Lagrangian); and $(d)$
$\ri \omega_x(Z,\overline{Z})>0$ for all
$x\in M$, $Z\in F_x(J)\setminus\{0\}$~(see V.\,Guillemin
and S.\,Sternberg~\cite[Lemma 4.3]{GS}).

In this case, the $2$-form $\omega$ is invariant with respect to
the automorphism $J$ of the real tangent bundle $TM$
and the bilinear form ${\mathbf g}={\mathbf g}(\omega,J)$, where
${\mathbf g}(Y,Z)\stackrel{\mathrm{def}}{=}\omega(JY,Z)$,
for all vector fields $Y,Z$ on $M$,
is symmetric and positive-definite.
It is clear that each positive-definite polarization $F$ on $(M,\omega)$
determines the K\"ahler structure $(\bfg,\omega,J)$
with complex tensor $J$ such that $F=F(J)$
and ${\mathbf g}={\mathbf g}(\omega,J)$.

Since $F$ is an involutive subbundle of
$T^{\mathbb C} M$, it is determined by the differential ideal
${\mathcal I}(F)\subset\Lambda T^{{\mathbb C}*} M$
(closed with respect to exterior differentiation). Then
$\tilde\pi^*({\mathcal I}(F))\subset\Lambda T^{{\mathbb C}*}\widetilde{M}$
is also a differential ideal and, consequently, its kernel
$\sF$ is an involutive subbundle of
$T^{\mathbb C}\widetilde{M}$. We will denote $\sF$ also by
$\tilde\pi_*^{-1}(F)$. This subbundle is uniquely
determined by two conditions:
(1)~$\dim_\bbC \sF= \frac12\dim_\bbR  M+\dim
K=\frac12 ( \dim \widetilde{M} + \dim K )$;
(2)~$\tilde\pi_*(\sF)= F$.
It is evident that $(\tilde\pi^*\omega) (\sF,\sF)=0$ and the subbundle
$\sF$ contains $\sK$. Moreover, $\sF$ is
$K$-invariant.

We can substantially simplify matters by working on the
manifold $\widetilde{M}$ with the subbundle
$\sF$ rather than on the manifold $M$ with the polarization
$F$.

\begin{lemma}\label{le.2.1}
Let $\widetilde{M}$ be a manifold with two commuting actions
of the Lie groups
$G$ and $K$. Suppose that the action of $K$ on
$\widetilde{M}$ is free and proper and let
$\tilde\pi \colon \widetilde{M}\to  M$, where $ M=\widetilde{M}/K$,
be the corresponding $G$-equivariant projection. Let $\omega$ be a
$G$-invariant symplectic structure on $M$.
Let $\sF$ be a $G$-invariant involutive complex subbundle of
$T^{\mathbb C}\widetilde{M}$ such that
\begin{mlist}
\item
[$(1)$]  $\sF$ is $K$-invariant;
\item
[$(2)$]  $ \sK^{\bbC} =\sF\cap\overline{\sF};$
\item
[$(3)$]  $\dim_\bbC \sF = \frac12\dim_\bbR M+\dim K;$
\item
[$(4)$]  $(\tilde\pi^* \omega) (\sF,\sF)=0;$
\item
[$(5)$]
$\ri (\tilde\pi^* \omega)_{\tilde x}(Z,\overline{Z})>0$ for all
$\tilde x\in \widetilde{M}$, $Z\in \sF_{\tilde x}\setminus
\sK^{\bbC}_{\tilde x}$.
\end{mlist}

Then
$F=\tilde\pi_*(\sF)$ is a positive-definite polarization on
$( M,\omega)$, i.e., there exists a unique
K\"ahler structure $(\bfg,\omega,J)$ on $ M$
such that $F=F(J)$ and ${\mathbf g}={\mathbf g}(\omega,J)$.

Conversely, any positive-definite
$G$-invariant polarization $F$ on
$( M,\omega)$ determines a unique $G$-invariant
involutive subbundle $\sF=\tilde\pi_*^{-1}(F)$ with
properties {\em (1)--(5)}.
\end{lemma}
\begin{proof}
The proof coincides up to some simple changes with that
of Mykytyuk~\cite[Lemma 3]{My}. Since
${\mathcal F}$ is $K$-invariant and the kernel
${\sK}$ of $\tilde\pi_*$ is contained in
${\mathcal F}$, the image $F=\tilde\pi_*({\mathcal F})$ of
${\mathcal F}$ is a well-defined subbundle of
$T^{\mathbb C}M$ of rank
$\frac12\dim_\bbR M$. We have $F\cap \overline F=0$ because
${\sK}^{\bbC}={\mathcal F}\cap \overline{\sF}$.
It then immediately follows from (4)
that the subbundle $F$ is Lagrangian.
To prove the smoothness and involutiveness of
$F$ we notice that $\tilde\pi$ is a submersion, i.e.\ for any point
$\tilde x\in \widetilde{M}$ there exists a convex neighborhood
$\widetilde U$ of ${\tilde x}$, coordinates
$x_1,\dotsc , x_{\widetilde N}$ on
$\widetilde U$ and coordinates
$x_1,\dotsc , x_{N}$ on the open subset
$U=\tilde\pi(\widetilde U)$ such that
$x_j({\tilde x})=0$,
$j=1, \dotsc ,\widetilde N$, and in these coordinates
$\tilde\pi|_{\widetilde U}$ is of type
$\tilde\pi \colon (x_1,x_2,\dotsc ,x_{\widetilde N})
\mapsto (x_1,x_2,\dotsc ,x_{N})$.
We can choose $\widetilde U$ such that
$(x_1,\dotsc,x_{N},0,\dotsc ,0)\in \widetilde U$
if $(x_1,\dotsc,x_{\widetilde N})\in \widetilde U$. Let
$Y(x_1,\dotsc ,x_{\widetilde N})= \sum _{j=1}^{\widetilde N}
a_{j}(x_1,\dotsc ,x_{\widetilde N})\partial
/\partial x_j $ be any section of
${\mathcal F}|_{\widetilde U}$. The subbundle
${\sK}$ is spanned on
$\widetilde U$ by $\partial /\partial x_j, j=N+1,\dotsc ,\widetilde N$,
and ${\mathcal F}|_{\widetilde U}$ is preserved by these
$\partial /\partial x_j $. Therefore, the smooth vector field
\begin{equation}\label{eq.2.1}
Y_0(x_1,\dotsc,x_{\widetilde N})=\sum_{j=1}^{\widetilde N}
a_{j}(x_1,\dotsc,x_{N},0, \dotsc ,0)\dfrac{\partial}{\partial x_j},
\end{equation}
is also a section of
${\mathcal F}|_{\widetilde U}$
(${\mathcal F}$ is preserved by $\partial /\partial x_j$
if and only if ${\mathcal F}$ is
preserved by the corresponding local one-parameter
group~\cite[Prob.\ 2.56,~p.\ 124]{GMM}).
Thus, $\tilde\pi_*\bigl(Y_0
(x_1, \dotsc ,x_{\widetilde N})\bigr)$ $=\sum_{j=1}^{N}a_{j}(x_1, \dotsc ,x_{N},0,
\dotsc ,0)\partial/\partial x_j$
is a smooth section of
$F|_{U} $. The involutiveness of $F$ follows
easily from~(\ref{eq.2.1}). Now, it follows from $(5)$
that $F$ is a positive-definite polarization.
\end{proof}

\subsection{Invariant Ricci-flat K\"ahler metrics}
\label{ss.2.2}

Let $G$ be a Lie group acting on the manifold $M$.
Let $J$ be some $G$-invariant complex structure
on $M$.  To substantially simplify matters
we will work on $M$
with the $G$-invariant K\"ahler (symplectic) form
$\omega$  rather than with the metric~$\bfg$:
$$
\bfg(X,Y)=\omega(JX,Y),\quad
\omega(X,Y)=-\bfg(JX,Y),\quad
\forall X,Y\in{\Gamma}(TM).
$$
Let $\dim M = 2n$ and
$z_1,\dotsc,z_n$ be some local complex coordinates
on
$(M, J)$. Then
$
\omega=\sum_{1\leqslant j,s\leqslant n}
w_{js}\rd z_j\land \rd\bar z_{s}.
$
In particular,
$w_{js}=\omega(\partial/\partial z_j, \partial/\partial \bar z_{s})$
and $w_{js}=-\overline{w_{s j}}$, that is, the matrix
$(\omega_{js})$ is skew-Hermitian. The Ricci form $\Ric(\bfg)$
(corresponding to the Ricci curvature) of the metric
$\bfg$ is the (global) form given in the local coordinates
$z_1,\dotsc,z_n$~(see \cite[Ch.\ IX, \S 5]{KN}) by
$
\Ric(\bfg)=-\ri\partial\bar\partial\ln\det(w_{j s}).
$
The right-hand side does not depend on the choice of local
complex coordinates. The Ricci form $\Ric(\bfg)$ is
$G$-invariant because so are
the complex structure $J$ and the form $\omega$.

Let $X_1, \dotsc, X_n$ be some linearly
independent $J$-holomorphic vector fields defined on some open
dense subset
$O$ of $M$. Using these holomorphic (possibly non-global)
vector fields $X_j$, $j=1, \dotsc, n$, we can calculate the function
$\det(w_{js})$ on the subset $O\subset M$
up to multiplications by some holomorphic and some
anti-holomorphic functions. Indeed, putting
$
W_{js}=\omega(X_j,\overline{X_{s}}),
$
we obtain that locally
$\det(W_{js})=\text{\sl h} \cdot\det(w_{js})
\cdot\overline{\text{\sl h}}$, where
$\text{\sl h}$ is some non-vanishing local holomorphic function
(specifically, some determinant). Thus
\begin{equation}\label{eq.2.2}
\Ric(\bfg)=- \ri \partial\bar\partial\ln\det(W_{js}).
\end{equation}

\section{The canonical complex structure on $T(G/K)$}
\label{s.3}

Consider a homogeneous manifold $G/K$, where $G$ is a
compact connected Lie group and $K$ is
some  closed  subgroup of $G$.
Let ${\mathfrak g}$ and
${\mathfrak k}$ be the Lie algebras of $G$ and $K$. There exists
a positive-definite $\Ad (G)$-invariant form
$\langle\cdot ,\cdot \rangle$ on ${\mathfrak g}$.

Denote by
${\mathfrak m}$ the
$\langle\cdot ,\cdot \rangle$-orthogonal complement to
${\mathfrak k}$ in ${\mathfrak g}$, that is,
$
{\mathfrak g}= {\mathfrak m} \oplus {\mathfrak k}
$
is the $\Ad(K)$-invariant vector space direct sum decomposition of
${\mathfrak g}$. Consider the trivial vector bundle
$G\times {\mathfrak m}$ with the two Lie group actions
(which commute) on it:
the left
$G$-action, $l_h \colon (g,w)\mapsto (hg,w)$ and the right
$K$-action $r_k \colon (g,w)\mapsto (gk,\Ad_{k^{-1}}w)$. Let
$$
\pi \colon G\times {\mathfrak m}\to G\times_K {\mathfrak m},
\quad (g,w)\mapsto [(g,w)],
$$
be the natural projection for this right
$K$-action. This projection is
$G$-equivariant. It is well known that
$G\times_K {\mathfrak m}$ and
$T(G/K)$ are diffeomorphic. The corresponding
$G$-equivariant diffeomorphism
\begin{equation}\label{eq.3.1}
\phi \colon G\times_K {\mathfrak m}\to T(G/K),
\quad
[(g,w)]\mapsto \frac{\rd}{\rd t}\Bigr|_0 g\exp(tw)K,
\end{equation}
and the projection $\pi$ determine
the $G$-equivariant submersion
$\Pi = \phi\circ\pi \colon G\times{\mathfrak m}\to T(G/K)$.
It is clear that
there exists a sufficiently small neighborhood
$O_\fm\subset{\mathfrak m}$ of zero in $\fm$
and an open subset $O\subset T(G/K)$ containing the
whole linear subspace $T_o(G/K)$, $o=\{K\}$,
such that the restriction of the map $\Pi$
\begin{equation}\label{eq.3.2}
\Pi|_{\exp(O_\fm)\times\fm}\colon \exp(O_\fm)\times\fm\to O
\end{equation}
is a diffeomorphism. We will use this special local section
$\exp(O_\fm)\times\fm\subset G\times{\mathfrak m}$
of the projection $\Pi$ in our
calculations below.

\begin{remark}
\em
In the case when we consider simultaneously
different homogeneous manifolds $G/K$ with the same
Lie group $G$ we will denote the mappings $\Pi,\pi$ and $\phi$
by $\Pi_K,\pi_K$ and $\phi_K$, respectively, and the $K$-orbit
$[(g,w)]$ of the element $(g,w)\in G\times \fm$  by $[(g,w)]_K$.
\end{remark}

Let $G^{\mathbb C}$ and $K^{\mathbb C}$ be the complexifications
of the Lie groups $G$ and $K$.
In particular, $K$ is a maximal compact subgroup of the Lie group
$K^{\mathbb C}$ and the intersection of $K$ with each connected
component of
$K^{\mathbb C}$ is not empty~(note that
$G^{\mathbb C}$, $K^{\mathbb C}$, $G$ and $K$ are algebraic groups).
Let $\fg^\mathbb{C}=\fg\oplus\ai\fg$ and $\fk^\mathbb{C}=\fk\oplus\ai\fk$
be the complexifications of the compact Lie algebras
$\fg$ and $\fk$.

We denote by
$G^{\mathbb R}$ and $K^{\mathbb R}$ the groups $G^{\mathbb C}$ and
$K^{\mathbb C}$, respectively, considered as real Lie groups.
Denote by $\xi_h^{\mathbf l}$, $\xi\in\fg$ (resp.\ $\xi_h^{\mathbf r}$)
the left (resp.\ right) $G^\bbC$-invariant (holomorphic) vector
fields on $G^\bbC$.
The natural (canonical) complex structure
$J_c$ on $G^{\bbR}=G^{\bbC}$ is defined by the right
$G^{\bbR}$-invariant $(1,0)$ vector fields
$\xi_h^{\mathbf r}=\xi^{r}- \ri (I \xi)^{r}$,
$\xi\in{\mathfrak g}$, where $I$ is a complex structure
(in particular, $I$ can be taken as the
multiplication by $\ai$)
on ${\mathfrak g}^{\mathbb C}$ and $\xi^{r}$ and $(I \xi)^{r}$ are
the right $G^{\mathbb R}$-invariant vector fields on the real Lie
group $G^\mathbb{R}$ obtained from $\xi$ and $I\xi$, respectively.
In turn, we denote by $\xi^l$ the corresponding left
$G^{\mathbb R}$-invariant
vector field on $G^{\mathbb R}$. Note here that, when
dealing with vector fields on the Lie group $G$,
we will use the same notation, i.e. $\xi^l$,
for the left $G$-invariant vector field on $G$
corresponding to $\xi\in\fg$.

Consider the complex homogeneous manifold
$G^\mathbb{C}/K^\mathbb{C}$ and the canonical holomorphic projection
$p_h \colon G^\mathbb{C}\to G^\mathbb{C}/K^\mathbb{C}$.
Since the vector field $\xi_h^{\mathbf r}$ on $G^\mathbb{C}$ is
right $K^\mathbb{C}$-invariant and $p_h$ is a holomorphic
submersion, its image $p_{h*}(\xi_h^{\mathbf r})$ is a well-defined
holomorphic vector field on the complex manifold
$G^\mathbb{C}/K^\mathbb{C}$.

Identifying $G^\mathbb{C}/K^\mathbb{C}$ naturally with the real
homogeneous manifold $G^\mathbb{R}/K^\mathbb{R}$ we obtain on
$G^\mathbb{R}/K^\mathbb{R}$ the canonical
left $G^\mathbb{R}$-invariant complex structure
$J^K_c$. This structure is defined by the global
$(1,0)$ vector fields
$p_{h*}(\xi_h^{\mathbf r})=p_*(\xi^{r})- \ri p_*(I \xi)^{r}$,
$\xi\in{\mathfrak g}$,
where $p$ is the canonical  projection
$p \colon G^\mathbb{R}\to G^\mathbb{R}/K^\mathbb{R}$.

Since
$G$ and $K$ are maximal compact Lie subgroups of
$G^{\mathbb C}$ and $K^{\mathbb C}$,
respectively, by a result of Mostow~\cite[Theorem~4]{Mo1}, we have that
$K^{\mathbb C}=K\exp(\ai  {\mathfrak k})$,
$G^{\mathbb C}=G\exp(\ai {\mathfrak m})\exp(\ai {\mathfrak k})$,
and the mappings
\begin{equation}
\label{eq.3.3}
\begin{split}
G\times {\mathfrak m}\times {\mathfrak k} &\to G^{\mathbb C}, \qquad
(g,w,\zeta) \mapsto g\exp(\ai w)\exp(\ai \zeta), \\
K\times {\mathfrak k} &\to K^{\mathbb C},  \qquad
(k,\zeta) \mapsto k\, \exp (\ai \zeta),
\end{split}
\end{equation}
are diffeomorphisms. Then the map
\begin{equation}\label{eq.3.4}
f^{\times}_K \colon G^{\mathbb C}/K^{\mathbb C}
\to G\times_K {\mathfrak m},
\qquad
g\exp(\ai w)\exp(\ai \zeta)K^{\mathbb C}\mapsto [(g,w)],
\end{equation}
is a $G$-equivariant diffeomorphism~\cite[Lemma~4.1]{Mo2}.
It is clear that
\begin{equation}\label{eq.3.5}
f_K \colon G^{\mathbb C}/K^{\mathbb C} \to T(G/K),
\qquad
g\exp( \ai w)\exp( \ai \zeta)K^{\mathbb C}\mapsto \Pi(g,w),
\end{equation}
is also a $G$-equivariant diffeomorphism. The diffeomorphism
$f_K$ supplies the manifold $T(G/K)$ with the
$G$-invariant complex structure $J_c^K$.
Moreover, this structure is determined by the set of global
holomorphic vector fields
$X_h^\xi$, $\xi\in\fg$, on
$T(G/K)$ which are images of the holomorphic vector fields
$p_*(\xi^{r})-  \ri p_*(I \xi)^{r}$,
under the tangent map $f_{K*}$.

\begin{lemma}\label{le.3.2}
Let $G/K$ be a homogeneous manifold,
where $G$ is a connected compact Lie group and $K$ is
a closed subgroup of $G$. Then for every $w\in \fm$
there exists a unique pair $(B_w^\fm,B_w^\fk)$
of ${\mathbb R}$-linear mappings
$B_w^\fm \colon {\mathfrak g}\to {\mathfrak m}$ and
$B_w^\fk \colon {\mathfrak g}\to {\mathfrak k}$ such that
\begin{equation}\label{eq.3.6}
\mathrm{Id} =\frac{\sin \ad_w}{\ad_w}\circ B_w^\fm
+(\cos \ad_w) \circ B_w^\fk
\quad\text{on the space}\quad\fg.
\end{equation}
The operator-functions
$$
B^\fm \colon \fm\to \End(\fg,\fm),\quad
w\mapsto B^\fm_w,\quad\text{and}\quad
B^\fk  \colon \fm\to \End(\fg,\fk),\quad
w\mapsto B^\fk_w,
$$
are smooth on $\fm$.
The global holomorphic vector field
$X_h^\xi$ on $T(G/K)=\Pi(G\times\fm)$, $\xi \in \fg$,
is the $\Pi_*$-image of the following global
vector field $X^\xi$ on $G\times\fm$,
\begin{equation}\label{eq.3.7}
X^\xi{(g,w)}{=}\left(
\Bigl(\xi'{-}\ri \frac{1{-}\cos \ad_w}{\ad_w}(B_w^\fm\xi')
{-}\ri \sin \ad_w (B_w^\fk\xi')\Bigr)^{l}(g),
{-}\ri (B_w^\fm\xi') \right),
\end{equation}
where $\xi'=\Ad_{g^{-1}}\xi$ and $T_{(g,w)}(G\times \fm)$
is identified naturally with the space
$T_g G\times\fm$.

If in addition the homogeneous manifold $G/K$ is a symmetric space,
we have the following exact solutions
of~{\em(\ref{eq.3.6}):}
\begin{equation}\label{eq.3.8}
B^\fm_w=\frac{\ad_w}{\sin \ad_w}\circ P_\fm,\qquad
B^\fk_w=\frac{1}{\cos \ad_w}\circ P_\fk,
\end{equation}
where $P_\fm \colon \fg\to\fm$ and $P_\fk \colon \fg\to\fk$ are the
natural projections determined by the splitting $\fg=\fm\oplus\fk$,
and therefore $X_h^\xi=\Pi_*\bigl(X^\xi\bigr)$, where
\begin{equation*}
X^\xi{(g,w)}=
\left(\Bigl(
\xi'- \ri\frac{1-\cos \ad_w}{\sin\ad_w}(P_\fm\xi')
-\ri\frac{\sin \ad_w}{\cos \ad_w} (P_\fk\xi')\Bigr)^{l}(g),
-\ri\frac{\ad_w}{\sin\ad_w}(P_\fm\xi')\right).
\end{equation*}
\end{lemma}
\begin{proof}
To prove the lemma we calculate
the components of
the image of the right $G^{\mathbb R}$-invariant vector field
$\xi^{r}- \ri (I \xi)^{r}$,
$\xi \in{\mathfrak g}$ on $G^\bbR$, under the
diffeomorphism
$$
G^{\mathbb C}\to G\times {\mathfrak m}\times {\mathfrak k},\quad
g\exp( \ai w)\exp( \ai \zeta)\mapsto (g,w,\zeta).
$$
To this end we consider two curves
$\exp(t\xi)g\exp( \ai w)$ and $\exp(t \ai \xi)g\exp( \ai w)$,
$t\in\mathbb{R}$,
in the group $G^\mathbb{C}=G^\mathbb{R}$ through the point $g\exp( \ai w)$
with tangent vectors
$\xi^{r}(g\exp( \ai w))$,
$(I \xi)^{r}(g\exp( \ai w))$, respectively
(here, $I\xi=\ai \xi$ for $\xi\in\fg$).
Using the diffeomorphism~(\ref{eq.3.3}) we obtain that
\begin{equation}\label{eq.3.9}
\exp(t\varepsilon \xi)g\exp( \ai w)=gg_\varepsilon(t)
\exp( \ai v_\varepsilon(t)) \exp( \ai k_\varepsilon(t)),
\end{equation}
where $\varepsilon\in\{1,\ai\}$ (hereafter
this means that either $\varepsilon = 1$ or
$\varepsilon = \ai$ all times in the formula),
$g_\varepsilon(0)=e$, $v_\varepsilon(0)=w$ and $k_\varepsilon(0)=0$
and $g_\varepsilon(t)$,
$v_\varepsilon(t)$ and $k_\varepsilon(t)$  are the (unique)
smooth curves in $G$, $\fm$  and $\fk$, respectively.
Hence, $X^\xi(g,w) \in T^\bbC_{(g,w)}(G \times \fm)
\cong T^\bbC_gG \times \fm^\bbC$ is given by
\begin{equation}
\label{eq.3.10}
X^\xi(g,w) = \big((g^\prime_1(0) - \ri g^\prime_\ai(0))^{l}(g), \,
v^\prime_1(0) - \ri v^\prime_\ai(0)\big).
\end{equation}
Moreover, transforming the curves~(\ref{eq.3.9}) to the curves
\begin{equation*}
\begin{split}
&\exp(- \ai w)g^{-1}\exp(t\varepsilon \xi)g\exp( \ai w)\\
&\qquad\quad =\bigl(\exp(- \ai w)g^{-1}gg_\varepsilon(t)\exp  \ai w)
\bigr)\bigl(\exp(- \ai w)\exp ( \ai v_\varepsilon(t))\bigr)
\exp( \ai k_\varepsilon(t))
\end{split}
\end{equation*}
through the identity in $G^\mathbb{C}$  and calculating their tangent
vectors at $e\in G^\mathbb{C}$, we obtain the following equation in
$\fg^\mathbb{C}=\fg\oplus\ai\fg$ for the tangent vectors
$g^\prime_\varepsilon(0)\in{\mathfrak g}$, $v^\prime_\varepsilon(0)
\in {\mathfrak m}$, $k^\prime_\varepsilon(0)\in {\mathfrak k}$:
\begin{equation}\label{eq.3.11}
\mathrm{e}^{- \ai \operatorname{ad} w}(\varepsilon\Ad_{g^{-1}}\xi)
=\mathrm{e}^{- \ai  \operatorname{ad} w }g^\prime_\varepsilon(0)+
\frac{1-\mathrm{e}^{-  \ai \operatorname{ad} w}}
{\ai \operatorname{ad}w}\ai v^\prime_\varepsilon(0) +
\ai k^\prime_\varepsilon(0),
\end{equation}
because $\displaystyle\left.\frac{\rd}{\rd t}\right|_0\exp(-X)\exp(X+tY)
=\frac{1-{\mathrm{e}}^{-\ad X}}{\ad X}(Y)$~(see  \cite[Ch.\ II, Theorem 1.7]{He}).
Since the map~(\ref{eq.3.3}) is a diffeomorphism, there exists
a unique solution $(g'(0),v'(0),k'(0)) = (g^\prime_\varepsilon(0),
v^\prime_\varepsilon(0), k^\prime_\varepsilon(0))\in\fg\times\fm\times\fk$
of Equation~(\ref{eq.3.11}) in $\fg^\mathbb{C}$,
for each $\varepsilon \in \{ 1, \ai\}$.
If $\varepsilon=1$, one directly gets that
$(g'(0),v'(0),k'(0))=(\Ad_{g^{-1}}\xi,0,0)$.
If $\varepsilon=  \ai $, we obtain one equation for~$\fg^\mathbb{C}$:
$$
 \ai \Ad_{g^{-1}}\xi=
g'(0)+
\frac{{\mathrm{e}}^{{ \ai \operatorname{ad}_w}}-1}
{\operatorname{ad}_w}(v'(0))+ \ai {\mathrm{e}}^{{ \ai
\operatorname{ad}_w}}k'(0),
$$
or two equations for $\fg$:
$\left\{\begin{array}{l}
\displaystyle
\Ad_{g^{-1}}\xi=\frac{\sin \ad_w}{\ad_w}v'(0)+\cos \ad_w k'(0),\\
\displaystyle
\hskip29pt 0=g'(0)+\frac{\cos \ad_w-1}{\ad_w}v'(0)-\sin \ad_w k'(0).
\end{array}\right.$ \\
As we remarked above, these equations possess a unique
solution. It is easy to see that the first of them defines
the operator-functions
$B^\fm \colon \fm\to \End (\fg,\fm)$, $w\mapsto B^\fm_w$, and
$B^\fk \colon \fm\to  \End  (\fg,\fk)$,
$w\mapsto B^\fk_w$, by $B^\fm_w(\Ad_{g^{-1}}\xi)=v'(0)$ and
$B^\fk_w(\Ad_{g^{-1}}\xi)=k'(0)$. Since the mapping
$G\times {\mathfrak m}\times {\mathfrak k} \to G^{\mathbb C}$
in~(\ref{eq.3.3}) is a diffeomorphism,
these operator-functions are smooth functions on $\fm$.

To complete the proof of the first part of our lemma, we
substitute  in \eqref{eq.3.10} the two
triples $(g^\prime_\varepsilon(0), v^\prime_\varepsilon(0),
k^\prime_\varepsilon(0))$ for $\varepsilon\in\{1, \ai \}$
calculated above.

To prove the second part of the lemma it is
sufficient to note that $\fg=\fm\oplus\fk$, $w\in\fm$, and
in the symmetric case the subspaces $\fm$  and $\fk$
are invariant subspaces of the operators $(\ad_w)^{2p}$, $p=0,1,2,...$
\end{proof}

\begin{remark}\label{re.3.3} \em
Since the map $\fm\to\mathrm{End}(\fg)$, $w\mapsto\ad_w$,
is $\Ad(K)$-equivariant, i.e.
$\Ad_k\circ \ad_w\circ\Ad_{k^{-1}}=\ad_{\Ad_k w}$, for all $k\in K$,
$\Ad(K)(\fm)=\fm$, $\Ad(K)(\fk)=\fk$, from the uniqueness
of the splitting~(\ref{eq.3.6}) we obtain the
$\Ad(K)$-equivariance of the maps $w\mapsto B^\fm_w$
and $w\mapsto B^\fk_w$:
\begin{equation}\label{eq.3.12}
\Ad_k\circ B^\fm_w\circ\Ad_{k^{-1}}=B^\fm_{\Ad_k w},
\qquad
\Ad_k\circ B^\fk_w\circ\Ad_{k^{-1}}=B^\fk_{\Ad_k w}.
\end{equation}
\end{remark}

\begin{remark}\label{re.3.4}
\em
Let $K_0\subset K$ be the identity component of the group $K$.
Then its complexification
$K_0^\mathbb{C}$ is also connected.
According to Stenzel~\cite[Lemma 2]{St}, there exists a
$G^{\mathbb C}$-invariant non-vanishing holomorphic
form $\Theta_h$ of maximal rank on the complex homogeneous space
$G^{\mathbb C}/K^{\mathbb C}_0$;
that is, the canonical bundle
$\Lambda^{(n,0)}(G^{\mathbb C}/K_0^{\mathbb C})$,
$n=\dim_\mathbb{C} G^{\mathbb C}/K_0^{\mathbb C}$,
is holomorphically trivial.
The existence of the form $\Theta_h$ relies on the fact that
as a group of transformations of $\fm^\mathbb{C}$ the group
$\Ad(K^{\mathbb C})|_{\fm^\mathbb{C}}$ is a subgroup of
the complex orthogonal
group $\mathrm{O}(\fm^\mathbb{C})$ but
$\Ad(K^{\mathbb C}_0)|_{\fm^\mathbb{C}}\subset \mathrm{SO}(\fm^\mathbb{C})$.
The $G$-equivariant diffeomorphism $f_{K_0}$ given in~\eqref{eq.3.5}
endowes the manifold $T(G/K_0)$ with the complex structure $J_c^{K_0}$
and a $G$-invariant nowhere-vanishing $J_c^{K_0}$-holomorphic $n$-form,
which we denote also by $\Theta_h$.
\end{remark}

Fix some orthonormal basis $\xi_1, \dotsc ,\xi_n$  (with respect to the
form $\langle\cdot,\cdot \rangle$) of the space $\fm$.
By definition, the holomorphic vector fields
$X_h^{\xi_1}, \dotsc ,X_h^{\xi_n}$ are linearly independent at the point
$\Pi(e,0)\in T(G/K)$ and therefore they are linearly independent on some
open dense subset of $T(G/K)$ (since the holomorphic vector fields
$p_{h*}((\xi_1)^{\mathbf r}_h),\dotsc,p_{h*}((\xi_n)^{\mathbf r}_h)$
are linearly independent at the point
$p_h(e)=\{K^\mathbb{C}\}\in G^\mathbb{C}/K^\mathbb{C}$).
However, these global $J_c^K$-holomorphic (specifically, $(1,0)$)
vector fields on $T(G/K)$ are not $G$-invariant.

We will construct now certain global
$G$-invariant smooth vector fields on
$G\times\fm$, which in turn determine certain local
$(1,0)$ vector fields on $T(G/K)$,
the latter having the important property that the form
$\Theta_h$ is a nonzero constant on them (see \eqref{eq.3.14}
below), whenever
$K$ is connected. To this end, we consider the special local
section
$\exp(O_\fm)\times{\mathfrak m}\subset G\times{\mathfrak m}$
of the projection $\Pi$ and the corresponding open subset
$O=\Pi(\exp(O_\fm)\times\fm)$ of
$T(G/K)$ (see~(\ref{eq.3.2})).

\begin{lemma}\label{le.3.5}
We retain the notation of {\em Lemma~\ref{le.3.2}}.
The (complex) vector fields
$Y^\xi$, $\xi\in\fg$, on the manifold
$G\times\fm$ defined by
\begin{equation}\label{eq.3.13}
\begin{split}
Y^\xi{(g,w)}
&=\left(
\Bigl(\frac{\cos \ad_w-1}{\ad_w\cos\ad_w}B^\fm_w(\sin \ad_w \xi)+
\frac{1}{\cos\ad_w}\xi\Bigr)^{l}(g), \, B^\fm_w(\sin \ad_w \xi)\right) \\
&\qquad\; - \ri \left(\Bigl(\frac{\cos \ad_w-1}{\ad_w\cos\ad_w}
B^\fm_w(\cos \ad_w \xi)
\Bigr)^{l}(g), \, B^\fm_w(\cos \ad_w \xi)\right),
\end{split}
\end{equation}
are smooth and $G$-invariant. The
vector fields $Y^\xi_{O}$ on the open subset $O\subset T(G/K)$,
$Y^\xi_{O}(\Pi(g,w))=\Pi_{*(g,w)}(Y^\xi(g,w))$,
$(g,w)\in \exp(O_\fm)\times{\mathfrak m}$, are smooth
$(1,0)$-vector fields. If the subgroup $K$ is connected
we have
\begin{equation}\label{eq.3.14}
\Theta_h(Y^{\xi_1}_O, \dotsc,Y^{\xi_n}_O)=\mathrm{const}\ne 0,
\end{equation}
where $\xi_1, \dotsc, \xi_n$ is the given orthonormal basis of $\fm$.

If, in addition, the homogeneous manifold
$G/K$ is a symmetric space
($K$ is not necessarily connected), then for each
$\xi\in\fm$ we have
\begin{equation}\label{eq.3.15}
Y^\xi{(g,w)}=
\left(\Bigl(
\frac{1}{\cos \ad_w}\xi- \ri \, \frac{\cos \ad_w-1}{\sin\ad_w}
\xi\Bigr)^{l}(g),
\,- \ri \, \frac{\ad_w\cos \ad_w}{\sin \ad_w} \xi\right).
\end{equation}
\end{lemma}
\begin{proof}
Suppose that
$K$ is connected and consider the canonical holomorphic
projection $p_h \colon G^\mathbb{C}\to G^\mathbb{C}/K^\mathbb{C}$.
Since the form
$\Theta_h$ on $G^\mathbb{C}/K^\mathbb{C}$
is $G^\mathbb{C}$-invariant and holomorphic, its lift
$p_h^*(\Theta_h)$ is also a
$G^\mathbb{C}$-invariant and holomorphic form on
$G^\mathbb{C}$. It is clear that
$$
p_h^*(\Theta_h)\big((\xi_1)_h^{\mathbf l},\dotsc,
(\xi_n)_h^{\mathbf l}\big)=\mathrm{const}\ne0,
$$
where $(\xi_1)_h^{\mathbf l}, \dotsc, (\xi_n)_h^{\mathbf l}$ are
the left $G^{\mathbb C}$-invariant (global holomorphic) vector fields
on the complex Lie group $G^\mathbb{C}$ corresponding to
the vectors $\xi_1, \dotsc ,\xi_n\in\fm$.

But, as we remarked above, the group
$G^{\mathbb C}$ is diffeomorphic to
$G\times {\mathfrak m}\times {\mathfrak k}$,
$G^{\mathbb C}=G\exp( \ai {\mathfrak m})$ $\exp( \ai {\mathfrak k})$.
Therefore for any $g\in G$, $w\in\fm$, $\xi\in\fg$, we have
\[
\xi_{h}^{\mathbf l}(g\exp ( \ai w))=Y^\xi(g\exp ( \ai w))+A^\xi
(g\exp ( \ai w))
\in T^\mathbb{C}(G^\mathbb{R}),
\]
where
the first component is tangent to the (real) submanifold
$G\exp( \ai {\mathfrak m})\subset G^\mathbb{R}$
at $g\exp ( \ai w)$ and the second one is tangent to the submanifold
$g\exp ( \ai w)\exp( \ai {\mathfrak k})\subset G^\mathbb{R}$
at $g\exp ( \ai w)$. Since the image
$p_h(g\exp ( \ai w)$ $\exp( \ai {\mathfrak k}))$ is a one-point subset
$p_h(g\exp ( \ai w))$ of $G^\mathbb{C}/K^\mathbb{C}$,
$p_{h*}(A^\xi(g\exp  \ai w))$ $=0$ and, consequently,
$$
p_h^*(\Theta_h)(Y^{\xi_1},\dotsc,Y^{\xi_n})|_{G\exp  \ai \fm}
=p_h^*(\Theta_h)((\xi_1)_h^{\mathbf l},
\dotsc,(\xi_n)_h^{\mathbf l})|_{G\exp  \ai \fm} =
\mathrm{const}\ne0.
$$
Taking into account that the restriction
$(f_K\circ p_h)|_{\exp(O_\fm)\exp( \ai \fm)}$
of the map (submersion)
$f_K\circ p_h$ (see Definitions~(\ref{eq.3.2}),
(\ref{eq.3.4}) and~(\ref{eq.3.5})),
$$
G^{\mathbb C} \stackrel{p_h}{\to}
G^\mathbb{C}/K^\mathbb{C}\stackrel{f_K}{\to} T(G/K),
\
g\exp( \ai w)\exp( \ai \zeta)\stackrel{p_h}{\mapsto}
g\exp( \ai w)\exp( \ai \zeta)K^\mathbb{C}\stackrel{f_K}{\mapsto}
\Pi(g,w)
$$
is a diffeomorphism onto the open subset $O\subset T(G/K)$,
we obtain relation~(\ref{eq.3.14}) --- note that --- for the vector fields
$(f_K\circ p_h)_*\left(Y^{\xi}|_{\exp(O_\fm)\exp( \ai \fm)}\right)$.
So that now to complete the proof of the lemma
it is sufficient to show that the component $Y^\xi(g\exp ( \ai w))$ of the
left $G^{\mathbb C}$-invariant vector field $\xi_h^{l}$
coincides with the vector field $Y^\xi(g, w)$ in (\ref{eq.3.13})
under the natural identification
of the submanifold $G\exp ( \ai \fm)\subset G^\mathbb{R}$
with the manifold $G\times\fm$.
By such identification, the vector field $Y^\xi$
on $G\times\fm$ is $G$-invariant and smooth.

Since the left $G^{\mathbb C}$-invariant vector field $\xi_h^{\mathbf l}$,
for $\xi \in {\mathfrak g}$, is given by $\xi_h^{\mathbf l}=\xi^{l}-
\ri (I \xi)^{l}$, $\xi\in{\mathfrak g}$,
we have to calculate the components of the
$G^{\mathbb R}$-invariant vector fields
$\xi^{l}$ and $(I \xi)^{l}$, tangent to the submanifold
$G\exp( \ai \fm)$.

To this end it is sufficient to consider the two
curves $g\exp( \ai w)\exp(t\xi)$ and $g\exp( \ai w)\exp(t\ai \xi)$,
$t\in\mathbb{R}$, in the group $G^\mathbb{C}=G^\mathbb{R}$
through the point
$g\exp( \ai w)$ with tangent vectors
$\xi^{l}(g\exp( \ai w))$,
$(I \xi)^{l}(g\exp( \ai w))$,
respectively. By (\ref{eq.3.3}),
$$
g\exp( \ai w) \exp(t\varepsilon \xi)=gg_\varepsilon(t)
\exp( \ai v_\varepsilon(t)) \exp( \ai k_\varepsilon(t)),
$$
where $\varepsilon\in\{1, \ai \}$,
$g_\varepsilon(0)=e$, $v_\varepsilon(0)=w$, $k_\varepsilon(0)=0$
and $g_\varepsilon(t)$,
$v_\varepsilon(t)$ and $k_\varepsilon(t)$  are smooth curves in
$G$, $\fm$  and $\fk$, respectively. Hence,
$Y^\xi(g,w) \in T^\bbC_{(g,w)}(G \times {\mathfrak m})
\cong T^\bbC_gG \times {\mathfrak m}^\bbC$ is given by
\begin{equation}\label{eq.3.16}
Y^\xi(g,w) = \big( (g^\prime_1(0) - \ri g^\prime_\ai(0))^{l}(g),\,
v^\prime_1(0) - \ri v^\prime_\ai(0)\big).
\end{equation}
From the equation
\begin{equation*}
\exp(t\varepsilon \xi)
=\bigl(\exp(- \ai w)g_\varepsilon(t)\exp ( \ai w)\bigr)
\bigl(\exp(- \ai w)\exp ( \ai v_\varepsilon(t))\bigr)
\exp( \ai k_\varepsilon(t))
\end{equation*}
in $G^\mathbb{C}$ we obtain the following equation in
$\fg^\mathbb{C}=\fg\oplus\ai\fg$
for the tangent vectors
$g^\prime_\varepsilon(0)\in{\mathfrak g}, v^\prime_\varepsilon(0)
\in {\mathfrak m}, k^\prime_\varepsilon(0)\in {\mathfrak k}$:
\begin{equation}\label{eq.3.17}
\varepsilon\xi=\mathrm{e}^{- \ai \operatorname{ad} w}g^{\prime}_\varepsilon(0)+
\frac{1-\mathrm{e}^{- \ai \operatorname{ad} w}}
{ \ai \operatorname{ad} w} \ai v^\prime_\varepsilon(0)+
\ai k^\prime_\varepsilon(0).
\end{equation}
Since the map~(\ref{eq.3.3}) is a diffeomorphism, there exists
a unique solution $(g'(0)$, $v'(0),k'(0)) = (g^\prime_\varepsilon(0)$,
$v^{\prime}_\varepsilon(0), k^{\prime}_\varepsilon(0))$
$\in\fg\times\fm\times\fk$
of Equation~(\ref{eq.3.17}) in $\fg^\mathbb{C}$.

If $\varepsilon=1$ we obtain one equation for $\fg^\mathbb{C}$:
$$
\xi={\mathrm{e}}^{- \ai \ad w}g'(0)+
\frac{1-{\mathrm{e}}^{{- \ai \operatorname{ad} w}}}
{\operatorname{ad} w}v'(0)+ \ai k'(0),
$$
or two equations for $\fg$:
\[
\xi=\cos \ad_w g'(0)+\frac{1-\cos \ad_w}{\ad_w}v'(0),
\
0=-\sin \ad_w g'(0)+\frac{\sin \ad_w}{\ad_w}v'(0)+k'(0),
\]
which are equivalent to the pair of equations
$$
\sin \ad_w \xi=\frac{\sin \ad_w}{\ad_w}v'(0)+\cos \ad_w k'(0),
\
g'(0)=\frac{\cos \ad_w-1}{\ad_w\cos\ad_w}v'(0)+
\frac{1}{\cos\ad_w}\xi.
$$
Thus,
\begin{equation*}
\begin{split}
&
v^\prime_1(0)\!=\!B^\fm_w(\sin \ad_w\! \xi),\;\,
k^\prime_1(0)\!=\!B^\fk_w(\sin \ad_w\! \xi), \;\,
\\ \noalign{\smallskip}
&\,
g^\prime_1(0)\!=\!\frac{\cos \ad_w-1}{\ad_w\!\cos\ad_w}
B^\fm_w(\sin \ad_w\! \xi)\!+\!
\frac{1}{\cos\ad_w}\xi,
\end{split}
\end{equation*}
where the operator-functions
$B^\fm \colon \fm\to \End (\fg,\fm)$ and
$B^\fk \colon \fm\to \End  (\fg,\fk)$,
$w\mapsto B^\fk_w$, are determined in Lemma~\ref{le.3.2}.
Similarly,
\begin{equation*}
v^\prime_\ai(0)=B^\fm_w(\cos \ad_w \xi),\
k^\prime_\ai(0)=B^\fk_w(\cos \ad_w \xi),\
g^\prime_\ai(0)=\frac{\cos \ad_w-1}{\ad_w\cos\ad_w}B^\fm_w(\cos \ad_w \xi).
\end{equation*}
Now, substituting the two triples
$(g^\prime_\varepsilon(0), v^\prime_\varepsilon(0),
k^\prime_\varepsilon(0))$,
for $\varepsilon\in\{1, \ai \}$, in \eqref{eq.3.16}
we obtain expressions \eqref{eq.3.13}. Note here that
because the mappings $w\mapsto \ad_w$, $w\mapsto B^\fm_w$, and
$w\mapsto B^\fk_w$ ($w\in\fm$), are
$\Ad(K)$-equivariant (see Remark~\ref{re.3.3}
and~(\ref{eq.3.12})), for any
$w\in\fm$, $\xi\in\fg$,
$k\in K$, it follows that
\begin{equation}\label{eq.3.18}
\text{if}\
Y^\xi(e,w)=(\eta,v)\in\fg\times\fm
\ \text{then}\
Y^{\Ad_k \xi}(e,\Ad_k w)=(\Ad_k \eta,\Ad_k v).
\end{equation}

To prove the last part of the lemma it is
sufficient to note that
in the symmetric case the subspaces $\fm$  and $\fk$
are invariant subspaces of the operators $(\ad_w)^{2p}$, $p=0,1,2,...$
and use~(\ref{eq.3.8}).
\end{proof}

Given a $G$-invariant  K\"ahler structure $(\bfg, \omega, J^K_c)$
on $T(G/K)$, where $J^K_c$ is its canonical complex structure, it follows
from~(\ref{eq.2.2}) that the Ricci form $\Ric(\bfg)$ is given by
$
\Ric(\bfg)=- \ri\,  \partial\bar\partial\ln
\det\Bigl(\omega\bigl(X_h^{\xi_j},
\overline{X_h^{\xi_{s}}}\bigr)\Bigr).
$
Since the vector fields
$X_h^{\xi_j}$  in~(\ref{eq.3.7}) are not
$G$-invariant, the calculation of the function
$\det(\omega(X_h^{\xi_j}, \overline{X_h^{\xi_{s}}}))$
is not simple. To substantially simplify this calculation we
will prove that this function is equal, up to a non-zero complex factor, to
$\cS\cdot \text{\sl h}$, where $\cS$ is some global
$G$-invariant function and $\text{\sl h}$ is locally expressed
as $\text{\sl h} =h\cdot\overline{h}$ for
some local $J^K_c$-holomorphic function $h$ on $T(G/K)$.
Specifically, we obtain the following result.

\begin{proposition}\label{pr.3.6}
Let $G$ be a compact Lie group. Let $\bfg$ be a
$G$-invariant K\"ahler metric on $T(G/K)$ associated with
the canonical complex structure $J^K_c$  and let $\omega$ be
its fundamental form.
Then the function $\widetilde \cS\colon G
\times \fm \to {\mathbb C}$ given by
\begin{equation}\label{eq.3.19}
\widetilde \cS(g,w) = \det \Bigl((\Pi^\ast\omega)_{(g,w)}\big( Y^{\xi_j},
\overline{Y^{\xi_{s}}} \big)\Bigr),
\end{equation}
is left $G$-invariant and right $K$-invariant and therefore
determines a unique $G$-invariant function
$\cS \colon T(G/K)\to\mathbb{C}$ such that $\widetilde \cS=\Pi^* \cS$.
We have
\begin{equation}
\label{eq.3.20}
\Ric(\bfg)=- \ri\,  \partial\bar\partial\ln \cS.
\end{equation}
\end{proposition}
\begin{proof}
The function $\widetilde \cS$ is
$G$-invariant because so are the vector fields
$\{Y^{\xi_j}\}$. Let us show that moreover
$\widetilde \cS$ is right $K$-invariant or, equivalently,
\begin{equation}\label{eq.3.21}
\det \Bigl((\Pi^\ast\omega)_{(e,w)}\big( Y^{\xi_j},
\overline{Y^{\xi_{s}}} \big)\Bigr) = \det \Bigl((\Pi^\ast\omega)_{(e,
\,\Ad_kw)}\big( Y^{\xi_j}, \overline{Y^{\xi_{s}}} \big)\Bigr),
\end{equation}
for all $k \in K$.
Indeed, since the form $\omega$ is $G$-invariant and the
projection $\Pi \colon G \times \fm \to T(G/K)$
is an equivariant submersion with respect to the natural left
$G$-actions on $G \times \fm$ and
$T(G/K)$, it follows that $\Pi^*\omega$ is a left
$G$-invariant and right $K$-invariant form on
$G \times \fm$. Therefore, for any $g\in G$,
$\xi_1, \xi_2 \in \fg$, $u_1,u_2 \in \fm$,
$k \in K$, we have
\begin{equation}\label{eq.3.22}
\begin{split}
(\Pi^*\omega)_{(e,w)}&\bigl((\xi_1,u_1),(\xi_2,u_2)\bigr)
=(\Pi^*\omega)_{(g,w)}\bigl((\xi_1^l(g),u_1),(\xi_2^l(g),u_2)\bigr)
\\ \noalign{\smallskip}
=&(\Pi^*\omega)_{(e,\,\Ad_k w)}\bigl(
(\Ad_k \xi_1,\Ad_k u_1),(\Ad_k \xi_2,\Ad_k u_2)\bigr)
\end{split}
\end{equation}
because
$(\exp t\xi\cdot k^{-1},\, \Ad_k(w+tv))
=(k^{-1}\cdot\exp t(\Ad_k\xi),\, \Ad_k w+t\Ad_kv))$.

Putting $Y^{\xi_j}(e,w)=(\eta_{j},v_j)\in
T^\mathbb{C}_eG\times T^\mathbb{C}_w\fm
=\fg^\mathbb{C}\times\fm^\mathbb{C}$,
we obtain that
\begin{equation}\label{eq.3.23}
\begin{split}
(\Pi^*\omega)_{(e,w)}&\bigl(Y^{\xi_j}, \overline{Y^{\xi_{s}}}\bigr)
\eqdef(\Pi^*\omega)_{(e,w)}((\eta_{j},v_j),\overline{(\eta_s,v_s)})\\
&\stackrel{\mathrm{~(\ref{eq.3.22})}}{=}
(\Pi^*\omega)_{(e,\Ad_k w)}((\Ad_k\eta_{j},\Ad_k v_j),
\overline{(\Ad_k\eta_s,\Ad_k v_s)})\\
&\stackrel{\mathrm{~(\ref{eq.3.18})}}{=}
(\Pi^*\omega)_{(e,\Ad_k w)}\bigl(Y^{\Ad_k \xi_j},
\overline{Y^{\Ad_k \xi_s}}\bigr).
\end{split}
\end{equation}

Since by \eqref{eq.3.13} the map
$\fm\to \fg^\mathbb{C}\times\fm^\mathbb{C}$,
$\xi\mapsto Y^\xi(e,w)$, is linear and the endomorphism
$\Ad_k \colon \fm\to\fm$ is orthogonal ($\det\Ad_k=\pm 1$), we obtain
\begin{equation*}
\begin{split}
\det & \Bigl((\Pi^*\omega)_{(e,w)}\bigl(Y^{\xi_j},
\overline{Y^{\xi_{s}}}\bigr)\Bigr) \!\!
\stackrel{\mathrm{~(\ref{eq.3.23})}}{=}
\det\Bigl((\Pi^*\omega)_{(e,\,\Ad_k w)}\bigl(Y^{\Ad_k \xi_j},
\overline{Y^{\Ad_k \xi_s}}\bigr)\Bigr) \\
&=\det(\Ad_k)^t\cdot \! \det(\Ad_k) \!\cdot \!
\det\Bigl(\!(\Pi^*\omega)_{(e,\,\Ad_k w)}\bigl(Y^{\xi_j},
\overline{Y^{\xi_{s}}}\bigr)\!\Bigr) \\
&=\det\Bigl(\!(\Pi^*\omega)_{(e,\,\Ad_k w)}\bigl(Y^{\xi_j},
\overline{Y^{\xi_{s}}}\bigr)\!\Bigr).
\end{split}
\end{equation*}
This proves \eqref{eq.3.21}. Now it is easy to see that the
function $\cS$ with $\widetilde \cS=\Pi^* \cS$ is well defined.
This function is smooth and $G$-invariant because the
mapping $\Pi$ is a $G$-equivariant submersion.

In order to prove \eqref{eq.3.20}, suppose first that the subgroup
$K\subset G$ is connected.
Consider the holomorphic form $\Theta_h$ on $T(G/K)$.
Then  there exists a constant
$\varepsilon_n\in{\mathbb C}$ such that
$\varepsilon_n\Theta_h\land\overline\Theta_h$ is a volume form
compatible with the orientation defined by the symplectic
structure $\omega$ on $T(G/K)$. Thus there is a positive, real
analytic function $\cS_1$ on $T(G/K)$ such that
\begin{equation}\label{eq.3.24}
\omega^n=\cS_1\cdot\varepsilon_n\Theta_h\land\overline{\Theta_h}.
\end{equation}
The function $\cS_1$ is $G$-invariant because
so are the forms $\omega$ and $\Theta_h$.

On the other hand, putting
$(e_1,\dotsc ,e_{2n}) = \bigl(X_h^{\xi_1},\dotsc ,X_h^{\xi_n},
\overline{X_h^{\xi_1}},\dotsc,\overline{X_h^{\xi_n}}\bigr)
$ and using the fact that $\omega$, considered as a
complex form, is of degree $(1,1)$, we
obtain, in particular,
$
\omega\bigl(X_h^{\xi_j}, X_h^{\xi_{s}}\bigr)=0,
\
\omega\bigl(\overline{X_h^{\xi_j}}, \overline{X_h^{\xi_{s}}}\bigr)=0.
$
Hence, we can deduce
from equation~\eqref{eq.3.24} that
for some local holomorphic function $h_1$ on $T(G/K)$,
\begin{equation}\label{eq.3.25}
\begin{split}
\cS_1\cdot & \varepsilon_n |h_1|^2
=\omega^n(X_h^{\xi_1},\dots,X_h^{\xi_n},
\overline{X_h^{\xi_1}},\dotsc,\overline{X_h^{\xi_n}}) \\
&\stackrel{\mathrm{def}}{=} 2^{-n}
\sum_{\sigma\in S_{2n}}\varepsilon(\sigma)
\omega(e_{\sigma(1)},e_{\sigma(2)})\cdots
\omega(e_{\sigma(2n-1)},e_{\sigma(2n)})\\
&\stackrel{\mathrm{*}}{=}\sum_{\sigma\in S_{n}}\varepsilon(\sigma)
\omega(X_h^{\xi_1},\overline{X_h^{\xi_{\sigma(1)}}})\cdots
\omega(X_h^{\xi_n},\overline{X_h^{\xi_{\sigma(n)}}})
\stackrel{\mathrm{*}}{=}
\det\Bigl(\omega\bigl(X_h^{\xi_j},
\overline{X_h^{\xi_{s}}}\bigr)\Bigr).
\end{split}
\end{equation}
Here $\stackrel{\mathrm{*}}{=}$
means ``equal up to a non-zero constant complex factor."

Similarly, by Lemma~\ref{le.3.5} and since
$
\omega\bigl(Y_O^{\xi_j}, Y_O^{\xi_{s}}\bigr)=0,
\ \omega\bigl(\overline{Y_O^{\xi_j}},
\overline{Y_O^{\xi_{s}}}\bigr)=0,
$
we have from~\eqref{eq.3.24} and \eqref{eq.3.14} that $\cS_1$
is locally expressed as
\begin{equation*}
\cS_1\stackrel{\mathrm{*}}{=}
\omega^n(Y_O^{\xi_1}, \dotsc,Y_O^{\xi_n},
\overline{Y_O^{\xi_1}},\dotsc,\overline{Y_O^{\xi_n}})
\stackrel{\mathrm{*}}{=}
\det\Bigl(\omega\bigl(Y_O^{\xi_j},
\overline{Y_O^{\xi_{s}}}\bigr)\Bigr).
\end{equation*}
Then, using Lemma~\ref{le.3.5} again we obtain that locally
$$
\Pi^* \cS_1
\stackrel{\mathrm{*}}{=}
\Pi^*\biggl(
\det\Bigl(\omega\bigl(Y_O^{\xi_j},
\overline{Y_O^{\xi_{s}}}\bigr)\Bigr)\biggr)
=
\det\Bigl((\Pi^*\omega)\bigl(Y^{\xi_j},
\overline{Y^{\xi_{s}}}\bigr)\Bigr)
=\widetilde \cS=\Pi^* \cS.
$$
Hence, using that $\cS_1$ is $G$-invariant, $G\cdot O=T(G/K)$,
we obtain that
$\cS(x)\stackrel{\mathrm{*}}{=}\cS_1(x)$
and from~(\ref{eq.3.25}) we have
$$
\det\Bigl(\omega\bigl(X^{\xi_j}_h,
\overline{X^{\xi_{s}}_h}\bigr)\Bigr) = \cS \cdot c\,\text{\sl h},
$$
where $c \in {\mathbb C}\backslash \{0\}$ and
$\text{\sl h}$ is a function on $T(G/K)$ such that locally
$\text{\sl h} =|h_1|^2$. Thus \eqref{eq.3.20} holds.

Finally, suppose that $K\ne K_0$, that is,
$K$ is not connected.

Since below we will consider the manifolds
$T(G/K_0)$ and $T(G/K)$ simultaneously,
we will use the notation introduced above for
objects on $T(G/K)$ but with indexes
$K_0$ and $K$ respectively (if they exist).
To complete the proof of the lemma,
consider the natural covering map
$
\Psi \colon G^\mathbb{C}/K_0^\mathbb{C}\to G^\mathbb{C}/K^\mathbb{C}.
$
This map is holomorphic and $G^\mathbb{C}$-equivariant.
There exists a unique map
$\psi \colon T(G/K_0)\to T(G/K)$ such that
the following diagram is commutative:
\begin{equation*}
\begin{array}{rrl}
G^\mathbb{C}/K_0^\mathbb{C}&\stackrel{\Psi}
\to&G^\mathbb{C}/K^\mathbb{C}\\ \noalign{\smallskip}
{\downarrow\makebox[16pt]{$\scriptstyle{f^{\times}_{K_0}}$}}
&&{\makebox[10pt]{$\scriptstyle{f^{\times}_K}$}\downarrow} \\
\noalign{\smallskip}
G\times_{K_0}\fm&& G\times_{K}\fm\\ \noalign{\smallskip}
{\downarrow\makebox[16pt]{$\scriptstyle{\phi_{K_0}}$}}
&&{\makebox[10pt]{$\scriptstyle{\phi_K}$}\downarrow} \\
\noalign{\smallskip} T(G/K_0)&
\stackrel{\psi}\to&T(G/K)
\end{array}
\begin{array}{rrl}
g\exp( \ai w)\exp( \ai \zeta) K^\mathbb{C}_0&
\stackrel{\Psi}\to&g\exp( \ai w)\exp( \ai \zeta) K^\mathbb{C} \\
\noalign{\smallskip}
{\downarrow\makebox[18pt]{$\scriptstyle{f^{\times}_{K_0}}$}}
&&{\makebox[10pt]{$\scriptstyle{f^{\times}_K}$}\downarrow} \\
\noalign{\smallskip}
[g,w]_{K_0}&&[g,w]_{K} \\ \noalign{\smallskip}
{\downarrow\makebox[18pt]{$\scriptstyle{\phi_{K_0}}$}}
&&{\makebox[10pt]{$\scriptstyle{\phi_K}$}\downarrow} \\
\noalign{\smallskip}
\phi_{K_0}([g,w]_{K_0})&
\stackrel{\psi}\to&\phi_K([g,w]_{K})
\end{array}.
\end{equation*}
By definition, the map $\psi$ is holomorphic and
$G$-equivariant. Here the maps (diffeomorphisms) $\phi_K$, $\phi_{K_0}$,
and $f^{\times}_K$, $f^{\times}_{K_0}$ are defined
by~(\ref{eq.3.1}) and~(\ref{eq.3.4}).

Since the global vector fields
$Y^\xi$, $\xi\in\fm$, in~(\ref{eq.3.13}) on $G\times\fm$
are determined only in terms of the pair of Lie algebras
$(\fg,\fk)$, the function
$\widetilde \cS$ in~(\ref{eq.3.19}) is the same for the spaces
$T(G/K)$ and $T(G/K_0)$, that is,
$\widetilde \cS=\Pi_K^* \cS_{K}= \Pi_{K_0}^* \cS_{K_0}$.
Thus $\cS_{K_0}=\psi^* \cS_K$.

Now, the map
$\psi$ is a local holomorphic diffeomorphism, hence
$\partial\bar\partial\ln \cS_{K_0}=
\psi^*(\partial\bar\partial\ln \cS_{K})$.
Moreover, the form
$\omega_{K_0}=\psi^*\omega_K$ ($\omega_K=\omega$)
is the fundamental form of the K\"ahler metric
$\mathbf{g}_{K_0}$ on
$T(G/K_0)$ associated with the canonical complex structure
$J^{K_0}_c$. Since
$\Ric(\mathbf{g}_{K_0})=\psi^*\bigl(\Ric(\mathbf{g}_{K})\bigr)$
and as we proved above,
$\Ric(\mathbf{g}_{K_0})=- \ri\, \partial\bar\partial\ln \cS_{K_0}$,
we obtain~(\ref{eq.3.20}).
\end{proof}

\section{Invariant Ricci-flat K\"ahler metrics on tangent
bundles of compact Riemannian symmetric spaces}
\label{s.4}

We continue with the previous notations but in this section
and the following one it is assumed in addition that
$G/K$ is a rank-$r$ Riemannian symmetric space of a
connected, compact (possibly with nontrivial center) Lie group
$G$.

\subsection{Root theory of Riemannian symmetric spaces
and reduced symmetric spaces of maximal rank}
\label{ss.4.1}

Here we will review a few facts about Riemannian symmetric
spaces~\cite[Ch.\ VII, \S2, \S11]{He}.
We have
\begin{equation}\label{eq.4.1}
\fg=\fm\oplus\fk,
\quad\text{where }\quad
[{\mathfrak m},{\mathfrak m}]\subset{\mathfrak k},
\quad [\fk,\fm]\subset \fm,
\quad [\fk,\fk]\subset \fk,
\quad \text{and} \quad \fk\bot \fm.
\end{equation}
In other words, there exists an involutive automorphism
$\sigma \colon \fg\to\fg$ such that
\begin{equation*}
\fk=(1+\sigma)\fg
\quad \text{and} \quad
\fm=(1-\sigma)\fg.
\end{equation*}
Moreover the scalar product $\langle\cdot ,\cdot \rangle$
is $\sigma$-invariant.

Let $\fa\subset\fm$ be some Cartan subspace of
the space ${\mathfrak m}$. There exists a $\sigma$-invariant Cartan
subalgebra $\ft$ of $\fg$ containing the commutative subspace $\fa$,
i.e.
\begin{equation*}
\ft=\fa\oplus\ft_0,
\;\; \text{where}\;\;
\fa=(1-\sigma)\ft, \;\;
\ft_0=(1+\sigma)\ft.
\end{equation*}
Then the complexification $\ft^\bbC$ is a Cartan subalgebra of the
reductive complex Lie algebra $\fg^\bbC$ and we have
the root space decomposition
\begin{equation*}
{\mathfrak g}^{\mathbb C}=\ft^\bbC\oplus\sum_{\alpha\in\Delta}
\tilde{\mathfrak g}_\alpha.
\end{equation*}
Here $\Delta$ is the root system of $\fg^\bbC$ with respect to
the Cartan subalgebra $\ft^\bbC$. For each $\alpha\in\Delta$ we have
\begin{equation*}
\tilde{\mathfrak g}_\alpha=\big\{\tilde\xi\in{\mathfrak g}^{\mathbb C}:
\ad_{\kern1pt \tilde t}\tilde\xi=\alpha(\tilde t)\tilde \xi,\;
\tilde t\in{\mathfrak t}^{\mathbb C}\big\}
\quad\text{and}\;\;
\dim_{\kern1pt \bbC} \tilde{\mathfrak g}_\alpha=1.
\end{equation*}
It is evident that the
centralizer $\tilde{\mathfrak g}_0$ of the space $\fa^\bbC$
in $\fg^\bbC$ is the subalgebra
\begin{equation}\label{eq.4.2}
\tilde{\mathfrak g}_0=\ft^\bbC\oplus\sum_{\alpha\in\Delta_0}
\tilde{\mathfrak g}_\alpha,
\end{equation}
where $\Delta_0=\{\alpha\in\Delta: \alpha|_{\fa^\bbC}=0\}$
is the root system of the reductive Lie algebra
$\tilde{\mathfrak g}_0$ with respect to its Cartan
subalgebra $\ft^\bbC$.

Since the algebra
$\tilde{\mathfrak g}_0$ coincides with the centralizer of
some (regular) element $x_\Pi\in\fa$ in
$\fg^\bbC$, there exists a basis
$\Pi$ of $\Delta$  (a system of simple roots) such that
$\Pi_0=\Pi\cap\Delta_0$ is a basis of
$\Delta_0$. Indeed, the element
$-\ai x_\Pi\in\ai\ft$ belongs to the closure of some Weyl
chamber in $\ai\ft$ determining the basis $\Pi$. Then
$\Pi_0=\{\alpha\in\Pi: \alpha(-\ai x_\Pi)=0\}$.  The bases
$\Pi$ and $\Pi_0$ determine uniquely the subsets
$\Delta^+$ and $\Delta^+_0$ of positive roots of
$\Delta$ and
$\Delta_0$, respectively. It is evident that
\begin{equation*}
\Delta^+\setminus \Delta^+_0=\{\alpha\in \Delta: \alpha(-\ai x_\Pi)>0\}.
\end{equation*}

The set
$\Sigma=\{\lambda\in (\fa^\bbC)^*:
\lambda=\alpha|_{\fa^\bbC},\ \alpha\in \Delta\setminus\Delta_0\}$
is the set of restricted roots of the triple $(\fg,\fk, \fa)$, which is
independent of the choice of the $\sigma$-invariant Cartan subalgebra
$\ft$ containing the Cartan subspace $\fa$. The following decomposition
\begin{equation}\label{eq.4.3}
{\mathfrak g}^{\mathbb C}=\tilde{\mathfrak g}_0\oplus
\sum_{\lambda\in\Sigma^+}
(\tilde{\mathfrak g}_\lambda\oplus\tilde{\mathfrak g}_{-\lambda}),
\quad \text{where}\quad
\tilde{\mathfrak g}_\lambda
= \sum_{\alpha\in\Delta\setminus\Delta_0,\, \alpha|_{\fa^\bbC}=\lambda}
\tilde{\mathfrak g}_\alpha,
\end{equation}
and $\Sigma^+$ denotes the subset
of positive restricted roots in $\Sigma$ determined by the
set of positive roots $\Delta^+$,
gives us the simultaneous diagonalization of
$\operatorname{ad}({\mathfrak a}^{\mathbb C})$ on
${\mathfrak g}^{\mathbb C}$. Denote by
$m_\lambda$ the multiplicity of the restricted root
$\lambda\in\Sigma^+$, that is,
$m_\lambda=\card\{\alpha\in\Delta: \alpha|_{\fa^\bbC}=\lambda\}$.

The set $\Sigma$ is an abstract (not necessarily reduced) root
system and its subset
$\Pi_\Sigma=\{\lambda\in (\fa^\bbC)^*:
\lambda=\alpha|_{\fa^\bbC},\ \alpha\in \Pi\setminus\Pi_0\}$
is a basis of
$\Sigma$ containing $\dim\fa$ elements~\cite[Ch.\ VII, Theorem
2.19]{He}.

\begin{lemma}\label{le.4.1}
Let $\Delta'=\{\alpha\in\Delta: \alpha(\ft_0)=0\}$. Then
$\Delta'\subset \Delta$ is a root subsystem of the root
system $\Delta$ and $\Delta'\subset \Delta\setminus\Delta_0$.
If $\alpha\in\Delta'$, $\widehat\alpha\in\Delta$,
with $\alpha\ne\widehat\alpha$ and
$\alpha|_{\fa}=\widehat\alpha|_{\fa}$, then
$\alpha-\widehat\alpha\in\Delta_0$.
In particular, for any root $\alpha\in\Delta'$
the following conditions are equivalent:
\begin{mlist}
\item [$(1)$]
$\alpha+\beta\not\in\Delta$  for\ all $\beta\in\Delta_0;$
\item [$(2)$]
the restricted root $\lambda=\alpha|_{\fa^\bbC}$ has multiplicity
$1$ (as an element of the restricted root system $\Sigma$$)$.
\end{mlist}
\end{lemma}
\begin{proof}
The set $\Delta'$ is an (abstract) root subsystem of
$\Delta$ because the subset $\Delta'\subset\Delta$ is symmetric
($\Delta'=-\Delta'$) and closed (if
$\alpha_1,\alpha_2\in\Delta'$,
$\alpha_1+\alpha_2\in\Delta$ then
$\alpha_1+\alpha_2\in\Delta'$). We have
$\Delta'\cap\Delta_0=\emptyset$ because
$\fa\oplus\ft_0=\ft$.

We now look at the standard scalar product on the real subspace
$V\subset(\ft^\bbC)^*$ spanned by the set $\Delta\subset(\ft^\bbC)^*$.
We can suppose that the Lie algebra $\fg$ is semisimple. Consider on
$\fg^\bbC$ the Killing form
$\langle\cdot ,\cdot\rangle_K$ (which up to multiplication
by a non-zero scalar coincides with our form
$\langle\cdot ,\cdot\rangle$ on each (real) simple ideal of
$\fg$).
For each $\bbC$-linear form $\mu$ on the Cartan
subalgebra $\ft^\bbC$ let $A_\mu\in\ft^\bbC$ be
determined by $\mu(A)=\langle A_\mu, A \rangle_K$
for all $A\in\ft^\bbC$ and put
$\langle\mu_1 ,\mu_2\rangle_K\eqdef \langle A_{\mu_1},
A_{\mu_2}\rangle_K$ for any two elements $\mu_1,\mu_2\in
(\ft^\bbC)^*$. It is well known that for each $\mu\in\Delta$
the vector $A_\mu\in \ai\ft$ and that the restriction
of the Killing form $\langle\cdot ,\cdot\rangle_K$ to $\ai\ft$
is positive-definite.

Since $\fa\bot\ft_0$,
$\alpha|_{\ft_0}=0$ and
$\alpha|_{\fa}=\widehat\alpha|_{\fa}$,
we obtain that $A_\alpha\in\ai\fa$,
$A_\alpha-A_{\widehat\alpha}\in\ai\ft_0$ and,
consequently, $\langle \alpha,\alpha-\widehat\alpha\rangle_K=0$.
Thus
$$
\langle \alpha,\widehat\alpha\rangle_K
=\langle \alpha,\alpha-(\alpha-\widehat\alpha)\rangle_K
=\langle \alpha,\alpha\rangle_K>0.
$$
By the well-known property of root systems
(see for example~\cite[Ch.\ 4,\S1, Theorem 1]{Bo})
if $\langle\alpha,\widehat\alpha\rangle_K>0$ then
$\alpha-\widehat\alpha\in\Delta$ unless $\alpha=\widehat\alpha$.
Since $(\alpha-\widehat\alpha)(\fa)=0$,
$\alpha-\widehat\alpha\in\Delta_0$
by definition of the root subsystem $\Delta_0$.
Hence $m_\lambda>1$ if and only if
there exists $\beta \in \Delta_0$ such that $\alpha + \beta \in \Delta$
(because $(\alpha + \beta)|_\fa = \alpha|_\fa$).
\end{proof}

For each linear form $\lambda$ on ${\mathfrak a}^{\mathbb C}$ put
\begin{equation}\label{eq.4.4}
\begin{split}
{\mathfrak m}_\lambda&\eqdef
\big\{\eta\in{\mathfrak m}:\operatorname{ad}^2_w(\eta)
=\lambda^2(w)\eta,\ \forall w\in{\mathfrak a}\big\}, \\
{\mathfrak k}_\lambda&\eqdef
\big\{\zeta\in{\mathfrak k}: \operatorname{ad}^2_w(\zeta)
=\lambda^2(w)\zeta,\
\forall w\in{\mathfrak a}\big\}.
\end{split}
\end{equation}

Then ${\mathfrak m}_{\lambda}={\mathfrak m}_{-\lambda}$,
${\mathfrak k}_{\lambda}={\mathfrak k}_{-\lambda}$,
${\mathfrak m}_0={\mathfrak a}$ and ${\mathfrak k}_0$ equals
${\mathfrak h}$, the centralizer of ${\mathfrak a}$ in
${\mathfrak k}$.

It is clear that $\fm^\bbC_\lambda\oplus\fk^\bbC_\lambda=
\tilde{\mathfrak g}_\lambda\oplus\tilde{\mathfrak g}_{-\lambda}$
for $\lambda\in\Sigma^+$ and
$\tilde{\mathfrak g}_0=\fm_0^\bbC\oplus\fk_0^\bbC=
\fa^\bbC\oplus\fh^\bbC$
(the Cartan subspace $\fa^\bbC$ is a maximal commutative subspace
of $\fm^\bbC$).
Note also here that by~(\ref{eq.4.2}) the subspace
\begin{equation*}
\ft_0=(1+\sigma)\ft
\quad\text{is a Cartan subalgebra of the centralizer}\ \fh
\ \text{and}\ \ft=\fa\oplus\ft_0.
\end{equation*}

By~\cite[Ch.\ VII, Lemma 11.3]{He}, the following
decompositions are direct and orthogonal:
\begin{equation}\label{eq.4.5}
{\mathfrak m}={\mathfrak a}\oplus\sum_{\lambda\in\Sigma^+}
{\mathfrak m}_\lambda,
\qquad
{\mathfrak k}={\mathfrak h}\oplus
\sum_{\lambda\in\Sigma^+}{\mathfrak k}_\lambda.
\end{equation}
We shall put
\[
{\mathfrak m}^+ \overset{\mathrm{def}}{=} \sum_{\lambda\in\Sigma^+}
{\mathfrak m}_\lambda,
\qquad
{\mathfrak k}^+ \overset{\mathrm{def}}{=}
\sum_{\lambda\in\Sigma^+}{\mathfrak k}_\lambda.
\]

Since the Lie algebra $\fg$ is compact then
$\lambda({\mathfrak a})\subset \ai {\mathbb R}$.
Define the linear function $\lambda' \colon \fa\to\bbR$,
$\lambda\in\Sigma^+$, by the relation $\ai\lambda'=\lambda$.
We need the following lemma, which is a
generalization of~\cite[Ch.\ VII, Lemma 2.3]{He}.

\begin{lemma}\label{le.4.2}
For any vector
$\xi_\lambda\in{\mathfrak m}_\lambda$,
$\lambda\in\Sigma^+$, there exists a unique vector
$\zeta_\lambda\in{\mathfrak k}_\lambda$ such that
\begin{equation}\label{eq.4.6}
[w,\xi_\lambda]=  -\lambda'(w)\zeta_\lambda,
\quad [w,\zeta_\lambda]= \lambda'(w)\xi_\lambda
\qquad\text{for all}\ w\in{\mathfrak a}.
\end{equation}
In particular, $\dim{\mathfrak m}_\lambda
=\dim{\mathfrak k}_\lambda=m_\lambda$ and there exists a unique
endomorphism $T \colon \fm^+ \oplus \fk^+ \to \fm^+ \oplus \fk^+$
such that
\begin{equation}\label{eq.4.7}
\ad_w|_{\fm_\lambda\oplus\fk_\lambda}
=\lambda'(w)T |_{\fm_\lambda\oplus\fk_\lambda},
\;\; T ({\fm_\lambda})={\fk_\lambda},
\;\; T ({\fk_\lambda})={\fm_\lambda},
\ \forall \lambda\in\Sigma^+.
\end{equation}
This endomorphism is orthogonal and $T^2= -\Id_{\fm^+ \oplus \fk^+}$.
\end{lemma}
\begin{proof}
 As we remarked above,
 \begin{equation}\label{eq.4.8}
 ({\mathfrak m}_\lambda\oplus{\mathfrak k}_\lambda)^{\mathbb
 C}=  \tilde{\mathfrak g}_\lambda\oplus\tilde{\mathfrak
g}_{-\lambda}.
 \end{equation}
 It is well known that
the $\bbC$-linear extension of $\sigma$
 is an involutive automorphism of the complex Lie algebra
 $\fg^\bbC$. We denote this involution also by
 $\sigma$. Then
 \begin{equation}\label{eq.4.9}
 \sigma(\tilde{\mathfrak g}_{\lambda})=\tilde{\mathfrak g}_{-\lambda}
 \quad\text{and}\quad
 \sigma(\tilde{\mathfrak g}_{-\lambda})=\tilde{\mathfrak g}_{\lambda}.
 \end{equation}
 Indeed, for any $w\in\fa$ and
 $E\in \tilde{\mathfrak g}_{\lambda}$ we have
 $[w,E]=\lambda(w)E$ and, consequently,
 $[-w,\sigma(E)]=\lambda(w)\sigma(E)$ because
 $\sigma(w)=-w$. Thus
 $\sigma(\tilde{\mathfrak g}_{\lambda})\subset
 \tilde{\mathfrak g}_{-\lambda}$. Similarly, we can show
 that $\sigma(\tilde{\mathfrak g}_{-\lambda})\subset
 \tilde{\mathfrak g}_{\lambda}$. Taking into account that $\sigma$
 is nondegenerate, we obtain~(\ref{eq.4.9}).

 Since $\tilde{\mathfrak g}_{\lambda}\cap\tilde{\mathfrak g}_{-\lambda}=0$
 and  $\sigma(\xi)=-\xi$ for $\xi\in\fm_\lambda^\bbC$,
 $\sigma(\zeta)=\zeta$ for $\zeta\in\fk_\lambda^\bbC$,
 from~(\ref{eq.4.9}) it follows that
 \begin{equation}\label{eq.4.10}
 \tilde{\mathfrak g}_{\lambda}\cap \fm_\lambda^\bbC=0,
 \quad
 \tilde{\mathfrak g}_{\lambda}\cap \fk_\lambda^\bbC=0,
 \quad
 \tilde{\mathfrak g}_{-\lambda}\cap \fm_\lambda^\bbC=0,
 \quad
 \tilde{\mathfrak g}_{-\lambda}\cap \fk_\lambda^\bbC=0.
 \end{equation}
 From~(\ref{eq.4.8}),~(\ref{eq.4.10}) and dimensional
 arguments it follows that for the subspace
 $\tilde{\mathfrak g}_{\lambda}\subset\fm_\lambda^\bbC\oplus
 \fk_\lambda^\bbC$ the natural projections
 $\tilde{\mathfrak g}_{\lambda}\to \fm_\lambda^\bbC$ and
 $\tilde{\mathfrak g}_{\lambda}\to \fk_\lambda^\bbC$
 are isomorphisms. Therefore for any
 $\xi_\lambda\in \fm_\lambda\subset\fm_\lambda^\bbC$ there exists a
 unique vector
 $E\in \tilde{\mathfrak g}_{\lambda}$ such that
 $E=\xi_\lambda+\zeta$, where $\zeta\in \fk_\lambda^\bbC$.
 But $[w,E]=\lambda(w)E$ for each $w\in\fa$, that is,
 $[w,\xi_\lambda+\zeta]=\lambda(w)(\xi_\lambda+\zeta)$.
 By the relations~(\ref{eq.4.1}),
 $[w,\xi_\lambda]=\lambda(w)\zeta$
 and $[w,\zeta]=\lambda(w)\xi_\lambda$.
 Taking into account that $\lambda(w)\in\ai\bbR$
 and $[w,\xi_\lambda]\in\fk$, putting
 $\zeta_\lambda=i\zeta$ we obtain~(\ref{eq.4.6}).

Relations~(\ref{eq.4.6}) and~(\ref{eq.4.7}) determine
a unique endomorphism $T$ for which
$T(\xi_\lambda)=-\zeta_\lambda$,
$T(\zeta_\lambda)=\xi_\lambda$, $\forall \lambda\in\Sigma^+$.
Now the latter assertion of the lemma is evident
excepting orthogonality.

To prove the orthogonality of
$T$ it is sufficient to note that, by~(\ref{eq.4.5}), for
different $\lambda,\mu\in\Sigma^+$, the subspaces
${\mathfrak m}_\lambda\oplus{\mathfrak k}_\lambda$ and
${\mathfrak m}_\mu\oplus{\mathfrak k}_\mu$
are orthogonal and for two arbitrary pairs
$\{\xi_\lambda,\zeta_\lambda\}$ and
$\{\xi^*_\lambda,\zeta^*_\lambda\}$
for which Conditions~(\ref{eq.4.6}) hold, we have by
Definition~(\ref{eq.4.4}) that
\begin{equation*}
\begin{split}
(\lambda'(w))^2\langle\zeta_\lambda, \zeta^*_\lambda\rangle
& =\langle\ad_w\xi_\lambda, \ad_w\xi^*_\lambda\rangle
=\langle -\ad^2_w\xi_\lambda, \xi^*_\lambda\rangle
\\ & =(\lambda'(w))^2\langle\xi_\lambda, \xi^*_\lambda\rangle
=(\lambda'(w))^2\langle T\zeta_\lambda, T\zeta^*_\lambda\rangle.
\end{split}
\end{equation*}
From this and the similar relation for
$\{\xi_\lambda, \xi^*_\lambda\}$, the orthogonality of $T$ follows.
\end{proof}

Each restricted root $\lambda$ defines a hyperplane
$\lambda(w)=0$ in the vector space $\fa$. These
hyperplanes divide the space $\fa$ into finitely many
connected components, called Weyl chambers.
These are open, convex subsets of
$\fa$~(see \cite[Ch.\ VII, \S 2,\S11]{He}).
Fix the Weyl chamber $W^+$ in $\fa$, containing the element $x_\Pi$:
\begin{equation*}
W^+=\bigl\{w\in\fa: \lambda'(w)>0
\text{ for each }\lambda\in\Sigma^+\bigr\}.
\end{equation*}

The subspace $\fm\subset\fg$ is $\Ad(K)$-invariant.
Each $\Ad(K)$-orbit in $\fm$ intersects
the Cartan subspace $\fa$, that is, $\Ad(K)(\fa)=\fm$.
The open connected subset $\fm^R=\Ad(K)(W^+)$ of $\fm$
is called the set of regular points in $\fm$.

Since the centralizer of a (regular) element
$w\in W^+\subset\fa$ in the space $\fm$ coincides with the Cartan
subspace $\fa\subset\fm$, the centralizer of the element
$w$ coincides with the centralizer of the Cartan
subspace $\fa$ in $\Ad(K)$, i.e.
\begin{equation}\label{eq.4.11}
H=\{k\in K: \Ad_k w=w\}=\{k\in K: \Ad_k u=u \text{ for all } u\in\fa\}
\end{equation}
because by~\cite[Ch.\ VII, Lemma\ 2.14]{He}, one has $G_w=G_\fa$,
where
\begin{equation*}
G_w =\{g\in G: \Ad_g w=w\}
\quad\text{and}\quad
G_\fa =\{g\in G: \Ad_g u=u \text{ for all } u\in\fa\}.
\end{equation*}

Note also that the space $\fh$ (see~(\ref{eq.4.5})),
\begin{equation}\label{eq.4.12}
\fh=\{u\in\fk: [u,\fa]=0\}=\{u\in\fk: [u,w]=0\}= \fk_w
\end{equation}
is the Lie algebra of $H$.
In particular, the subalgebra $\fa\oplus\fh= \fg_w$
($(\fa\oplus\fh)^\bbC=\tilde\fg_0$) is a subalgebra of
$\fg$ of maximal rank.

Recall that a subalgebra $\fb\subset\fg$ is said to be {\it regular}
if its normalizer $\fn(\fb)$ in $\fg$ has maximal rank, that is,
$\rank \fn(\fb)=\rank\fg$. In other words, $\fb$ is regular
if and only if $\fb$ is normalized by some Cartan subalgebra
of the algebra $\fg$.

Our interest now centers on what will be shown to be an important
subalgebra of $\fg$. Let $\fg_H\subset\fg$ be the subalgebra
of fixed points of the group $\Ad(H)$, i.e.
\begin{equation}\label{eq.4.13}
\fg_H\eqdef \{u\in\fg: \Ad_h u=u
\text{ for all } h\in H\}.
\end{equation}
It is evident that $\fg_H\subset \fg_\fh$, where
\begin{equation}\label{eq.4.14}
\fg_\fh \eqdef\bigl\{u\in\fg: [u, \zeta]=0
\text{ for all } \zeta\in \fh\bigr\}
\end{equation}
is the centralizer of the algebra $\fh$ in $\fg$.
Note that in the general case one has
$\fg_H\ne \fg_\fh$ (see Example~\ref{ex.4.6} below).

To understand the structure of the algebra $\fg_H$
we consider more carefully the centralizer $\fg_\fh$.
Since $\fh$ is a compact Lie algebra,
$\fh=\fz(\fh)\oplus[\fh,\fh]$, where $\fz(\fh)$ is the center of
$\fh$ and $[\fh,\fh]$ is a maximal
semisimple ideal of $\fh$. It is clear that
\begin{equation*}
\fz(\fh)\subset \fg_\fh
\quad\text{and}\quad
\fg_\fh\cap [\fh,\fh]=0
\quad\text{because}\
\langle \fg_\fh,[\fh,\fh] \rangle
=\langle [\fg_\fh,\fh],\fh \rangle=0.
\end{equation*}
Thus $\fg_\fh\oplus[\fh,\fh]$ is a subalgebra of $\fg$.

By its definition, $\fz(\fh)$
is a subspace of the center of the algebra $\fg_\fh$.
Moreover, by~(\ref{eq.4.12}), $\fa\subset\fg_\fh$.
The space
$\fa\oplus \fz(\fh)\subset\fg_\fh$ is a Cartan subalgebra of
$\fg_\fh$. Indeed, as we remarked above, the centralizer
$\fg_{\fa}^{\bbC} = \tilde{\mathfrak g}_0 =\fa^\bbC\oplus\fh^\bbC$
(by its definition, $\fk_0=\fh$) is a subalgebra of
$\fg^\bbC$ of maximal rank. Now since
$\fg_\fa=\fa\oplus \fz(\fh)\oplus[\fh,\fh]$ with
$\fa\oplus \fz(\fh)\subset\fg_\fh$, and
$\fg_\fh\oplus[\fh,\fh]$
is a subalgebra of $\fg$, then
$\rank \fg_\fa - \rank [\fh,\fh] = \rank \fg_\fh
= \dim (\fa\oplus \fz(\fh))$.

Since $\fa \oplus \ft_0$ is a Cartan subalgebra of the algebras $\fg$
and $\fg_\fa = \fa \oplus \fh$, the algebra $\ft_0$ is a Cartan subalgebra
of the algebra $\fh$ and, consequently, $\fz(\fh) \subset \ft_0$.
Moreover, since $[\ft_0,\fm] \subset \fm$, $[\ft_0,\fk] \subset \fk$,
$[\fa,\ft_0] = 0$, from Definitions \eqref{eq.4.4}
and \eqref{eq.4.7} we obtain that
\begin{equation}\label{eq.4.15}
\begin{split}
& [\ft_0,\fm_\lambda] \subset \fm_\lambda
\quad \text{and} \quad  [\ft_0,\fk_\lambda] \subset \fk_\lambda
\quad \text{for each} \quad \lambda \in \Sigma^+, \\
& [\ad_x, T] = 0 \quad \text{on} \quad \fm^+ \oplus \fk^+
\quad \text{for each} \quad x \in \ft_0.
\end{split}
\end{equation}

By definition, $[\fg_\fh,\fh]=0$. Hence the space
$\fg_\fh + \fh$
is a subalgebra of $\fg$. Since $\fa\subset\fg_\fh$ and
$\ft_0\subset\fh$, then $\fa \oplus \ft_0\subset \fg_\fh + \fh$.
But $\fa\oplus  \ft_0=\ft$ is a Cartan subalgebra of $\fg$.
This means that the complex reductive Lie algebras
$(\fg_\fh+\fh)^\bbC$, $\fg_\fh^\bbC$ and $\fh^\bbC$ are
$\ad(\ft^\bbC)$-invariant (regular) subalgebras of $\fg^\bbC$.
Taking into account that
$\ft\cap\fg_\fh=\fa\oplus\fz(\fh)$ and $\ft\cap\fh=\ft_0$,
we obtain the following direct sum decompositions:
\begin{equation}\label{eq.4.16}
\fg_\fh^\bbC=\fa^\bbC\oplus\fz(\fh)^\bbC\oplus
\sum_{\alpha\in\Delta_\fh} \tilde{\mathfrak g}_\alpha
\quad\text{and}\quad
\fh^\bbC=\ft_0^\bbC\oplus\sum_{\alpha\in\Delta_0}
\tilde{\mathfrak g}_\alpha,
\end{equation}
where $\Delta_\fh$ is some subset of the root system $\Delta$.
Since the spaces
$\fa\oplus\fz(\fh)\subset\ft$ and $\ft_0\subset\ft$
are Cartan subalgebras of the algebras $\fg_\fh$
and $\fh$ respectively, the decompositions above
are the root space decompositions of
$(\fg_\fh^\bbC,(\fa\oplus\fz(\fh))^\bbC)$ and
$(\fh^\bbC,\ft_0^\bbC)$, respectively. In particular, the subset
$\Delta_\fh\subset\Delta$ is the root system of
$(\fg_\fh^\bbC,(\fa\oplus\fz(\fh))^\bbC)$.

\begin{proposition}\label{pr.4.3}
The algebra $\fg_\fh$ is a $\sigma$-invariant
regular compact
subalgebra (possibly with nontrivial center)  of $\fg$, in particular,
\begin{equation}\label{eq.4.17}
\fg_\fh=\fm_\fh\oplus\fk_\fh,
\quad\text{where}\quad
\fm_\fh=\fg_\fh\cap\fm, \quad \fk_\fh=\fg_\fh\cap\fk,
\end{equation}
and $(\fg_\fh,\fk_\fh)$ is a symmetric pair. The space
$\fa$ is a Cartan subspace of $\fm_\fh\subset\fg_\fh$
and $\fa\oplus\fz(\fh)$ is a Cartan subalgebra of
$\fg_\fh$.
The root subsystem
$\Delta_\fh\subset\Delta$ in~\emph{(\ref{eq.4.16})}
of the reductive complex Lie algebra
$\fg_\fh^\bbC$
is defined by the following relation
\begin{equation}\label{eq.4.18}
\Delta_\fh=\big\{\alpha\in\Delta:
\alpha(\ft_0)=0,\
\alpha+\beta\not\in\Delta\  {\rm for\ all }\ \beta\in\Delta_0
\big\}.
\end{equation}

The set
$\Sigma_\fh=\{\lambda\in (\fa^\bbC)^*:
\lambda=\alpha|_{\fa^\bbC},\ \alpha\in \Delta_\fh\}\subset\Sigma$
is the set of restricted roots of the triple
$(\fg_\fh,\fk_\fh, \fa)$.
Each element $\lambda\in\Sigma_\fh\subset\Sigma$ has
multiplicity $1$, that is, $\dim\fm_\lambda=\dim\fk_\lambda=1$,
and  the following decompositions
are direct and orthogonal:
\begin{equation*}
{\mathfrak m}_\fh={\mathfrak a}\oplus
\sum_{\lambda\in\Sigma_\fh\cap \Sigma^+}{\mathfrak m}_\lambda,
\qquad
{\mathfrak k}_\fh=\fz(\fh)\oplus\sum_{\lambda\in\Sigma_\fh\cap \Sigma^+}
{\mathfrak k}_\lambda.
\end{equation*}
\end{proposition}
\begin{proof}
Since $\fh\subset\fk$, then
$\sigma(\fh)=\fh$ and the centralizer $\fg_\fh$ of
$\fh$ in $\fg$ is $\sigma$-invariant, i.e.\ (\ref{eq.4.17}) holds.

As we proved above, $\fa\oplus\fz(\fh)\subset\fg_\fh$ is a Cartan subalgebra
of $\fg_\fh$. Since $\fa\subset\fm_\fh=\fg_\fh\cap\fm$,
this subspace is a Cartan subspace of $\fm_\fh$ as it is a
maximal commutative subspace of $\fm$.

The root system
$\Delta_\fh\subset\Delta$ of the algebra
$\fg_\fh^\bbC$ in~(\ref{eq.4.16})
is a subset of the root system $\Delta'$
(see Lemma~\ref{le.4.1}).
Indeed, since $[\fg_\fh,\fh]=0$ and
$\ft_0=(1+\sigma)\ft$ is a subspace of
$\fh= \fg_\fa\cap\fk$, we obtain that
$\alpha(\ft_0)=0$ for all $\alpha\in\Delta_\fh$,
that is, $\Delta_\fh\subset\Delta'$.
Now to prove the relation~(\ref{eq.4.18}) describing the
root system $\Delta_\fh$ it is sufficient
to recall that for any roots $\alpha,\beta\in\Delta$, $\alpha+\beta\ne0$,
the commutator $[\tilde\fg_\alpha,\tilde\fg_\beta]=\tilde\fg_{\alpha+\beta}$
if $\alpha+\beta\in\Delta$ and $[\tilde\fg_\alpha,\tilde\fg_\beta]=0$
otherwise~\cite[Ch.\ III, Theorem 4.3]{He}. Taking into
account the second relation in~(\ref{eq.4.16}) we obtain~(\ref{eq.4.18}).
By Lemma~\ref{le.4.1} each restricted root
$\lambda=\alpha|_{\fa^\bbC}$, $\alpha\in\Delta_\fh\subset\Delta$
has multiplicity $1$.

Now all the latter assertions of the proposition follow
from~(\ref{eq.4.3}), (\ref{eq.4.4}), (\ref{eq.4.5}) and
the first decomposition in~(\ref{eq.4.16}).
\end{proof}

To describe the algebra $\fg_H$ we consider now in more detail the
subgroup $H\subset K$. By its definition, $H=K\cap G_\fa$.
By~\cite[Ch.\ VII, Corollary\ 2.8]{He}, the centralizer $G_\fa$
of the commutative subalgebra $\fa$ is connected (it is the union of all
maximal tori containing the torus $\exp\fa\subset G$).
Since $\fg_\fa =\fa\oplus\fh$ is the Lie algebra of the compact
Lie group $G_\fa$, we have that $G_\fa = \exp(\fa\oplus\fh)$.
But $\exp(\fa\oplus\fh)=\exp(\fa)\exp(\fh)$ because $[\fa,\fh]=0$.
The set $H_0=\exp \fh$ is the identity component of the Lie
group $H$ and $H_0\subset K$ because $\fh\subset\fk$.
Therefore $H= G_\fa \cap K= (\exp (\fa)\cap K) H_0$.

\begin{proposition}\label{pr.4.4}
The subalgebra $\fg_H\subset \fg_\fh\subset\fg$
is determined by the relation
\begin{equation}\label{eq.4.19}
\fg_H=\{u\in\fg_\fh: \Ad(D_\fa) u=u\},
\end{equation}
where $D_\fa$ stands for the commutative finite group
\begin{equation}\label{eq.4.20}
D_\fa=\exp(\fa)\cap K= \exp\big(\{v\in \fa: \exp v=\exp(-v)\}\big).
\end{equation}

The algebra $\fg_H\subset \fg_\fh$ is a $\sigma$-invariant
regular compact subalgebra of $\fg$. In particular,
\begin{equation*}
\fg_H=\fm_H\oplus\fk_H,
\quad\text{where}\quad
\fm_H=\fg_H\cap\fm, \quad \fk_H=\fg_H\cap\fk,
\end{equation*}
and $(\fg_H,\fk_H)$ is a symmetric pair. The space
$\fa$ is a Cartan subspace of $\fm_H\subset\fg_H$
and the space $\fa\oplus\fz(\fh)$ is a Cartan subalgebra of $\fg_H$.
For each $\lambda\in\Sigma^+$ and $g\in D_\fa$ we have
that $\Ad_g(\fm_\lambda\oplus\fk_\lambda)=\fm_\lambda\oplus\fk_\lambda$.
The set
\[
\Sigma_H=\{\lambda\in \Sigma_\fh:
\Ad_g|_{\fm_\lambda \oplus \fk_\lambda}
=\Id_{\fm_\lambda \oplus \fk_\lambda}
\ \text{for all}\ g\in D_\fa\}
\]
is the set of restricted roots of the triple
$(\fg_H,\fk_H, \fa)$.
Each element $\lambda\in \Sigma_H\subset\Sigma_\fh\subset\Sigma$ has
multiplicity $1$, that is, $\dim\fm_\lambda=\dim\fk_\lambda=1$.

The following decompositions
are direct and orthogonal:
\begin{equation*}
{\mathfrak m}_H={\mathfrak a}\oplus
\sum_{\lambda\in\Sigma_H\cap\Sigma^+}{\mathfrak m}_\lambda,
\qquad
{\mathfrak k}_H=\fz(\fh)\oplus\sum_{\lambda\in\Sigma_H\cap\Sigma^+}
{\mathfrak k}_\lambda.
\end{equation*}
\end{proposition}
\begin{proof}
Since $[\fh,\fg_\fh]=0$, the connected Lie group
$\Ad(H_0)$ with Lie algebra $\ad(\fh)$ acts trivially on
$\fg_\fh$. Taking into account that
$H=D_\fa H_0$ we obtain~(\ref{eq.4.19}). Since $K$
is a subgroup of the group of fixed points of certain
involutive automorphism on
$G$ acting by $\exp v\mapsto\exp(-v)$ on
$\exp(\fa)$, we obtain the second relation
in~(\ref{eq.4.20}).  The group
$D_\fa$ is a commutative finite group because
$\exp(\fa)\subset G$ is a toral subgroup,
$K$ is compact, and the intersection
$\exp(\fa)\cap K$ is a group of dimension $0$
($\fa\cap\fk=0$).

The algebra $\fg_H$ is $\sigma$-invariant because
by Definition~(\ref{eq.4.11}), $\sigma \Ad(H)\sigma= \Ad(H)$.
Since: $D_\fa=\{\exp v_1,\ldots,\exp v_s\}$, where
$v_j\in\fa$, $j=1,\ldots, s$; $[\fa,\fa\oplus\fz(\fh)]=0$;
and $\Ad_{\exp v_j}=\exp(\ad_{v_j})$,
then from relations~(\ref{eq.4.6}) we obtain that
$\Ad_{\exp v_j}v=v$ for all $v\in \fa\oplus\fz(\fh)$
and
\begin{equation*}
\begin{split}
\Ad_{\exp v_j}\xi_\lambda
&=\cos (\lambda'(v_j))\xi_\lambda{-}
\sin (\lambda'(v_j))\zeta_\lambda, \\
\Ad_{\exp v_j}\zeta_\lambda
&=\cos (\lambda'(v_j))\zeta_\lambda{+}
\sin (\lambda'(v_j))\xi_\lambda
\end{split}
\end{equation*}
for arbitrary $\xi_\lambda\in\fm_\lambda$ and
$\zeta_\lambda\in\fk_\lambda$, $\lambda\in\Sigma^+$
satisfying condition~\eqref{eq.4.6}. But
$\exp v_j=\exp(-v_j)$ and, consequently,
$\sin (\lambda'(v_j))=0$. Then
$\cos (\lambda'(v_j))\in\{1,-1\}$. Now it is clear that
$\fg_H=\fa\oplus\fz(\fh)\oplus\sum_{\lambda\in\Sigma_H\cap\Sigma^+}
({\mathfrak m}_\lambda\oplus {\mathfrak k}_\lambda)$,
where
\begin{equation*}
 \Sigma_H\cap\Sigma^+  =\{\lambda\in \Sigma_\fh\cap\Sigma^+:
\cos (\lambda'(v_j))=1
\ \text{for all}\  j=1,\ldots,s\}.
\end{equation*}
The last assertion of the proposition follows from
Proposition~\ref{pr.4.3}.
\end{proof}

\begin{remark}\label{re.4.5} \em
Put ${\mathfrak m}_H^+=\sum_{\lambda\in\Sigma_H\cap\Sigma^+}
{\mathfrak m}_\lambda$
and ${\mathfrak k}_H^+=\sum_{\lambda\in\Sigma_H\cap\Sigma^+}
{\mathfrak k}_\lambda$. Consider the orthogonal decompositions:
${\mathfrak m}^+={\mathfrak m}_H^+\oplus{\mathfrak m}_{\tw}^+$
and
${\mathfrak k}^+={\mathfrak k}_H^+\oplus{\mathfrak k}_{\tw}^+$,
where $\fm^+_*  = \sum_{\lambda \in \Sigma^+\backslash \Sigma_H}
\fm_\lambda$
and $\fk^+_*  = \sum_{\lambda \in \Sigma^+\backslash \Sigma_H}
\fk_\lambda$.
Since the following decompositions are orthogonal
\begin{equation*}
\fg_H=\fa\oplus\fm_H^+\oplus\fk_H^+\oplus\fz(\fh),
\quad
\fg=\fa\oplus\fm_H^+\oplus\fk_H^+
\oplus\fm_{\tw}^+\oplus\fk_{\tw}^+\oplus\fh=
\fg_H\oplus(\fm_{\tw}^+\oplus\fk_{\tw}^+)\oplus[\fh,\fh]
\end{equation*}
and $[\fg_H,\fh]=0$, one sees that $\fg_H\oplus[\fh,\fh]$
is a subalgebra of $\fg$. Then
\begin{equation*}
[\fg_H\oplus[\fh,\fh],\,\fm_{\tw}^+\oplus\fk_{\tw}^+]
\subset \fm_{\tw}^+\oplus\fk_{\tw}^+
\quad\text{and}\quad
[\fa\oplus\fh,\, \fm_H^+\oplus\fk_H^+]\subset \fm_H^+\oplus\fk_H^+.
\end{equation*}
Moreover, because by its definition,
$T ({\fm_\lambda})={\fk_\lambda}$,
$T ({\fk_\lambda})={\fm_\lambda}$,
for all restricted roots $\lambda\in\Sigma^+$, we have that
\begin{equation*}
T(\fm_H^+)=\fk_H^+,
\quad
T(\fk_H^+)=\fm_H^+
\quad\text{and}\quad
T(\fm_{\tw}^+)=\fk_{\tw}^+,
\quad
T(\fk_{\tw}^+)=\fm_{\tw}^+.
\end{equation*}
\end{remark}

\begin{example}\label{ex.4.6} \em
Let $G/K= \mathrm{SU}(n)/\mathrm{SO}(n)$, $n\geqslant2$.
Let $\fg= \mathfrak{su}(n)$ and $\fk= \mathfrak{so}(n)$ be the Lie algebras
of $G$ and $K$, i.e.\ the spaces of traceless skew-Hermitian complex
and skew-symmetric real $n \times n$ matrices, respectively. It is clear
that the space $\fa=\{\mathrm{diag}(\ri t_1,\ldots,\ri t_n),
t_j\in\bbR, \sum_{j=1}^n t_j = 0 \}$,
is a Cartan subspace of the space $\fm\subset\fg$.
Since $\fa$ is a Cartan subalgebra of the algebra $\fg$,
the centralizer $\fh\eqdef \fg_\fa \cap\fk=\fa\cap\fk=0$, that is,
the Lie algebra of the group $H$, is trivial
and $\fg_\fh=\fg$. Then by Proposition~\ref{pr.4.4},
$H=D_\fa$ and
$D_\fa\eqdef\exp (\fa) \cap K=\{\mathrm{diag}
(\varepsilon_1,\ldots,\varepsilon_n)\}$, where $\varepsilon_j=\pm 1$
and $\prod_{j=1}^n \varepsilon_j=1$.  It is easy then, on account of
\eqref{eq.4.19}, to verify that $\fg_H=\fa$ if $n\geqslant 3$
(in this case for any $k,j\leqslant n$, $k\ne j$ there
exists an element $g\in D_\fa$ for which $\varepsilon_k\varepsilon_j=-1$)
and $\fg_H=\fg=\mathfrak{su}(2)$ if $n=2$.
In the latter case the group
$H$ coincides with the center of the Lie group $\mathrm{SU}(2)$,
i.e.\ the action of $\Ad(H)$ on $\fg$ is trivial.
\end{example}

Fix in each subspace
$\fm_\lambda$, $\lambda\in\Sigma^+$, some basis
$\{\xi_\lambda^j, j=1, \dotsc ,m_\lambda\}$,
orthonormal with respect to the form
$\langle\cdot,\cdot \rangle$. In the case that
$\lambda\in\Sigma_\fh \cap\Sigma^+$,
$m_\lambda=1$, we have a unique vector
$\xi_\lambda^1$. By Lemma \ref{le.4.2}, for each
$\lambda\in\Sigma^+$ there exists a unique basis
$\{\zeta_\lambda^j,\, j=1, \dotsc,m_\lambda\}$ of
$\fk_\lambda$ such that for each pair
$\{\xi_\lambda^j,\zeta_\lambda^j,
j=1, \dotsc ,m_\lambda\}$, the condition~(\ref{eq.4.6}) holds.
The basis $\{\zeta_\lambda^j,
j=1, \dotsc,m_\lambda\}$, $\lambda\in\Sigma^+$, of $\fk_\lambda$,
is also orthonormal
due to the orthogonality of the operator $T$
(see Lemma~\ref{le.4.2}).
Fix also some orthonormal basis $\{X_1,\dotsc,X_r\}$
of the Cartan subspace $\fa$ and some orthonormal basis
$\{\zeta^k_0, k=1,\dotsc,\dim\fh\}$
of the centralizer $\fh$ of $\fa$ in $\fk$.
We will use the orthonormal basis
\begin{equation}\label{eq.4.21}
X_1,\dotsc,X_r;\;\, \xi_\lambda^j,\,
\zeta_\lambda^j,\, j=1,\dotsc,m_\lambda, \,\lambda\in\Sigma^+;\;
\zeta^k_0,\; k=1,\dotsc,\dim\fh,
\end{equation}
of the algebra $\fg$ in our calculations below.

\subsection{The canonical complex structure on $G/H\times W^+$}
\label{ss.4.2}

Each element  $w\in W^+$ is regular in $\fm$.
Therefore if for some $k\in K$,  $\Ad_k w\in W^+$, then
$\Ad_k(\fa)=\fa$ because $0=[w,\fa]=[\Ad_k w,\Ad_k\fa]$.
Since the Weyl group of the symmetric pair
$(\fg,\fk)$ is simply transitive on the set of Weyl chambers in
$\fa$~(see \cite[Ch.\ VII, Theorem 2.12]{He}),
$\Ad(K)w\cap W^+=\{w\}$. Then by
definition~(\ref{eq.4.11}) of the group $H$, the map
\begin{equation*}
K/H\times W^+\to \fm^R,\quad
(kH,w)\mapsto \Ad_k w,
\end{equation*}
is a well-defined diffeomorphism.
Thus the map
\begin{equation*}
f^+ \colon  G/H\times W^+\to G\times_K\fm^R,
\quad
(gH,w)\mapsto[(g,w)]_K,
\end{equation*}
is a well-defined $G$-equivariant diffeomorphism
of $G/H\times W^+$ onto the subset
$D^+=G\times_K\fm^R$,
which is an open dense subset of $G\times_K\fm$.

It is clear that the following
diagram is commutative
\begin{equation}\label{eq.4.22}
\begin{array}{rrl}
G\times W^+&\stackrel{\mathrm{id}}\longrightarrow&G\times \fm^R\\
\noalign{\smallskip}
{\downarrow\makebox[20pt]{\,$\scriptstyle{\pi_H\times \mathrm{id}}$}}
&&{\makebox[10pt]{$\scriptstyle{\pi}$}\downarrow} \\ \noalign{\smallskip}
G/H\times W^+&
\stackrel{f^+}\longrightarrow&G\times_K\fm^R
\end{array}
\hskip20pt
\begin{array}{rrl}
\makebox[35pt]{$(g,w)$\hfill}&\stackrel{\mathrm{id}}\longmapsto&(g,w)\\
\noalign{\smallskip}
{\downarrow\makebox[18pt]{\;$\scriptstyle{\pi_{H}\times \mathrm{id}}$}}
&&{\makebox[10pt]{$\scriptstyle{\pi}$}\downarrow} \\ \noalign{\smallskip}
(gH,w)&
\stackrel{f^+}\longmapsto&[(g,w)]_K
\end{array},
\end{equation}
where $\pi_H \colon  G\to G/H$ is the canonical projection.

The submersion (projection)
$\pi \colon  G\times\fm\to G\times_K\fm$ is (left)
$G$-equivariant. Therefore, the kernel
${\sK}\subset T(G\times {\mathfrak m})$ of the tangent map
$\pi_{*} \colon  T(G\times\fm)\to T(G\times_K\fm)$
is generated by the global (left)
$G$-invariant vector fields $\zeta^L$, for
$\zeta\in {\mathfrak k}$, on $G\times {\mathfrak m}$,
\begin{equation}\label{eq.4.23}
\zeta^L{(g,w)}=(\zeta^{l}(g) ,[w,\zeta])\in T_g G\times T_w\fm,
\end{equation}
where the tangent space
$T_w\fm$ is identified canonically with the space
$\fm$.

Note also here that the tangent space $T_o(G/H)$ at
$o=\{H\}\in G/H$ can be identified naturally with the
space
$
\fm\oplus\fk^+
= \fa\oplus\fm^+\oplus\fk^+,
$
because by definition $\fk=\fh\oplus\fk^+$
and $\fh$ is the Lie algebra of the group $H$.

We can rewrite the expression for the vector field
$Y^\xi$, $\xi\in\fm$, in~(\ref{eq.3.15}) on the product
$G\times\fm$ in a simpler way
using Lemma~\ref{le.4.2} and the basis~(\ref{eq.4.21}) of
the algebra $\fg$. Indeed, for $w\in W^+\subset\fa$, by~(\ref{eq.3.15})
we have
\begin{align}\label{eq.4.24}
Y^{X_j}(g,w) & =(X_j^{l}(g),\,- \ri X_j),
\quad j=1,\dotsc,r, \\
\label{eq.4.25}
Y^{\xi_\lambda^j}(g,w) &
=\left(\left(\tfrac{1}{\cosh\lambda'_w}\cdot\xi_\lambda^j-  \ri \,
\tfrac{\cosh\lambda'_w - 1}{\sinh\lambda'_w}\cdot
\zeta_\lambda^j\right)^{l}(g)
,\,- \ri \,\tfrac{\lambda'_w\cosh\lambda'_w}{\sinh\lambda'_w}
\cdot\xi_\lambda^j\right),
\end{align}
where $j=1,\dotsc,m_\lambda$, $\lambda\in\Sigma^+$, and
$\lambda'_w\eqdef\lambda'(w)\in\mathbb{R}$.

By Lemma~\ref{le.4.2},
for any (regular) element $w\in W^+\subset\fa\subset\fm$,
the map $\ad_w \colon \fk^+\to\fm$, $\zeta\mapsto [w,\zeta]$,
is nondegenerate
and thus $\fa\oplus\ad_w(\fk^+)=\fm$. Therefore
\begin{equation*}
T_{(e,0)}(G\times\fm)
=\fg\times\fm
=(\fm\times\fm)\oplus {\sK}(e,0)
=((\fm\oplus\fk^+)\times\fa)\oplus {\sK}(e,0),
\end{equation*}
where, as we remarked above, $\fm\oplus\fk^+$ is
identified naturally with the tangent space of $G/H$
at the point $\{H\}$.

Using Lemma~\ref{le.4.2}
again (note that in \eqref{eq.4.25},
$\coth \lambda'_w \cdot \zeta^j_\lambda\in \fk$
and the second component is $\coth \lambda'_w \cdot [w,\zeta^j_\lambda]$)
and from the expression~(\ref{eq.4.23}) of the kernel $\sK(g,w)$
of $\Pi_*(g,w)$, we have for all $w\in W^+$ that
\begin{equation}\label{eq.4.26}
\Pi_*(g,w)(Y^{\xi_\lambda^j}(g,w))
=\Pi_*(g,w)\left(\left(\tfrac{1}{\cosh\lambda'_w}\cdot\xi_\lambda^j- \ri \,
\tfrac{-1}{\sinh\lambda'_w}\cdot\zeta_\lambda^j\right)^{l}(g),\,0\right).
\end{equation}

To describe the
$G$-invariant Ricci-flat K\"ahler metrics on $T(G/K)$ associated
to the canonical complex structure $J^K_c$, we first attempt to
describe such metrics on the
$G$-invariant open and dense subset
\begin{equation*}
T^+(G/K)\eqdef (\phi\circ f^+)(G/H\times W^+)\subset T(G/K),
\end{equation*}
isomorphic to the direct product
$G/H\times W^+\subset G/H\times\mathbb{R}^r$,
where the action of the group $G$
is natural on the first component and trivial
on the second component
(see the commutative diagram~(\ref{eq.4.22})).
Since this diffeomorphism is $G$-equivariant, we denote the
corresponding complex structure on $G/H\times W^+$
also by $J^K_c$.
Consider the coordinates $(x_1,\dotsc,x_r)$ on $W^+$
associated with the basis $(X_1,\dotsc,X_r)$ of $\fa$,
that is, $w=x=x_1X_1+\cdots+x_rX_r$. By the $G$-invariance it suffices
to describe the operators $J^K_c$ only at the
points $(o,x)\in G/H\times W^+$, where $o=\{H\}$.
Then from~(\ref{eq.4.24}), (\ref{eq.4.26}) and the
commutative diagram~\eqref{eq.4.22} we see that
\begin{equation}\label{eq.4.27}
\begin{split}
J^K_c(o,x)(X_j,0)
&= \Big(0,\frac{\partial}{\partial x_j}\Big),
\; j=1,\dotsc ,r, \\
J^K_c(o,x) (\xi_\lambda^j,\,  0)
&=\Big(\frac{- \cosh\lambda'_x}{\sinh\lambda'_x}
\cdot\zeta_\lambda^j, \, 0\Big), \
j=1,\dotsc ,m_\lambda,\ \lambda\in\Sigma^+,
\end{split}
\end{equation}
where, recall,
$\lambda'_x=\lambda'(x)\in\mathbb{R}$.
Here
$T_o(G/H)$ is identified naturally with the space
$\fa\oplus \sum_{\lambda\in\Sigma^+}\fm_\lambda
\oplus \sum_{\lambda\in\Sigma^+}\fk_\lambda$ and,
in the first equation, we use naturally the usual
basis $\{\partial/\partial x_j\}$ of
$T_x\bbR^r$  ($W^+$ is an open subset of $\bbR^r$).

Often we will use the second relation in~(\ref{eq.4.27})
in the more general form:
\begin{equation*}
J^K_c(o,x)(\xi,  0)
=\Big(\frac{- \cos\ad_x}{\sin\ad_x}\xi,0\Big),
\quad\text{where }\xi\in\fm^+.
\end{equation*}

Let $F=F(J^K_c)$ be the
subbundle of $(1,0)$-vectors of the structure $J^K_c$
on the manifold $G/H\times W^+$.
We can substantially simplify calculations by working on the
manifold $G\times W^+$ with the subbundle
$\sF=(\pi_H\times \mathrm{id})_*^{-1}(F)$ rather than on the manifold
$G/H\times W^+$ with $F$ (see Subsection~\ref{ss.2.1}).
 From~(\ref{eq.4.27}) (see also~(\ref{eq.4.24}) and~(\ref{eq.4.26}))
it follows that the subbundle $\sF$ of $T^\mathbb{C}(G\times W^+)$
is generated by the kernel $\sH$ of the submersion
$\pi_H\times \mathrm{id}$,
\begin{equation}\label{eq.4.28}
\sH(g,x)=\{(\zeta^l(g),0),\,\zeta\in\fh\},
\quad g\in G, \;\; x\in W^+,
\end{equation}
 and the left $G$-invariant vector fields
\begin{equation}\label{eq.4.29}
\begin{split}
Z^{X_j}(g,x) &
= \Big(X_j^l(g),- \ri\, \frac{\partial}{\partial x_j}\Big),
\quad j=1,\dotsc,r,\\
\; Z^{\xi_\lambda^j}(g,x)
& =\left(\left(\frac{1}{\cosh\lambda'_x}\cdot\xi_\lambda^j-  \ri \,
\frac{-1}{\sinh\lambda'_x}\cdot\zeta_\lambda^j\right)^{l}(g) ,0\right),
\end{split}
\end{equation}
where $j=1,\dotsc,m_\lambda$, $\lambda\in\Sigma^+$, and
$\lambda'_x\eqdef\lambda'(x)\in \mathbb{R}$.

To simplify calculations in the forthcoming subsection,
we will use for the vector fields of the second family
the following more general expression,
\begin{equation}\label{eq.4.30}
Z^{\xi}(g,x)=\left(\left(R_x\xi-  \ri \,
S_x\xi\right)^{l}(g) ,0\right),
\quad\xi\in\fm^+,
\end{equation}
in terms of the two operator-functions
$R \colon W^+\to\End(\fg)$
and $S \colon W^+\to\End(\fg)$
on the set $W^+$ such that
\begin{equation}\label{eq.4.31}
\begin{split}
R_x\eta&=\frac{1}{\cos \ad_x}\, \eta
\quad\text{if}\quad\eta\in\fm^+\oplus\fk^+,\qquad
R_x\eta=0 \quad\text{if}\quad\eta\in\fa\oplus\fh, \\
S_x\eta&=\frac{-1}{\sin \ad_x}\, \eta
\quad\text{if}\quad\eta\in\fm^+\oplus\fk^+,\qquad
S_x\eta=0 \quad\text{if}\quad\eta\in\fa\oplus\fh,
\end{split}
\end{equation}
where, recall, $x=\sum_{j=1}^rx_jX_j\in W^+$.
Remark also that
$\frac{1}{\cos \ad_x}\, \eta=\eta$  if $\eta \in \fa \oplus \fh$
but $R_x\eta=0$ in this case.
Since the operator $\ad_x$ is skew-symmetric with
respect to the scalar product on $\fg$, each operator $R_x$
is symmetric and $S_x$ is skew-symmetric:
\begin{equation}\label{eq.4.32}
\langle R_x\eta_1,\eta_2 \rangle=
\langle \eta_1,R_x\eta_2 \rangle,
\ \langle S_x\eta_1,\eta_2 \rangle=
\langle \eta_1,-S_x\eta_2 \rangle,
\ x\in W^+, \ \eta_1,\eta_2\in\fg.
\end{equation}
Moreover, it is clear that for all $x\in W^+$, the restrictions
$R_x|_{\fm^+\oplus\fk^+}$ and
$S_x|_{\fm^+\oplus\fk^+}$ are nondegenerate and by Remark~\ref{re.4.5}
the following relations hold:
\begin{equation}\label{eq.4.33}
R_x(\fm_s^+)=\fm_s^+, \
R_x(\fk_s^+)=\fk_s^+, \
S_x(\fm_s^+)=\fk_s^+, \
S_x(\fk_s^+)=\fm_s^+, \
s\in\{H,\tw\}.
\end{equation}
It is clear also that
\begin{equation}\label{eq.4.34}
R_x|_{\fm_\lambda\oplus\fk_\lambda}=
\frac{1}{\cosh \lambda'_x}\Id_{\fm_\lambda\oplus\fk_\lambda},
\quad
S_x|_{\fm_\lambda\oplus\fk_\lambda}=
\frac{1}{\sinh \lambda'_x}\Id_{\fm_\lambda\oplus\fk_\lambda}
\cdot T|_{\fm_\lambda\oplus\fk_\lambda}
\end{equation}
for all $\lambda\in\Sigma^+$, and $[R_x,T]=[S_x,T]=0$ on
$\fm^+\oplus\fk^+$ for all
$x\in W^+$, where, recall that the operator
$T$ is defined by expression~(\ref{eq.4.7}).

\begin{proposition}\label{pr.4.7}
Assume that the group $G$ is semisimple.
Let $f_h\colon G/H\times W^+\to\bbC$ be a $G$-invariant $J^K_c$-harmonic function,
that is, $\partial\bar\partial f_h= 0$. Then $f_h=\mathrm{const}$.
\end{proposition}
\begin{proof}
It is clear that
$$
(\pi_H\times \mathrm{id})^*(\partial\bar\partial f_h)
=(\pi_H\times \mathrm{id})^*(\rd \bar\partial f_h)
=\rd((\pi_H\times \mathrm{id})^*\bar\partial f_h)
=0.
$$
Let us calculate the 1-form
$\alpha_h=(\pi_H\times \mathrm{id})^*\bar\partial f_h$ on
$G\times W^+$. By its definition,
$$
\alpha_h|_{\sF}=0
\quad\text{and}\quad \alpha_h|_{\overline{\sF}}=
\rd\bigl(
(\pi_H\times \mathrm{id})^* f_h)\bigr)|_{\overline{\sF}}.
$$
Since the function $f_h$ is $G$-invariant,
$f_h$ is determined uniquely by some smooth function
$f\colon W^+\to \bbC$, $(x_1,\ldots,x_r) \mapsto f(x_1,\ldots,x_r)$.
Taking into account the description of the vector fields
$Z^{X_j},Z^{\xi_\lambda^j}$ generating the subbundle
$\sF$ of $T(G\times W^+)$ (see~\eqref{eq.4.29}), we obtain that
$\alpha_h=\frac12\sum_{j=1}^r\frac{\partial f}{\partial x_j}
\bigl(\ri\, \theta^{X_j}+\rd x_j\bigr)$,
where $\theta^{X_j}$ is the left $G$-invariant $1$-form on
$G$ (considered as a form on $G\times W^+$) such that
$\theta^{X_j}(e)(\xi)=\langle X_j,\xi \rangle$, $\xi\in\fg$.
But
$$
2\rd\alpha_h= \ri  \sum_{j=1}^r
\rd\Bigl(\frac{\partial f}{\partial x_j}\Bigr)\land
\theta^{X_j}
+ \ri  \sum_{j=1}^r
\frac{\partial f}{\partial x_j} \rd \theta^{X_j}
=0.
$$
It is clear that the first and second summands above vanish
(they are independent as differential two-forms on $G\times W^+$).
Since the left-invariant forms $\{\theta^{X_j}\}_{j=1}^r$
are independent on $G$,
we obtain that  $\frac{\partial f}{\partial x_j}=c_j$, $c_j\in\bbC$.
Taking into account that
$\sum_{j=1}^r c_j
 \rd \theta^{X_j}|_{e}(\xi,\eta)=-\langle \sum_{j=1}^r c_j
X_j,[\xi,\eta] \rangle$
and the algebra $\fg$ is semisimple ($[\fg,\fg]=\fg$), we obtain that
$c_j=0$ for all $j= 1, \dotsc, r$.
\end{proof}

\subsection{Invariant Ricci-flat K\"ahler metrics on $G/H\times W^+$}
\label{ss.4.3}

Let $\cK(G/H \times W^+)=\{(\bfg,\omega, J^K_c)\}$ (resp.
$\cR(G/H \times W^+)=\{(\bfg,\omega, J^K_c)\}$) be the set of all
$G$-invariant K\"ahler (resp.\ Ricci-flat K\"ahler)
structures on $G/H \times W^+$, identified also with the set
$\cK(T^+(G/K))$ (resp. $\cR(T^+(G/K))$) of all
$G$-invariant K\"ahler (resp.\ Ricci-flat K\"ahler)
structures on  the open dense subset
$T^+(G/K)$ of $T(G/K)$, associated with $J^K_c$, via the
$G$-equivariant diffeomorphism $\phi \circ f^+\colon G/H \times W^+
\to T^+(G/K)$.
Put
$$\{T_1,\dotsc,T_n\}=\{Z^{X_1},\dotsc,Z^{X_r}\}\cup
\{Z^{\xi_\lambda^j}, \,\lambda\in\Sigma^+,\,j=1, \dotsc, m_\lambda\}.
$$
\begin{theorem}
\label{th.4.8}
Let $\cK(G\times W^+)=\{\tio\}$ be the set of all
$2$-forms $\tio$ on $G\times W^+$ such  that
\begin{mlist}
\item[$(1)$]
the form $\tio$ is closed;
\item[$(2)$]
the form $\tio$ is left $G$-invariant and right $H$-invariant;
\item[$(3)$]
the kernel of $\tio$ coincides with
the subbundle $\sH\subset T(G\times W^+)$ in~{\em (\ref{eq.4.28});}
\item[$(4)$]
$\tio(T_j,T_k)=0$,\ $j,k=1,\dotsc,n;$
\item[$(5)$]
$\ri \, \tio(T,\overline{T})>0$ for each $T=\sum_{j=1}^n c_jT_j$,
where $(c_1,\ldots,c_n)\in\mathbb{C}^n\setminus\{0\}.$
\end{mlist}

Let $\cR(G\times W^+)=\{\tio\}$ be the subset of the set
$\cK(G\times W^+)=\{\tio\}$ such that for its elements $\tio$
(in addition) the following condition holds:
\begin{mlist}
\item[$(6)$]
$\det\big(\tio(T_j,\overline{T_k}) \big) =
\mathrm{const}$ on $G\times W^+.$
\end{mlist}
Then {\em(i)}
For any $2$-form $\tio\in \cK(G\times W^+)$
there exists a unique
$2$-form $\omega$ on $G/H\times W^+\cong T^+(G/K)$ such that
$(\pi_H \times \mathrm{id})^\ast \omega =\tio$.
The map $\tio\mapsto \omega$ is a one-to-one map from
$\cK(G\times W^+)$ onto $\cK(G/H\times W^+)\cong \cK(T^+(G/K))$.

{\em (ii)}  If the group $G$ is semisimple then the restriction of this map to
$\cR(G\times W^+)$
is a one-to-one map from
$\cR(G\times W^+)$ onto $\cR(G/H\times W^+)\cong \cR(T^+(G/K))$.
\end{theorem}

\begin{proof}
From $(1)\!\!-\!\!(3)$, $\omega$ becomes into a
$G$-invariant symplectic structure on $G/H\times W^+$.
Then item (i) of the theorem follows from
Lemma~\ref{le.2.1}, using $(4)$ and $(5)$ and taking
$\sF = (\pi_H \times \mathrm{id})^{-1}_\ast(F)$, where $F = F(J^K_c)$
is the subbundle of $(1,0)$-vectors of the structure $J^K_c$
on the manifold $G/H \times W^+$.
To prove assertion (ii) of the
theorem, let $\omega\in \cK(G/H\times W^+)$ and
$\tio=(\pi_H \times \mathrm{id})^\ast \omega $.
By Proposition~\ref{pr.3.6},
the form $\underline{\omega}=((\phi\circ f^+)^{-1})^*
\omega\in \cR(T^+(G/K))$ if and only if the
$G$-invariant function
$\ln \cS$ ($\cS=\cS(\underline{\omega})$, see~(\ref{eq.3.20})) on
$T^+(G/K)\cong G/H\times W^+$ is a
$J^K_c$-harmonic function. In this case, by
Proposition~\ref{pr.4.7},
$\cS=\mathrm{const}$. Now to complete the proof of the theorem
it is sufficient to remark that, by~(\ref{eq.3.19}),
$\Pi^*\cS= \det \Bigl((\Pi^\ast \underline{\omega})\big( Y^{\xi_j},
\overline{Y^{\xi_{k}}}\bigr)\Bigr)$;
by the commutative diagram~\eqref{eq.4.22},
$(\Pi^\ast\underline{\omega})|_{G\times W^+}=\tio$;
and by the definition
of the vector fields
$Z^{\xi}$, $\xi\in\fm$, the difference
$Z^\xi-Y^\xi$ belongs to the kernel of the tangent
map $\Pi_*$:
\begin{equation*}
\begin{split}
\mathrm{const}=\det\Bigl(\!(\Pi^\ast\underline{\omega})\big( Y^{\xi_j},
\overline{Y^{\xi_{k}}}\bigr)|_{G\times W^+}\!\Bigr)
&=\det\Bigl(\!(\Pi^\ast\underline{\omega})\big( Z^{\xi_j},
\overline{Z^{\xi_{k}}}\bigr)|_{G\times W^+}\!\Bigr) \\
&=\det\Bigl(\tio\big( Z^{\xi_j},
\overline{Z^{\xi_{k}}}\bigr)\!\Bigr). \qedhere
\end{split}
\end{equation*}
\end{proof}

\begin{remark}\label{re.4.9} \em
Note that condition $(5)$
of the previous theorem is equivalent
to the following condition: the Hermitian
matrix-function $\mathbf{w}$ on
$W^+$ with entries $w_{jk} (x)
=\ri \tio(T_j,\overline{T_k})(e,x)$, $j,k=1,\dotsc,n$,
is positive-definite.
\end{remark}

\begin{corollary}\label{co.4.10}
Let $\omega\in\cK(G/H\times W^+)$ and
$\tio=(\pi_H \times \mathrm{id})^\ast \omega $. Then the form
$\underline{\omega}=((\phi\circ f^+)^{-1})^*\omega\in \cK(T^+(G/K))$.
Suppose that there exists a smooth form (extension)
$\underline{\omega}_0$ on the whole tangent bundle $T(G/K)$
such that $\underline{\omega}_0=\underline{\omega}$ on
$T^+(G/K)$. Then the form $\underline{\omega}_0$ determines
a $G$-invariant K\"ahler structure
on $T(G/K)$ (associated
to the canonical complex structure $J^K_c$) if and only if
for each limit point $\overline{x}\in \overline{W^+}\setminus W^+\subset\fa$
and some sequence $x_m\in W^+$, $m\in\mathbb{N}$,
such that $\lim_{m\to\infty}x_m=\overline{x}$,
the Hermitian matrix $\mathbf{w}(\overline{x})$ with entries
$w_{jk}(\overline{x})=\lim_{m\to\infty} w_{jk}(x_m)=
\lim_{m\to\infty} \ri \tio(T_j,\overline{T_k})(e,x_m) $, $j,k=1,\dotsc,n$,
is positive-definite.
\end{corollary}
\begin{proof}
The form $\tio$ on $G\times W^+$ is left $G$-invariant,
right $H$-invariant and $\ker\tio=\sH$.
By the commutative diagram~\eqref{eq.4.22} there exists
a unique 2-form $\underline{\tio}$ on the space $G\times\fm^R$
which is left $G$-invariant,
right $K$-invariant, $\ker\underline{\tio}=\sK|_{G\times\fm^R}$,
$\underline{\tio}|_{G\times W^+}=\tio$ and
$\Pi^* \underline{\omega}=\underline{\tio}$.
Here, recall, ${\sK}\subset T(G\times {\mathfrak m})$ is
the kernel of the tangent map
$\pi_{*} \colon  T(G\times\fm)\to T(G\times_K\fm)$.
By Lemma~\ref{le.2.1}, Item $(5)$,
the form (extension) $\underline{\omega}_0$
determines a $G$-invariant K\"ahler structure
on $T(G/K)$ if and only if
the Hermitian matrix $\mathbf{v}(x)$ for each
$x\in\fm\setminus\fm^R$ with entries
$v_{jk}(x)=\ri (\Pi^* \underline{\omega}_0)
\big( Y^{\xi_j},\overline{Y^{\xi_{k}}}\bigr)(e,x)$
is positive-definite. Remark that $\Ad(K)(\overline{W^+})=\fm$
and $\Ad(K)({W^+})=\fm^R$.
Since the form $\Pi^* \underline{\omega}_0$
on $G\times\fm$ is smooth,
$$
v_{jk}(\overline{x})=\lim_{m\to\infty} v_{jk}(x_m)=
\lim_{m\to\infty} \ri (\Pi^* \underline{\omega}_0)
\big( Y^{\xi_j},\overline{Y^{\xi_{k}}}\bigr)(e,x_m).
$$
But as we remark above, at each point $(e,x_m)\in G\times W^+$
$$
(\Pi^* \underline{\omega}_0)(e,x_m)=
(\Pi^* \underline{\omega})(e,x_m)=\underline{\tio}(e,x_m)
=\tio(e,x_m)
$$
restricted to the subspace $T_{(e,x_m)}(G\times W^+)\subset
T_{(e,x_m)}(G\times \fm)$.
Taking into account again (as in the proof of Theorem~\ref{th.4.8})
that the difference $(Z^\xi-Y^\xi)(e,x_m)$, $\xi\in\fm$,
belongs to the kernel of the tangent
map $\Pi_*(e,x_m)$ we obtain that
$v_{jk}(\overline{x})=\lim_{m\to\infty} \ri \tio
\big( Z^{\xi_j},\overline{Z^{\xi_{k}}}\bigr)(e,x_m) $.
Now all the other required properties of the
form $\underline{\omega}_0$ follow by continuity.
\end{proof}

\section{Description of the space $\cR(G\times W^+)$}
\label{s.5}

For any vector $a\in \fg$, denote by $\theta^a$ the left
$G$-invariant 1-form on the group $G$ such that
$\theta^a(\xi^{l})=\langle a,\xi \rangle$.
Since $r_g^*\theta^a=\theta^{\Ad_g a}$,
where $g\in G$, the form $\theta^a$ is
right $H$-invariant if and only if
$\Ad_h a=a$ for all $h\in H\subset G$.
Because
\begin{equation}\label{eq.5.1}
\rd\theta^a(\xi^{l},\eta^{l})
=-\theta^a([\xi^{l},\eta^{l}])
=-\langle a,[\xi,\eta] \rangle,
\end{equation}
the $G$-invariant form $\omega^a$ on $G$,
\begin{equation}\label{eq.5.2}
\omega^a(\xi^{l},\eta^{l})\eqdef\langle a,[\xi,\eta] \rangle,
\quad
\xi,\eta\in\fg,
\end{equation}
is a closed $2$-form on $G$.

Let $\mathrm{pr}_1 \colon G\times W^+\to G$ and
$\mathrm{pr}_2 \colon G\times W^+\to W^+$
be the natural projections. Choosing some orthonormal basis
$\{e_1, \dotsc, e_N\}$ of the Lie algebra
$\fg$, where $e_j=X_j$, $j=1,\dotsc,r$, put
$\tith^{e_k}\eqdef \mathrm{pr} _1^*(\theta^{e_k})$ and
$\tio^{e_k}\eqdef \mathrm{pr} _1^*(\omega^{e_k})$.
For any vector-function $\mathbf{a} \colon W^+\to\fg$,
$\mathbf{a}(x)=\sum_{k=1}^{N} a^k(x)e_k$, denote by
$\tith^{\mathbf{a}}$ (resp.\ $\tio^{\mathbf{a}}$)
the $G$-invariant 1-form $\sum_{k=1}^{N}a^k \cdot\tith^{e_k}$
(resp.\ 2-form $\sum_{k=1}^{N}a^k \cdot \tio^{e_k}$).
Then we have

\begin{theorem}
\label{th.5.1}
Let $\tio$ be a $2$-form belonging to
$\cK(G\times W^+)$, where the compact Lie group
$G$ is semisimple. Then there exists a unique (up to a real
constant) smooth function
$f \colon W^+\to \mathbb{R}$,
$x\mapsto f(x)$,
and a unique smooth vector-function
$\mathbf{a} \colon W^+\to \fg_H$ given by
\begin{equation}\label{eq.5.3}
\begin{split}
& \mathbf{a}(x) =\sum_{j=1}^{r}\frac{\partial f}{\partial x_j}(x) X_j
+z_\fh+a^\fk(x)+a^\fm(x), \ \text{where}
\quad z_\fh\in\fz(\fh), \;\;\\
& a^\fk(x)=\sum_{\lambda\in\Sigma_H\cap\Sigma^+}
\tfrac{c^\fk_\lambda}{\cosh \lambda'(x)}\zeta^{1}_\lambda
\in \fk_H^+, \;\;
a^\fm(x)=\sum_{\lambda\in\Sigma_{H}\cap\Sigma^+}
\tfrac{c^\fm_\lambda}{\sinh \lambda'(x)}\xi^{1}_\lambda
\in \fm_H^+,
\end{split}
\end{equation}
$c^\fm_\lambda, c^\fk_\lambda\in\bbR$, such that $\tio$ is the
exact form expressed
in terms of $\mathbf a$ as
\begin{equation}\label{eq.5.4}
\tio=\rd \tith^{\mathbf{a}}
=\sum_{j=1}^r  \rd x_j \land \tith^{\mathbf{a}_{[j]}}  -\tio^{\mathbf{a}},
\quad \text{where}\quad
\mathbf{a}_{[j]}=\frac{\partial \mathbf{a}}{\partial x_j}.
\end{equation}
Moreover, for all points $x\in W^+$, the following conditions
$(1)\!\!-\!\!(3)$ hold:
\begin{mlist}
\item
[$(1)$] the components $a^\fk(x)+z_\fh$ and $a^\fm(x)$ of
the vector-function ${\mathbf a}(x)$ in~\emph{(\ref{eq.5.3})}
satisfy the following commutation relations
\begin{equation}\label{eq.5.5}
\begin{split}
\bigl(R_x \cdot \ad_{a^\fk(x)} \cdot R_x
+S_x \cdot \ad_{a^\fk(x)} \cdot S_x
+ (R^2_x + S^2_x)\ad_{z_\fh}\big)(\fm^+) & =0, \\
\bigl(R_x \cdot \ad_{a^\fm(x)} \cdot S_x-
S_x \cdot \ad_{a^\fm(x)} \cdot R_x\bigr)(\fm^+) & =0;
\end{split}
\end{equation}
$z_\fh=0$ if $a^\fk(x)\equiv 0$ and $G/K$
is an irreducible Riemannian symmetric space;
\item
[$(2)$]
the Hermitian $p\times p$-matrix-function
${\mathbf w}_H(x)=\bigl(w_{k|j}(x)\bigr)$, $p=\dim\fm_H=\dim\fa+
\card(\Sigma_H\cap\Sigma^+)$,
with indices $k,j\in\{1,\dotsc,r\}
\cup\{{}_\lambda^{1},\, \lambda\in \Sigma_H\cap\Sigma^+\}$
and entries
\begin{equation*}
\begin{split}
w_{k|j}(x)&=2\frac{\partial^2 f}{\partial x_k\partial x_j}(x),
\quad k,j\in\{1,\ldots,r\}, \\
w_{k|{}_\lambda^{1}}(x)
&=2\lambda'(X_k)\Bigl(\ri  \dfrac{c^\fk_\lambda}{\cosh^2 \lambda'_x}
-\dfrac{c^\fm_\lambda}{\sinh^2 \lambda'_x} \Bigr), \
k\in\{1,\ldots,r\}, \;\lambda\in \Sigma_H\cap\Sigma^+, \\
w_{{}_\lambda^{1}|{}_\mu^{1}}(x), & \quad
\lambda,\mu\in \Sigma_H\cap\Sigma^+,  \quad
\text{determined by}~(\ref{eq.5.6}),
\end{split}
\end{equation*}
 is positive-definite;
\item
[$(3)$]
if $\fm_{\tw}^+\ne 0$ then
the Hermitian $s\times s$-matrix
${\mathbf w}_{\tw}(x)=\bigl(w_{{}_\lambda^j|{}_\mu^k}(x)\bigr)$,
$s=\dim\fm_{\tw}^+= \Sigma_{\lambda \in \Sigma^+ \setminus
\Sigma_H,} m_{\lambda}$,
with indices ${}_\lambda^j,\, {}_\mu^k\in
\{{}_\lambda^j,\, \lambda\in\Sigma^+\setminus \Sigma_H,
j=1,\ldots,m_\lambda\}$ and entries
\begin{align}\label{eq.5.6}
w_{{}_\lambda^j|{}_\mu^k}(x)
& = -\frac{2\ri}{\sinh \lambda'_x \sinh \mu'_x}\langle
(\ad_{a^\fk(x)+z_\fh})\zeta^j_\lambda, \, \zeta^k_\mu\rangle \\
& \quad - \frac{2}{\cosh \lambda'_x \sinh \mu'_x}\langle
(\ad_{a^\fa(x)+a^\fm(x)})\xi^j_\lambda, \, \zeta^k_\mu\rangle  \notag
\end{align}
is positive-definite.
\end{mlist}
If in addition
\begin{mlist}
\item
[$(4)$] either $\det {\mathbf w}_H(x)\cdot
\det {\mathbf w}_{\tw}(x) = \mathrm{const}$
when $\fm^+_{\tw}\ne0$ or  $\det {\mathbf w}_H(x)\equiv
\mathrm{const}$ otherwise,
\end{mlist}
then $\widetilde\omega \in \cR(G \times W^+)$.

Conversely, any $2$-form as in~\emph{(\ref{eq.5.4})}
determined by a vector-function
$\mathbf{a} \colon W^+\to \fg_H$ as in~\emph{(\ref{eq.5.3})}
for which conditions $(1)\!\!-\!\!(3)$ hold, belongs to
$\cK(G\times W^+)$ and if in addition $(4)$ holds,
it belongs to $\cR(G\times W^+)$.
\end{theorem}

\begin{proof}
The following lemma is crucial for our proof.
\begin{lemma}\label{le.5.2}
Suppose that the Lie group $G$ is semisimple.
Let $\omega_0$ be a $G$-invariant closed $2$-form on $G$. Then
there exists a unique vector $a\in\fg$ such that $\omega_0=\omega^a$.
The form $\omega^a$, $a\in\fg$, on $G$ is right $H$-invariant
if and only if $\Ad(H)(a)=a$.
The kernel of such a right $H$-invariant form $\omega^a$
contains the vector fields $\xi^{l}$, for all $\xi\in\fh$.

The map $a\mapsto\omega^a$, $a\in\fg$, is an injection.
\end{lemma}
\begin{proof} (Of the lemma.)
Since the $G$-invariant form $\omega_0$ is closed, we have
$$
0=\rd\omega_0(\xi_1^{l},\xi_2^{l},\xi_3^{l})=
-\omega_0([\xi_1^{l},\xi_2^{l}],\xi_3^{l})
-\omega_0([\xi_2^{l},\xi_3^{l}],\xi_1^{l})
-\omega_0([\xi_3^{l},\xi_1^{l}],\xi_2^{l}),
$$
i.e.\ the map $c \colon \fg\times\fg\to\bbR$,
 $c(\xi,\eta)=\omega_0(\xi^{l},\eta^{l})(e)$,
is a cocycle on the Lie algebra $\fg$.
The cocycle $c$ determines the central extension
of $\fg$ (i.e.\ the linear space $\fg\oplus \bbR$, which equipped with
the commutator $[(\xi,0),(\eta,0)]=([\xi,\eta],c(\xi,\eta))$
and $[(\xi,0),(0,z)]=0$ is a Lie algebra).

Since any central extension of a semisimple Lie algebra
$\fg$ is trivial, this cocycle is exact. Indeed,
by the Malcev theorem
for the radical $\bbR$ there exists some complement
which is an algebra, evidently isomorphic to $\fg$. This
complement has the basis $(\xi,\alpha(\xi))$, where $\alpha$
is a linear function on $\fg$. Since this complement is
a subalgebra, $c(\xi,\eta)=\alpha([\xi,\eta])$. In other
words, $c(\xi,\eta)=\langle a, [\xi,\eta]\rangle$ for some
vector $a\in\fg$.
But $[\fg,\fg]=\fg$ (the algebra is semisimple)
and, consequently, such a vector $a$ is unique.
Now by the $G$-invariance, $\omega_0=\omega^a$.

It is easy to see (using Definition~\ref{eq.5.2})
that $r_h^*\omega^a=\omega^{\Ad_h a}$
for any $h\in H\subset G$.
Using the relation $[\fg,\fg]=\fg$ again we obtain that
the map $a\mapsto\omega^a$, $a\in\fg$, is an injection.
Therefore $\omega^a$ is a right $H$-invariant form on $G$ if
and only if $\Ad(H)(a)=a$.

By the invariance of the form $\langle\cdot ,\cdot \rangle$
on $\fg$, that is,
\begin{equation}
\label{invf.5.7}
\langle a,[\xi,\eta] \rangle=\langle [a,\xi],\eta \rangle,
\end{equation}
the kernel of $\omega^a$ is generated by the vector fields
$\xi^{l}$, where $\xi$ is an element of the centralizer
$\fg_a$ of $a$ in $\fg$. But if $\Ad(H)(a) =a$ then $[\fh,a]=0$,
that is, $\fh\subset \fg_a$.
\end{proof}

An arbitrary
$G$-invariant 2-form $\tio$ on $G\times W^+$ reads
\begin{equation*}
\tio=\mathrm{pr}_2^*(q)+\sum_{k=1}^{N}\sum_{j=1}^{r}
b^k_j\cdot \rd x_j \land \tith^{e_k}+
\sum_{1\leqslant s<k\leqslant N}
b^{sk}\cdot \tith^{e_s} \land \tith^{e_k},
\end{equation*}
where $q$ is a 2-form on $W^+$, and
$b^k_j$, $b^{sk}$ are smooth real-valued functions
\mbox{on $W^+$}.

Suppose that $\tio\in\cK(G\times W^+)$, that is, the form
$\tio$ satisfies conditions (1)-(5) of Theorem~\ref{th.4.8}. It is
easy to verify that if the form
$\tio$ is closed then the form $q$ is closed and each form
$\widetilde\Delta(x)=\sum_{1\leqslant s<k\leqslant N} b^{sk}(x)\cdot
\tith^{e_s} \land \tith^{e_k}$
for arbitrary but fixed
$x \in W^+$ also is closed as a form on $G$.
Then by Lemma~\ref{le.5.2} there exists a unique (smooth)
vector-function $\mathbf{b}\colon W^+\to\fg$,
$\mathbf{b}(x)=\sum_{k=1}^{N}b^k(x)e_k$, such that
$\Delta(x)=\tio^{{\mathbf b}(x)}$, that is,
\begin{equation*}
\tio=\mathrm{pr}_2^*(q)+\sum_{k=1}^{N}\sum_{j=1}^{r}
b^k_j\cdot \rd x_j \land \tith^{e_k}+
\sum_{k=1}^{N} b^k\cdot \tio^{e_k},
\quad \rd q=0.
\end{equation*}
The form $\tio$ is closed if and only if
\begin{equation*}
\begin{split}
\rd\tio
&=\sum_{k=1}^{N}\sum_{j=1}^{r}\left(
\rd b_j^k\land  \rd x_j\land \tith^{e_k}
+b_j^k\cdot\rd x_j \land \tio^{e_k}\right)
+\sum_{k=1}^{N}\rd b^k\land \tio^{e_k} \\
&=\sum_{k=1}^{N}\biggl(
\Bigl(\sum_{j=1}^{r} b_j^k \cdot \rd x_j +\rd b^k\Bigr)
\land \tio^{e_k}+
\Bigl(\sum_{j=1}^{r} \rd b_j^k  \land \rd x_j\Bigr)
\land \tith^{e_k}\biggr)=0,
\end{split}
\end{equation*}
because
$\rd (\rd x_j \land \tith^{e_k})=-\rd x_j\land \rd \tith^{e_k}$
and by~\eqref{eq.5.1}, \eqref{eq.5.2} we have
$\rd\theta^{e_k}=-\omega^{e_k}$.
Since the (closed) 2-forms $\{\omega^{e_k}\}$
and 1-forms $\{\theta^{e_k}\}$  are linearly
independent forms on $G$ (see Lemma~\ref{le.5.2}), we
see that $\sum_{j=1}^{r} b_j^k \rd x_j+\rd b^k=0$ and
$\sum_{j=1}^{r}\rd b_j^k \land \rd x_j= 0$
for all $k=1, \dotsc, N$. However the second relation
is the differential of the first one, i.e.\
these two sets of relations are equivalent to the relations
\begin{equation*}
\quad b^k_j(x)=-\frac{\partial b^k}{\partial x_j}(x),\quad j=1, \dotsc, r,
\;\;\, k=1, \dotsc, N; \;\; x\in W^+.
\end{equation*}
Then
\begin{equation*}
\tio=\mathrm{pr}_2^*(q)+\sum_{k=1}^{N}\sum_{j=1}^{r}
-\frac{\partial b^k}{\partial x_j} \cdot \rd x_j \land
\tith^{e_k}
+\sum_{k=1}^{N} b^k \tio^{e_k}=\mathrm{pr}_2^*(q)+
\rd\Biggl(-\sum_{k=1}^{N} b^k \cdot \tith^{e_k} \Biggr).
\end{equation*}

Since $\tio\in\cK(G\times W^+)$,
$\tio(Z^{X_s},Z^{X_p})=0$, where, recall,
$Z^{X_j}= (X_j^l$, $- \ri\, \partial/\partial x_j)$ and
$(X_1, \dotsc, X_r)$ is the given basis of
$\fa$. Now, the subalgebra
$\fa\subset\fg$ is commutative. Thus the restriction
$\omega^{\mathbf{b}(x)}(e)|_\fa$
vanishes (see~(\ref{eq.5.2})) and,
consequently,
$\bigl(\sum_{k=1}^{N} b^k\cdot \tio^{e_k}\bigr)(Z^{X_s},Z^{X_p})=0$.
Taking into account that
$\mathrm{pr}_2^*(q)(Z^{X_s}$, $Z^{X_p})\in\bbR$ and
$(\rd x_j \land \tith^{e_k})(Z^{X_s},Z^{X_p})\in\ri\bbR$,
we obtain that
\begin{align*}
\mathrm{pr}_2^*(q)(Z^{X_s},Z^{X_p})=-q\Big(\frac{\partial}{\partial x_s},
\frac{\partial}{\partial x_p}\Big)=0,
\\
\Biggl(\sum_{k=1}^{N}\sum_{j=1}^{r}
-\frac{\partial b^k}{\partial x_j}  \cdot
\rd x_j \land \tith^{e_k}\Biggr)
(Z^{X_s},Z^{X_p})=0.
\end{align*}
Therefore, the form $q$ vanishes on $W^+$
and
$\displaystyle
\frac{\partial b^p}{\partial x_s}(x)
=\frac{\partial b^s}{\partial x_p}(x)$
for all $x\in W^+$, $s,p\in\{1, \dotsc, r\}$.
Since the domain
$W^+\subset\bbR^r$ is convex,
there exists a unique (up to a
real constant) smooth function
$f \colon W^+\to\bbR$ such that
\begin{equation*}
b^k(x) = - \frac{\partial f}{\partial x_k}(x),\;\; k=1, \dotsc, r, \;
x\in W^+.
\end{equation*}
In other words, a $G$-invariant $2$-form $\tio$ on $G\times W^+$
is closed and
$\tio(Z^{X_s},Z^{X_p})=0$, $s,p\in\{1, \dotsc, r\}$,
if and only if
\begin{equation*}
\tio=\rd \tith^{\mathbf a},\quad
{\mathbf a}(x)=\sum_{k=1}^{r}\frac{\partial f}{\partial x_k}(x)X_k+
\mathbf{a}^\bot(x),
\end{equation*}
where $\mathbf{a}^\bot \colon W^+\to\fm^+
\oplus\fk$  is a
smooth vector-function. It is clear that
\begin{equation}\label{eq.5.8}
\tio=\sum_{j=1}^r\rd x_j\land\tith^{{\mathbf{a}_{[j]}}}-\tio^{\mathbf{a}},
\ \text{where}\quad
\mathbf{a}_{[j]}=\frac{\partial \mathbf{a}}{\partial x_j}=
\sum_{k=1}^{r}\frac{\partial^2 f}{\partial x_j\partial x_k}X_k+
\frac{\partial \mathbf{a}^\bot}{\partial x_j}.
\end{equation}
This form is right $H$-invariant if
$
\tio=r^*_h \tio=
\sum_{j=1}^r\rd x_j\land\tith^{\Ad_h \mathbf{a}_{[j]}}-\tio^{\Ad_h
\mathbf{a}}
$ for all $h\in H$.
Since the maps $a\mapsto \theta^a$ and $a\mapsto \omega^a$,
$a\in\fg$, are injections (cf.\ Lemma \ref{le.5.2})
we obtain that $\mathbf{a}_{[1]}(x),\dotsc,\mathbf{a}_{[r]}(x),
\mathbf{a}(x)\in\fg_H$ (see Definitions~(\ref{eq.4.13}) and~(\ref{eq.4.14})).
In this
case by Lemma~\ref{le.5.2} the kernel of the form
$\tio^{\mathbf a}$ contains the subbundle
$\sH\subset TG\times TW^+$. It is easy now to verify that
the kernel of $\tio$ contains $\sH$ if and only if
$\langle \mathbf{a}_{[j]}(x),\fh \rangle=0$ for all
$x\in W^+$ and $j=1,\dotsc ,r$. Thus ${\mathbf a}_{[j]}(x)\bot\fz(\fh)$
for all $x\in W^+$ because $\fg_H\cap\fh=\fz(\fh)$.
This means that the $\fz(\fh)$-component
of the vector ${\mathbf a}(x)$ is a constant.
Taking into account Proposition~\ref{pr.4.4} and Remark~\ref{re.4.5}
we obtain that $\mathbf{a}(x)=a^\fa(x)+z_\fh+a^\fk(x)+a^\fm(x)$, where
\begin{equation}\label{eq.5.9}
\begin{split}
a^\fa(x)&=\sum_{j=1}^{r}\frac{\partial f}{\partial x_j}(x)X_j,
\quad z_\fh\in\fz(\fh),\\
a^\fk(x)&=\sum_{\lambda\in \Sigma_{H}\cap\Sigma^+}
a^\fk_\lambda(x)\zeta^{1}_\lambda \in \fk_{H}^+, \quad
a^\fm(x)=\sum_{\lambda\in \Sigma_{H}\cap\Sigma^+}
a^\fm_\lambda(x)\xi^{1}_\lambda \in \fm_{H}^+.
\end{split}
\end{equation}

It is convenient now, using~(\ref{eq.5.8}),
 to calculate $\ri\tio(Z^{X_k},\overline{Z^{X_j}})$
$=\ri\tio(Z^{X_k},\overline{Z^{X_j}} - Z^{X_j})$:
\begin{equation*}
\ri\tio(Z^{X_k},\overline{Z^{X_j}})
=\ri\tio\left( \left(X_k^l,- \ri\, \dfrac{\partial}{\partial x_k}\right),
\left(0, 2\ri\, \dfrac{\partial}{\partial x_j}\right) \right)
= -\ri\langle \mathbf{a}_{[j]},X_k \rangle 2\ri
=2\frac{\partial^2 f}{\partial x_k\partial x_j}.
\end{equation*}
Using~(\ref{eq.4.30})
for the vector field $Z^{\xi}$, $\xi\in\fm^+$
we obtain, for all $(g,x) \in G \times W^+$,
\begin{equation*}
\begin{split}
\tio_{(g,x)}(Z^{X_k},&Z^{\xi})
= -\ri\langle \mathbf{a}_{[k]}(x),R_x\xi-\ri S_x\xi\rangle
-\langle \mathbf{a}(x),[X_k, R_x\xi-\ri S_x\xi]\rangle
\\ \noalign{\smallskip}
&= -\ri\langle (R_x+\ri S_x)\mathbf{a}_{[k]}(x),\, \xi\rangle
+\langle (R_x+\ri S_x)\ad_{X_k}\mathbf{a}(x),\, \xi\rangle
\\ \noalign{\smallskip}
&= -\ri\langle R_x \mathbf{a}_{[k]}(x)-S_x\ad_{X_k} \mathbf{a}(x),\,
\xi\rangle
+\langle S_x \mathbf{a}_{[k]}(x)+R_x\ad_{X_k} \mathbf{a}(x),\,\xi\rangle.
\end{split}
\end{equation*}
Thus $\tio(Z^{X_k},Z^{\xi})=0$ for all $k=1,\dotsc,r$,
and $\xi\in\fm^+$ if and only if
\begin{equation}
\label{eq.5.10}
R_x a^\fm_{[k]}(x)-S_x\ad_{X_k} a^\fm(x)=0
\quad\text{and}\quad
S_x a^\fk_{[k]}(x)+R_x\ad_{X_k} a^\fk(x)=0
\end{equation}
for all $x\in W^+$, because $[\fa,\fa\oplus\fz(\fh)]=0$ and~(\ref{eq.4.33}) hold.
 In other words, we obtain
the following differential relations:
\begin{equation*}
\frac{\partial a^\fm_\lambda(x)}{\partial x_k}=
-\frac{\lambda'(X_k)\cdot\cosh \lambda'_x}{\sinh \lambda'_x}\,
a^\fm_\lambda(x)
\quad\text{and}\quad
\frac{\partial a^\fk_\lambda(x)}{\partial x_k}=
-\frac{\lambda'(X_k)\cdot\sinh \lambda'_x}{\cosh \lambda'_x}\,
a^\fk_\lambda(x),
\end{equation*}
with solutions
\[
a^\fm_\lambda(x) =\dfrac{c^\fm_\lambda}{\sinh \lambda'_x}, \; \quad
\; a^\fk_\lambda(x) = \dfrac{c^\fk_\lambda}{\cosh \lambda'_x}, \; \qquad
c^\fm_\lambda, c^\fk_\lambda \in \bbR,\
\lambda \in \Sigma_{H}\cap\Sigma^+.
\]
In this place it is convenient to calculate also
$\ri\tio(Z^{X_k},\overline{Z^{\xi}})(g,x)$,
$\xi\in\fm^+$:
\begin{equation*}
\begin{split}
\ri\tio_{(g,x)}(Z^{X_k},\overline{Z^{\xi}} -  Z^\xi)
\stackrel{\mathrm{(\ref{eq.5.8})}}=
&  \ri\bigl(-\ri \langle \mathbf{a}_{[k]}(x), \, 2\ri S_x\xi\rangle
-\langle \mathbf{a}(x),\, [X_k, \,2\ri S_x\xi]\rangle\bigr)  \\
\stackrel{\mathrm{(\ref{eq.5.9})}}=
& 2\ri\bigl\langle -S_x a^\fk_{[k]}(x), \,\xi\bigr\rangle
+2\langle a^\fm(x),\, [X_k, \,S_x\xi]\rangle \\
\stackrel{\mathrm{(\ref{eq.5.10})}}=
&2\ri\langle R_x \ad_{X_k} a^\fk(x), \,\xi\rangle
+2\langle a^\fm(x), \,(\ad_{X_k} S_x)\xi\rangle \\
\stackrel{\mathrm{(\ref{eq.4.32})}}=
&-2\ri\langle a^\fk(x),\, (\ad_{X_k}R_x)\xi\rangle
+2\langle a^\fm(x),\,(\ad_{X_k} S_x)\xi\rangle.
\end{split}
\end{equation*}
Thus, $\ri\tio(Z^{X_k},\overline{Z^{\xi}})=0$ if
$\xi\in\fm_{\tw}^+$ and for vectors
$\xi^1_\lambda\in \fm_H^+$,
$\lambda\in \Sigma_{H}\cap\Sigma^+$:
\begin{equation*}
\ri\tio_{(g,x)}(Z^{X_k},\overline{Z^{\xi^{1}_\lambda}})
\stackrel{\mathrm{(\ref{eq.4.34})},\mathrm{(\ref{eq.4.6})}}=
2\ri \frac{\lambda'(X_k)}{\cosh \lambda'_x}\,
\bigr\langle a^\fk(x),\, \zeta^{1}_\lambda \bigl\rangle
-2\frac{\lambda'(X_k)}{\sinh \lambda'_x}\,
\bigl\langle a^\fm(x), \, \xi^{1}_\lambda\bigr\rangle.
\end{equation*}

Using the invariance \eqref{invf.5.7} of the scalar product, the
properties~(\ref{eq.4.32}),~(\ref{eq.4.33})
of the operator-functions $R$, $S$
and the commutation relations~(\ref{eq.4.1}),
we calculate $\tio(Z^{\xi},Z^{\eta})=-\tio^{{\mathbf a}(x)}(Z^{\xi},
$ $Z^{\eta})$,
$\xi,\eta\in\fm^+$, putting $A(x)=\ad_{\mathbf{a}(x)}$,
$A^\fm(x)=\ad_{a^\fa(x)+a^\fm(x)}$, $A^\fk(x)=\ad_{a^\fk(x)+z_\fh}$:
\begin{equation*}
\begin{split}
\tio(Z^{\xi},Z^{\eta})
&=-\langle \mathbf{a},\, [R\xi-\ri S\xi, \, R\eta-\ri S\eta]\rangle\\
&=-\langle \mathbf{a},\, [R\xi,\, R\eta]-[S\xi,\,S\eta]\rangle
+\ri\langle \mathbf{a}, \, [R\xi,\, S\eta]+[S\xi,\, R\eta]\rangle \\
&\stackrel{\eqref{invf.5.7}}=
-\langle [\mathbf{a},\,R\xi], \, R\eta\rangle
+\langle[\mathbf{a},\,S\xi],\,S\eta\rangle
+\ri\langle [\mathbf{a},\,R\xi], \,S\eta\rangle
+\ri\langle[\mathbf{a},\,S\xi], \,R\eta\rangle \\
& \stackrel{\mathrm{(\ref{eq.4.32})}}=
-\langle (RAR+SAS)\xi,\,\eta\rangle+\ri\langle (RAS-SAR)\xi,\,\eta\rangle\\
& \stackrel{\mathrm{(\ref{eq.4.1})},\mathrm{(\ref{eq.4.33})}}=
-\big\langle (RA^\fk R+SA^\fk S)\xi,\,\eta \big\rangle
+\ri \big\langle (RA^\fm S-SA^\fm R)\xi,\,\eta \big\rangle.
\end{split}
\end{equation*}

Since the algebra $\fa$ is commutative,
$[R_x,\,S_x]=[R_x,\,\ad_{a^\fa(x)}] = [S_x,\,\ad_{a^\fa(x)}]=0$ on
$\fg$ for any $x\in W^+$. Similarly, from $[\fa, \fz(\fh)] = 0$
it follows that
$[R_x, \ad_{z_\fh}] =  [S_x, \ad_{z_\fh}] = 0$ on $\fg$ for any
$x \in W^+$.  Thus
$\tio(Z^{\xi},\,Z^{\eta}) = 0$ for all
$\xi,\eta\in\fm^+$ if and only if for all
$x\in\ W^+$,
Equations \eqref{eq.5.5} hold, because
relations~(\ref{eq.4.1}), (\ref{eq.4.33}) hold and
the space $\fa\oplus\fh$ ($\fa\subset\fm$, $\fh\subset\fk$)
is the kernel of $R_x$ and $S_x$ ($R_x(\fm)=\fm^+$, $S_x(\fk)=\fm^+$,
that is, $R_x(\fa)=0$, $S_x(\fh)=0$).

If $a^\fk(x)=0$ then the condition~(\ref{eq.5.5}) implies
$(R_x^2+S_x^2)\ad_{z_\fh}(\fm^+)=0$. Since for each
$x\in W^+$ by~(\ref{eq.4.15}),
$\ad_{z_\fh}(\fm_\lambda)\subset\fm_\lambda$
and  $\cosh^{-2} \lambda'(x)-\sinh^{-2} \lambda'(x)\ne 0$,
we obtain that $\ad_{z_\fh}(\fm^+)=0$. But $\ad_{z_\fh}(\fa)=0$
by its definition, that is, $[z_\fh,\fm]=0$. If $G/K$ is an irreducible
symmetric space then the algebra $\fg$ is generated by $\fm$
and therefore $z_\fh=0$.

Next we calculate also the value
$\ri\tio(Z^{\xi},\overline{Z^{\eta}})
= \ri\tio(Z^{\xi},\overline{Z^{\eta}} -Z^\eta)$:
\begin{equation}\label{eq.5.11}
\begin{split}
&\ri\tio(Z^{\xi},\overline{Z^{\eta}})
= -\ri\langle \mathbf{a},[R\xi-\ri S\xi,\,2\ri S\eta]\rangle \\
&=-2\ri\langle \mathbf{a},[S\xi,\,S\eta]\rangle
+2\langle \mathbf{a},[R\xi,\,S\eta]\rangle \\ \noalign{\smallskip}
&=
-2\ri\langle [\mathbf{a},\,S\xi],S\eta\rangle
+2\langle [\mathbf{a},\,R\xi],S\eta\rangle
=2\ri\langle (SAS)\xi,\,\eta\rangle-2\langle (SAR)\xi,\,\eta\rangle\\
\noalign{\smallskip}
& =2\ri \big\langle \big(S(\ad_{a^\fk+z_\fh})S \big)\xi,\,\eta\rangle
-2\big\langle \big( S(\ad_{a^\fa+a^\fm})R \big)\xi,\,\eta\big\rangle
\end{split}
\end{equation}
or well, for any vectors $\xi^j_\lambda$, $\xi^k_\mu$,
$\lambda,\mu\in \Sigma^+$, Equations~\eqref{eq.5.6} hold.

Since by~(\ref{eq.4.33}) the subspace
$\fm_{H}^+\oplus\fk_{H}^+\subset\fg_{H}$
 is $R_x$, $S_x$ invariant,
$[\fg_{H},\fm_{H}^+\oplus\fk_{H}^+]\subset\fg_{H}\bot
\fm_{\tw}^+\oplus\fk_{\tw}^+$,
$[\fz(\fh),\fg_H]=0$, then from~(\ref{eq.5.11}) we obtain that
$\ri\tio(Z^{\xi},\overline{Z^{\eta}})=0$ for
all $\xi\in \fm_H^+$ and $\eta\in \fm_{\tw}^+$.

Now all assertions of Theorem~\ref{th.5.1} follows from
Theorem~\ref{th.4.8}.
\end{proof}

Also Theorem~\ref{th.4.8} immediately implies the following
\begin{corollary}
Let $G/K$ be a Riemannian symmetric space of compact type.
Each $G$-invariant K\"ahler metric ${\bf g}$, associated with
the canonical complex structure $J^K_c$ on
$G/H \times W^+ \cong T^+(G/K)$ $(\,T^+(G/K)$ is an open
dense subset of $T(G/K)\,)$,
is determined precisely by the K\"ahler form
$\omega(\cdot, \cdot) = {\bf g}(-J^K_c \cdot, \cdot)$
on $G/H \times W^+$ given by
\[
(\pi_H \times \mathrm{id})^\ast \omega = \rd\widetilde\theta^{\mathbf{a}},
\]
where $\mathbf{a}$ is the unique smooth vector-function
$\mathbf{a} \colon \bbR^+ \to \fg_H$ in \eqref{eq.5.3} satisfying
Conditions
$(1)\!\!-\!\!(3)$ of \emph{Theorem~\ref{th.5.1}}.
If, in addition, condition $(4)$
of \emph{Theorem~\ref{th.5.1}} holds,
this metric ${\mathbf g}$ is Ricci-flat.
\end{corollary}

\begin{corollary}\label{co.5.4}
The $G$-invariant function $Q \colon
G/H\times W^+\to \bbR$, $Q(gH,x)=2f(x)$,
where $f\in C^\infty(W^+,\mathbb{R})$,
is a potential function of the K\"ahler structure $(\omega,J^K_c)$
on $G/H\times W^+$
(equivalently $(\pi_H \times \mathrm{id})^\ast \omega \in
\cK(G\times W^+)$$)$
if and only if
$(\pi_H \times \mathrm{id})^\ast \omega= \rd\tilde\theta^{\mathbf{a}}$,
where $\mathbf{a} \colon W^+\to W^+$,
$\mathbf{a}(x)= \sum_{k=1}^r \frac{\partial f}{\partial x_k}(x) X_k$,
is a $W^+$-valued  vector-function such that for all $x\in W^+$
the matrix $\displaystyle \Big(\frac{\partial^2 f}{\partial x_j
\partial x_k}(x)\Big)$ is positive-definite.

This K\"ahler structure with $G$-invariant potential function
$Q$ is Ricci-flat K\"ahler
(equivalently $(\pi_H \times \mathrm{id})^\ast \omega \in
\cR(G\times W^+)$$)$
if and only if
\begin{equation}\label{eq.5.12}
\det\left(\frac{\partial^2 f}{\partial x_j\partial x_k}\right)
\cdot\prod_{\lambda\in\Sigma^+}
\left( \frac{2\lambda'(\mathbf{a})}{\sinh 2\lambda'(\mathbf{a})}
\right)^{m_\lambda}
\equiv\mathrm{const}.
\end{equation}
\end{corollary}
\begin{proof}
 Let $f\in C^\infty(W^+,\mathbb{R})$
be an arbitrary function. Consider the form
$\ri \bar\partial\partial Q$.
By definition, $\partial Q|_{F}=\rd Q|_{F}$ and
$\partial Q|_{\overline{F}}=0$. Denote by
$\widetilde\Delta$ the $1$-form
$(\pi_H\times \mathrm{id})^*(\partial Q)$ on
$G\times W^+$.  By~\eqref{eq.4.29},
the form
$\widetilde\Delta$ is the unique $1$-form on
$G\times W^+$ such that
\begin{equation*}
\begin{split}
\widetilde\Delta_{(g,x)}\big(X^l_k(g), - \ri\,
\tfrac{\partial}{\partial x_k}(x) \big)
&= -2\ri \tfrac{\partial f}{\partial x_k}(x),
\quad
\widetilde\Delta_{(g,x)}\big(X^l_k(g), \ri\,
\tfrac{\partial}{\partial x_k}(x)  \big)=0,  \\
\widetilde\Delta_{(g,x)}\big(\eta^l(g),0\big)
& =0 \quad\text{for all }\ \eta\in\fg,
\ \langle \eta,\fa \rangle=0,
\quad k=1,\ldots,r.
\end{split}
\end{equation*}
It is easy to verify that
$\widetilde\Delta= -\ri \tith^{\mathbf{a}}+\frac12\rd Q$:
$$\textstyle
\widetilde\Delta_{(g,x)}\big(\xi^l(g),\,
\sum_{k=1}^r t_k\tfrac{\partial }{\partial x_k}(x)\big)
=
-\ri \big\langle \sum_{k=1}^r \frac{\partial f}{\partial x_k}(x) X_k
,\,\xi \big\rangle+
\sum_{k=1}^r t_k\tfrac{\partial f}{\partial x_k}(x),
$$
for all $\xi\in\fg, \; t_k\in\bbR$.
Thus $\ri \cdot \rd\widetilde\Delta= \rd \tith^{\mathbf{a}}$.
The form $\rd \tith^{\mathbf{a}}=
\sum_{k=1}^r\rd x_k\land\tith^{{\mathbf{a}_{[k]}}}-\tio^{\mathbf{a}}
$ is right $H$-invariant because $\mathbf{a}(x)\in\fa\subset\fg_H$.
Its kernel contains kernel~\eqref{eq.4.28} of the submersion
$\pi_H\times \mathrm{id}$ because $\langle \fa,\fh \rangle=0$
and $[\fa,\fh]=0$.
Therefore there exists a
unique 2-form $\omega$ on $G/H\times W^+$ such that
$ \rd \tith^{\mathbf{a}}=(\pi_H\times \mathrm{id})^* \omega$.
Since
$$
\rd \tith^{\mathbf{a}}=
\ri \cdot \rd\widetilde\Delta
\eqdef \ri \cdot \rd \big((\pi_H\times \mathrm{id})^*(\partial Q)\big)
=\ri (\pi_H\times \mathrm{id})^*\big( \rd(\partial Q)\big)
=(\pi_H\times \mathrm{id})^*\big(\ri \bar\partial(\partial Q)\big)
$$
and $\pi_H\times \mathrm{id}$ is a submersion,  we obtain that
$\ri \bar\partial\partial Q=\omega$.

To prove that the form $\omega$ is a K\"ahler form
note that our form $\rd \tith^{\mathbf{a}}$ is a special case
of the form considered in Theorem~\ref{th.5.1}.

Indeed, choosing the vector-function $\mathbf{a}$
as in Theorem~\ref{th.5.1}
such that its components $z_\fh$, $a^\fk$, $a^\fm$
vanish identically on $W^+$
so that $\mathbf{a}(x)= \sum_{k=1}^r
\frac{\partial f}{\partial x_k}(x) X_k$,
we obtain from \eqref{eq.5.6} for
this function $\mathbf{a}$ that:
(1) the $p\times p$-matrix-function
${\mathbf w}_H(x)$, $p=\dim\fm_H$,
is diagonal except for
the first $r\times r$-block
$\big(2\frac{\partial^2 f}{\partial x_k\partial x_j}(x)\big)$,
$k,j\in\{1,\ldots,r\}$;
(2) the $s\times s$-matrix
${\mathbf w}_{\tw}(x)$, $s=\dim\fm_{\tw}^+$, is diagonal.

Then the Hermitian matrices ${\mathbf w}_H(x)$, ${\mathbf w}_{\tw}(x)$
are positive-definite if and only if
the matrix $\textstyle \left(\frac{\partial^2 f}{\partial x_j
\partial x_k}(x)\right)$ is positive-definite and
the condition
\linebreak
$w_{{}_\lambda^j|{}_\lambda^j}(x)=
2 \lambda'\bigl(\mathbf{a}(x)\bigr)/\bigl(\cosh \lambda'(x)$
$\cdot \sinh \lambda'(x)\bigr) >0$
is satisfied for all restricted roots
$\lambda\in\Sigma^+$ and $j=1,\dotsc,m_\lambda$.
Since $x \in W^+$, one has
$\sinh \lambda'(x) >0$, $\cosh \lambda'(x) >0$,
so the previous condition can be simplified to
$\lambda'\bigl(\mathbf{a}(x)) > 0$ for all
$\lambda\in\Sigma^+$, which amounts to the fact that
$\mathbf{a}(W^+)\subset W^+$.

Taking into account that condition~(\ref{eq.5.12})
is condition (4) of Theorem~\ref{th.5.1} in our special case,
we obtain the last statement of the corollary.
\end{proof}

\section{New complete invariant Ricci-flat K\"ahler
metrics on $T\mathbb{S}^2$}
\label{s.6}

Let $\fg$ be a compact Lie algebra and let $\sigma$,
$\fk$, $\fm$, $\fa$, $\Sigma$, etc. be as in
Subsection~\ref{ss.4.1}. We continue with the previous
notations but in this section it is assumed in addition that
$G/K$ is the rank-one Riemannian symmetric space $\mathbb{S}^2$,
that is $G/K= \mathrm{SO}(3)/\mathrm{SO}(2)$
\big(also $\mathbb{S}^2 \cong {\mathbb C}{\mathbf P}^1
\cong \mathrm{SU}(2)/\mathrm{S}(\mathrm{U}(1){\times} \mathrm{U}(1))\big)$.
In this case the Lie algebra $\fg$ is the algebra $\mathfrak{so}(3)$ of
skew-symmetric $3\times 3$ real matrices.
Denote by $E_{jk}$ the elementary
$3\times 3$ matrix with $1$ in the entry in the
$j$th row and the $k$th column and
$0$ elsewhere. Then the set of vectors $\{X,Y,Z\}$, where
$$
X\eqdef E_{12} - E_{21},
\quad
Y\eqdef E_{13} - E_{31}
\quad\text{and}\quad
Z\eqdef E_{23} - E_{32},
$$
is a basis of the three-dimensional Lie algebra $\fg$.
The compact Lie subalgebra
${\mathfrak k}=\bbR Z$ of the semisimple Lie algebra
${\mathfrak g}=\mathfrak{so}(3)$
is a Cartan subalgebra of $\fg$.

Fix on $\fg= \mathfrak{so}(3)$ the
invariant trace form given by $\langle B_1, B_2 \rangle=
-\tfrac12 \tr B_1B_2$, $B_1,B_2\in \mathfrak{so}(3)$.
Then all three vectors $X,Y,Z$ have the same length equal to $1$
and the space $\fm=\bbR X\oplus\bbR Y$ is the orthogonal complement
of $\fk=\bbR Z$ in $\fg$. Since $G/K$ is a rank-one symmetric space,
each nonzero vector from the  subspace $\fm$
generates a Cartan subspace of $\fm$.
Fix the Cartan subspace $\fa=\bbR X$  of $\fm$.
It is easy to verify that
\begin{equation}\label{eq.6.1}
[X,Y]=- Z, \
[X,Z]= Y, \
[Z,Y]= X \
\text{and}\;\,
[Z,[Z,w]] =-w,\
\forall w\in\fm.
\end{equation}
From~(\ref{eq.6.1}) it follows that
the restricted root system $\Sigma=\{\pm\varepsilon\}$,
where $\varepsilon\in({\mathfrak a}^{\mathbb C})^*$
and $\varepsilon'(X)=1$, where, recall,
$\varepsilon=\ai\varepsilon'$. Also by~(\ref{eq.6.1}),
$\fm^+=\fm_\varepsilon=\bbR Y$, $\fk^+=\fk_\varepsilon=\bbR Z$
and the algebra $\fh=0$ ($\fh$ is the centralizer of $\fa=\bbR X$
in $\fk=\bbR Z$).
Therefore the centralizer $\fg_\fh$
of $\fh=0$ in $\fg$ coincides with the whole Lie algebra $\fg$.
Remark also that the domain $W^+=\{xX : x\in\bbR, \,x>0\}$ can be
naturally identified with $\bbR^+$. From~(\ref{eq.6.1}) we have
$\Ad_{\exp tZ}X = \mathrm{e}^{t\ad_Z} (X) = \cos tX - \sin tY$.

But $K=\{\exp tZ, t\in\bbR\}$ and, as it is easy to verify,
$\exp tZ=E_{11}+\cos t(E_{22}+E_{33})+\sin t Z$.
Thus the map $K\to \fm$,  $\exp tZ\mapsto \Ad_{\exp tZ}X$, is
a one-to-one map and therefore,
$H = \{e\}$, $\fm^R = \fm \backslash \{ 0\}$ and $\fg_H=\fg$.
Moreover, one obtains
\[
G \times W^+ \cong D^+ = G \times_K \fm^{R} =
(G \times_K \fm)\backslash (G \times_K \{0\})  \cong T^+\mathbb{S}^2,
\]
where $T^+\mathbb{S}^2$ is the {\it punctured\/}
tangent bundle $T^+\mathbb{S}^2 = T\mathbb{S}^2 \backslash
\{\mathrm{zero \;section}\}$ of~$\mathbb{S}^2$.
\begin{theorem}\label{th.6.1}
Let $G/K = \mathrm{SO}(3)/\mathrm{SO}(2)=\mathbb{S}^2$. A $2$-form
$\omega$ on the punctured tangent bundle $G \times W^+
\cong T^+\mathbb{S}^2$ of
$\mathbb{S}^2$ defines a $G$-invariant
K\"ahler structure, associated to the canonical complex structure
$J^K_c$, and the corresponding metric
$\omega(J^K_c \cdot,\cdot)$ is  Ricci-flat,
if and only if $\omega$ on
$G \times W^+$ is expressed as
$\omega = \rd \tilde\theta^{\mathbf a}$,
where the vector-function
${\mathbf a}(x) = f'(x)X + \tfrac{c_Z}{\cosh x}Z$,
$c_Z$ being an arbitrary real number and
\begin{equation}\label{eq.6.2}
f'(x)=\sqrt{C\sinh^2 x+c_Z^2\sinh^2 x\cosh^{-2}x+C_1},
\end{equation}
for some real constants $C>0$ and $C_1 \geqslant 0$.

The corresponding $G$-invariant Ricci-flat K\"ahler metric
$\bfg =\bfg(C,C_1,c_Z)$ on
$T^+\mathbb{S}^2\cong G \times W^+$
is uniquely extendable to a smooth complete
metric on the tangent bundle
$T\mathbb{S}^2$ if and only if
$C_1 = 0$ (that is, $\lim_{x\to 0}f'(x)=0)$.
\end{theorem}

\begin{proof}
By Theorem~\ref{th.5.1} we have to describe all
vector-functions $\mathbf{a} \colon \bbR^+\to \fg$
($\fg_H=\fg$)  satisfying
Conditions $(1)\!\!-\!\!(4)$ of that theorem. Then the 2-form
$\tio=\rd \tith^{\mathbf{a}}$ belongs to the space
$\cR(G\times W^+)$. Remark here that since $H=\{e\}$,
one has $G/H=G$ and  $\tio=\omega$.

By their definitions, for $R_x\eqdef R_{xX}$
and $S_x\eqdef S_{xX}$, $xX\in W^+$,
\begin{equation}\label{eq.6.3}
\begin{split}\textstyle
R_x|_{\fm_{\varepsilon}\oplus\fk_{\varepsilon}}
=\frac{1}{\cosh x}
\mathrm{Id}_{\fm_{\varepsilon}\oplus\fk_{\varepsilon}}
\quad & \text{and}\quad \textstyle
S_x|_{\fm_{\varepsilon}\oplus\fk_{\varepsilon}}
=\frac{1}{\sinh x}
\ad_X|_{\fm_{\varepsilon}\oplus\fk_{\varepsilon}}.
\end{split}
\end{equation}

Put $\xi^{1}_{\varepsilon}=Y\in \fm_{\varepsilon}$.
In the notation of the previous subsection,
$\zeta^{1}_{\varepsilon}=Z\in \fk_{\varepsilon}$.
Now we have to verify Conditions $(1)\!\!-\!\!(4)$ of
Theorem~\ref{th.5.1} for the vector-function
\begin{equation*}
\mathbf{a}(x)=a^\fa(x)+a^\fk(x)+a^\fm(x)=
f'(x)X +f_Z(x)Z +f_Y(x)Y,
\end{equation*}
where
$$
f\in C^\infty(\bbR^+,\bbR),\quad
f_Y(x)=\frac{c_Y}{\sinh x},\quad
f_Z(x)=\frac{c_Z}{\cosh x},\quad
\ c_Y, c_Z\in\bbR.
$$
Remark here that $\fh=0$ and, consequently, the center
$\fz(\fh)=0$. Consider now Conditions~(\ref{eq.5.5}). We
have $\fm^+=\bbR Y$. Using~(\ref{eq.6.3}), we can rewrite the
first condition in~(\ref{eq.5.5}) for the vector
$Y=\xi^1_\varepsilon$ as
\begin{equation}\label{eq.6.4}\textstyle
\frac{1}{\cosh x}\cdot R_x
\bigl[f_Z(x)Z,Y\bigr]
+\frac{1}{\sinh x}\cdot S_x
\bigl[f_Z(x)Z, \ad_X Y\bigr]=0.
\end{equation}
The first term in~(\ref{eq.6.4}) vanishes because
$[Z,Y]=X\in\fa$ and
$R_x(\fa)=0$; the second term vanishes because
$\ad_X Y=-Z$.

Consider now the second condition in~(\ref{eq.5.5}).
Using~(\ref{eq.6.3}) again, we can rewrite this
condition for the vector $Y\in\fm_\varepsilon$ as
\begin{equation}\label{eq.6.5}\textstyle
\frac{1}{\sinh x}\cdot R_x
\bigl[f_Y(x)Y,\ad_X Y\bigr]
-\frac{1}{\cosh x}\cdot S_x
\bigl[f_Y(x)Y, Y\bigr]=0.
\end{equation}
The first term vanishes because
$\ad_X Y=-Z$, $[Y,-Z]= X\in\fa$ and $R_x(\fa)=0$.
Thus condition~(\ref{eq.6.5}) holds.

It is easy to verify that the $2\times 2$ Hermitian matrix
${\mathbf w}_H(x)$ (see Theorem \ref{th.5.1},$(2)$) is the matrix
with entries
\begin{equation}\label{eq.6.6}
w_{11}(x)=2f''(x), \
w_{1|{}_\varepsilon^{1}}(x)=2\Bigl(\ri  \tfrac{c_Z}{\cosh^2 x}
-\tfrac{c_Y}{\sinh^2 x} \Bigr), \
w_{{}_\varepsilon^{1}|{}_\varepsilon^{1}}(x)
 = \tfrac{2f'(x)}{\cosh x \sinh x}.
\end{equation}
Calculating the determinant of the Hermitian matrix
${\mathbf w}_H(x)$ (as $\fm_\tw^+=0$ in our case)
we obtain that by Theorem
\ref{th.5.1},$(4)$, the corresponding form
$\tio=\rd \tith^{\mathbf{a}}$ belongs to the space
$\cR(G\times W^+)$ if and only if
\begin{equation}\label{eq.6.7}
f''(x)>0 \quad\text{and}\quad
f''(x)f'(x)
=\Bigl(C+\frac{c_Z^2}{\cosh^4 x}+\frac{c_Y^2}{\sinh^4 x}\Bigr)
\cosh x\sinh x
\end{equation}
for all $x\in\bbR^+$ and for some constant $C\in\bbR^+$. Then
\begin{equation}\label{eq.6.8}
f'(x)=\sqrt{
C\cosh^2 x-c_Z^2\cosh^{-2} x
-c_Y^2\sinh^{-2} x+C_2},
\end{equation}
where $C_2\in \bbR$. By~(\ref{eq.6.7}),
$f''(x)>0$ if and only if $f'(x)>0$. Therefore there exists
a solution of~(\ref{eq.6.7}) on the whole semi-axis if and only if $c_Y=0$.
Putting $C_2=c_Z^2-C+C_1$ one can rewrite~(\ref{eq.6.8})
(with $c_Y=0$)
in the form~(\ref{eq.6.2}).

Let us prove the last statement of the theorem. By its definition,
$\tio=\rd \widetilde\theta^{\mathbf{a}}$.
Since $\mathbf{a}(x)=f'(x)X+\frac{c_Z}{\cosh x}Z$, by the
expression~\eqref{eq.5.4} at the point
$(g,x)\in G\times W^+$ ($W^+=\bbR^+$)
we have
\begin{equation}\label{eq.6.9}
\begin{split}
\tio_{(g,x)} &\big((\xi_1^l(g),t_1\tfrac{\partial}{\partial x}),
(\xi_2^l(g),t_2\tfrac{\partial}{\partial x})\big)=
-\bigl\langle f'(x)X+\tfrac{c_Z}{\cosh x}Z,[\xi_1,\xi_2] \bigr\rangle \\
&
+f''(x)\Bigl(t_1\bigl\langle  X, \xi_2 \bigr\rangle
-t_2\bigl\langle X, \xi_1 \bigr\rangle\Bigr) -
\tfrac{c_Z\sinh x}{\cosh^2 x}\Bigl(t_1\bigl\langle Z, \xi_2 \bigr\rangle
-t_2\bigl\langle Z, \xi_1 \bigr\rangle\Bigr),
\end{split}
\end{equation}
where $\xi_1,\xi_2\in\fg=T_eG$ and
$t_1,t_2\in\bbR$.
Our aim is to find the expression for the form
$\omega^R=((f^+)^{-1})^*\omega$ on the space
$G\times_K\fm^R\cong T^+(G/K)$ where, recall,
$f^+\!\colon G/H\times W^+\to G\times_K\fm^R$ is a
$G$-equivariant diffeomorphism.
But by the diagram~\eqref{eq.4.22}
there exists a unique form $\widetilde\omega^R$ on $G\times\fm^R$
such that
\begin{equation}\label{eq.6.10}
\widetilde\omega^R=\pi^*\omega^R
\quad\text{and}\quad
\tio=\mathrm{id}^*\widetilde\omega^R.
\end{equation}

Thus it is sufficient to calculate the form $\widetilde\omega^R$
on the space $G\times\fm^R$.
By the second expression in~\eqref{eq.6.10},
$$
\tio^R_{(g,xX)}\big(\big(\xi_1^l(g),t_1 X\big),
\big(\xi_2^l(g),t_2 X\big)\big)=
\tio_{(g,x)}\big( \big(\xi_1^l(g),t_1\tfrac{\partial}{\partial x}\big),
\big(\xi_2^l(g),t_2\tfrac{\partial}{\partial x} \big)\big)
$$
and because $\langle X, X\rangle=1$, one gets
\begin{equation}\label{eq.6.11}
\begin{split}
&\tio^R_{(g,xX)}\big((\xi_1^l(g),t_1 X),(\xi_2^l(g),t_2 X)\big) \\
   & = -\bigl\langle \tfrac{f'(x)}{x}xX
+\tfrac{c_Z}{\cosh x}Z,[\xi_1,\xi_2] \bigr\rangle
+\bigl\langle f''(x) t_1 X, \xi_2 \bigr\rangle
-\bigl\langle f''(x) t_2 X, \xi_1 \bigr\rangle \\
& \quad\; -\tfrac{c_Z\sinh x}{x\cosh^2 x}
\Bigl(\bigl\langle t_1 X, xX \bigr\rangle \bigl\langle Z, \xi_2 \bigr\rangle
-\bigl\langle t_2 X, xX \bigr\rangle\bigl\langle Z, \xi_1 \bigr\rangle\Bigr).
\end{split}
\end{equation}
Consider on  the whole tangent space
$T_{(g,w)}(G\times\fm^R)$ ($w\ne 0$), the
following bilinear form~$\Delta$,
\begin{equation}\label{eq.6.12}
\begin{split}
&\Delta_{(g,w)} \big((\xi_1^l(g),u_1),(\xi_2^l(g),u_2)\big)=
-\bigl\langle \tfrac{f'(|w|)}{|w|} w
+\tfrac{c_Z}{\cosh |w|}Z,[\xi_1,\xi_2] \bigr\rangle \\
& +\bigl\langle \tfrac{f'(|w|)}{|w|}u_1
+\left(\tfrac{f'(|w|)}{|w|}\right)'
\tfrac{\langle w ,u_1 \rangle}{|w|} w, \xi_2 \bigr\rangle \\
&-\bigl\langle \tfrac{f'(|w|)}{|w|}u_2
+\left(\tfrac{f'(|w|)}{|w|}\right)'
\tfrac{\langle w ,u_2 \rangle}{|w|} w,
\xi_1 \bigr\rangle \\
& -\tfrac{c_Z\sinh |w|}{|w|\cosh^2 |w|}
\Bigl(
\bigl\langle u_1, w \bigr\rangle \bigl\langle Z, \xi_2 \bigr\rangle
-\bigl\langle u_2, w \bigr\rangle\bigl\langle Z, \xi_1 \bigr\rangle
+\bigl\langle Z, [u_1,u_2]\bigr\rangle
\Bigr),
\end{split}
\end{equation}
where $\xi_1,\xi_2\in\fg=T_eG$, $u_1,u_2\in \fm=T_w\fm^R$.
Here $|w|^2={\langle w, w\rangle}$ ($|xX|= x$).
It is clear that this form is skew-symmetric.
Since
$\left(\tfrac{f'(|w|)}{|w|}\right)'=\tfrac{f''(|w|)}{|w|}
-\tfrac{f'(|w|)}{|w|^2}$ and $[t_1 X,t_2 X]=0$,
it is easy to verify that
$$
\Delta_{(g,xX)}\big((\xi_1^l(g),t_1X),(\xi_2^l(g),t_2 X)\big)=
\tio^R_{(g,xX)}\big((\xi_1^l(g),t_1 X),(\xi_2^l(g),t_2 X)\big),
$$
i.e.\ the restrictions of $\tio^R$ and $\Delta$
to $G\times W^+$ coincide.
Now to prove that the differential forms $\tio^R$ and $\Delta$
coincide on the whole space $G\times\fm^R$ it is
sufficient to show that the form $\Delta$ is
left $G$-invariant, right
$K$-invariant and its kernel
contains (and therefore coincides with) the subbundle $\sK$
defined by relation~\eqref{eq.4.23}.

Since for each $k\in K$ the scalar product $\langle \cdot,\cdot \rangle$
is $\Ad_k$-invariant,  $\Ad_k$ is an automorphism of $\fg$  and
$\Ad_k(Z)=Z$ ($k=\exp tZ$ for some $t\in \bbR$)
whence~\eqref{eq.3.22}
holds, that is,
$\Delta$ is left $G$-invariant and right $K$-invariant.
We now prove that $\sK \subset \ker \Delta$. Taking into
account that by definition $\langle Z,\fm \rangle=0$,
$\langle Z, Z\rangle=1$,
by the invariance of the scalar product, $\langle \xi,[\xi,\eta] \rangle=0$,
$\forall \xi,\eta\in\fg$,
and by~(\ref{eq.6.1})
$$
\big\langle Z, [u_1,[w,Z]]\big\rangle
=\big\langle Z, [[Z,w],u_1]\big\rangle
=\big\langle [Z,[Z,w]],u_1\big\rangle
=-\langle w,u_1\rangle,
$$
we obtain that
\begin{equation*}
\begin{split}
& \Delta_{(g,w)}\big((\xi_1^l(g),u_1),(Z^l(g),[w,Z])\big)=
-\bigl\langle \tfrac{f'(|w|)}{|w|} w
+\tfrac{c_Z}{\cosh x}Z,[\xi_1,Z] \bigr\rangle  \\
& \qquad +\bigl\langle \tfrac{f'(|w|)}{|w|}u_1
+\left(\tfrac{f'(|w|)}{|w|}\right)'
\tfrac{\langle w ,u_1 \rangle}{|w|} w, Z \bigr\rangle
-\bigl\langle \tfrac{f'(|w|)}{|w|}[w,Z] \\
& \qquad +\left(\tfrac{f'(|w|)}{|w|}\right)'
\tfrac{\langle w ,[w,Z] \rangle}{|w|} w,
\xi_1 \bigr\rangle  \\
& \qquad -\tfrac{c_Z\sinh |w|}{|w|\cosh^2 |w|}
\Bigl(\bigl\langle u_1, w \bigr\rangle \bigl\langle Z, Z \bigr\rangle
-\bigl\langle [w,Z], w \bigr\rangle\bigl\langle Z, \xi_1 \bigr\rangle
+\bigl\langle Z, [u_1,[w,Z]]\bigr\rangle
\Bigr)  \\
& =  -\bigl\langle \tfrac{f'(|w|)}{|w|} w,[\xi_1,Z] \bigr\rangle
 -\bigl\langle \tfrac{f'(|w|)}{|w|}[w,Z], \xi_1 \bigr\rangle
 -\tfrac{c_Z\sinh |w|}{|w|\cosh^2 |w|}
\Bigl(\bigl\langle u_1, w \bigr\rangle
-\bigl\langle w, u_1\bigr\rangle\Bigr)=0.
\end{split}
\end{equation*}
Thus $\tio^R=\Delta$ on $G\times\fm^R$
($\sK =\ker \Delta$ because the form
$\omega=\tio$ is nondegenerate).
It is easy to verify that there exists
an even real analytic function, $\psi_4(x)$, on
the whole axis $\bbR$, such that
$(f'(x))^2$ is the restriction to
$\bbR^+$ of the function
\begin{equation}\label{eq.6.13}
\psi(x)=C_1+(C+c_Z^2)x^2+\psi_4(x)x^4,
\quad
\psi(x) > C_1,\ x\in \bbR\setminus\{0\}
\end{equation}
(see \eqref{eq.6.2}). Expression~\eqref{eq.6.11}
determines a smooth $2$-form at $(g,0)\in G\times \fm$ if and only if
$\lim_{x\to 0}f'(x)=0$, that is, if and only if
$C_1=0$. In this case, expression~\eqref{eq.6.12}
(which, possibly, is not the unique
expression representing the form $\tio^R$)
determines a smooth $2$-form on the whole space
$G\times \fm$ because
$\tfrac{\sqrt{\psi(x)}}{x}$ and
$\tfrac{1}{x}\Big(\tfrac{\sqrt{\psi(x)}}{x}\Big)'$
are even real analytic functions on
the whole axis.
We will denote this form (extension)
on $G\times \fm$ by
$\tio^R_0$. There exists a unique $2$-form
$\omega^R_0$ on $G\times_K\fm\cong T(G/K)$
such that $\tio^R_0=\pi^*\omega^R_0$.
The forms $\omega^R_0$ and $\omega^R$ coincide, by
construction, on the open submanifold $G\times_K\fm^R\cong T^+(G/K)$,
i.e.\ $\omega^R_0$ is a smooth extension of $\omega^R$.
But  by~(\ref{eq.6.6}) and~(\ref{eq.6.13})
for $C_1=0$
$$
\lim_{x\to 0} w_{11}(x)=2\sqrt{C+c_Z^2},
\quad
\lim_{x\to 0} w_{1|{}_\varepsilon^{1}}(x)
=2\ri  c_Z,
\quad
\lim_{x\to 0} w_{{}_\varepsilon^{1}|{}_\varepsilon^{1}}(x)
 =  2\sqrt{C+c_Z^2},
$$
i.e.\ the corresponding limit $2\times 2$ Hermitian matrix
$\lim_{x\to 0}{\mathbf w}_H(x)$ is positive-definite.
Thus by Corollary~\ref{co.4.10},
$\omega^R_0$ is the K\"ahler form
of the metric ${\mathbf g}_0$ (the extension of ${\mathbf g}$)
on $G\times_K\fm\cong T(G/K)$, for $G/K={\mathbb S}^2$.

Next, we show that the metric ${\mathbf g}_0$ on
$G\times_K\fm\cong T{\mathbb S}^2$ is complete.
To prove this, we consider again the description of the form
$\omega$
in~\eqref{eq.6.9} on the space
$G\times W^+\cong G\times_K\fm^R\cong T^+{\mathbb S}^2$
($G=\mathrm{SO}(3)$, $H=\{e\}$ and
$\tio=\omega$). For our aim it is sufficient to calculate the
distance $\mathrm{dist}(b,c)$ between the compact subsets
$G\times\{b\}$ and $G\times\{c\}$, where
$\mathrm{dist}(b,c)= \inf \{d(p_b,p_c), p_b\in G\times\{b\},
p_c\in G\times\{c\}\}$.
Since the sets
$G\times\{x\}$ are compact, it is clear that the metric
${\mathbf g}_0$ is complete if and only if for some $b>0$
one has $\lim_{c\to\infty}\mathrm{dist}(b,c)=\infty$.

To calculate the function $\mathrm{dist}(b,c)$
note that the tangent bundle $T(G/K)\cong G\times_ K\fm$
is a cohomogeneity-one manifold, i.e.\ the Lie group $G$
acts on this manifold with a codimension-one orbit.
We will use only one fundamental fact
on the structure of these manifolds~\cite{Al}:
A unit smooth vector field $U$ on a $G$-invariant
domain $D\subset T(G/K)$ which is
${\mathbf g}_0$-orthogonal to each $G$-orbit in $D$
is a geodesic vector field,
i.e.\ its integral curves are geodesics of the metric
${\mathbf g}_0$.

We now describe such a vector field $U$ on
$G\times W^+ \cong T^+(G/K)$. Put
\begin{equation}\label{eq.6.14}
f_U(x)=\left({\frac{f'(x) \cosh^3 x}{f''(x)f'(x) \cosh^3 x
- {c^2_Z}\sinh x}}\right)^{1/2},
\ x\in\bbR^+.
\end{equation}
\begin{lemma}\label{le.6.2}
Such a unit vector field $U$ on $G\times W^+$ is determined by
the expression
\begin{equation*}
U(g,x)= f_U(x)\cdot
\left(\frac{c_Z\sinh x}{f'(x) \cosh^{2} x} Y^l(g),
\,\frac{\partial}{\partial x}\right).
\end{equation*}
For the coordinate function
$x$ on $G\times W^+$ the following inequality holds
\begin{equation}\label{eq.6.15}
\big|\rd x_{(g,x)}\big(\xi^l(g),\,t\tfrac{\partial}{\partial x}\big)
\big|\leqslant f_U(x)\cdot
\big\|(\xi^l(g),t\tfrac{\partial}{\partial x})\big\|_{(g,x)},
\end{equation}
where $\big(\xi^l(g),\,t\tfrac{\partial}{\partial x}\big)\in
T_{(g,x)}(G\times W^+)$ and
$\|\cdot\|$ is the norm determined by the metric ${\mathbf g}$.
\end{lemma}
\begin{proof} (Of the Lemma.) Let us rewrite the
expression~\eqref{eq.6.9}
as
\begin{equation}\label{eq.6.16}
\begin{split}
\omega_{(g,x)} & \big(\big(\xi_1^l(g),t_1\frac{\partial}{\partial x}\big),
\big(\xi_2^l(g),t_2\frac{\partial}{\partial x}\big)\big)=
t_1\Bigl(f''(x)\bigl\langle  X, \xi_2 \bigr\rangle
-\frac{c_Z\sinh x}{\cosh^2 x}\bigl\langle Z, \xi_2 \bigr\rangle
\Bigr) \\
&
+ \bigl\langle [f'(x)X+\frac{c_Z}{\cosh x}Z,\xi_2]
-f''(x)t_2X
+\frac{c_Z\sinh x}{\cosh^2 x}t_2Z, \xi_1 \bigr\rangle.
\end{split}
\end{equation}
Therefore for $\xi_2=a X+b Y+c Z$, $a,b,c\in\bbR$, by the
commutation relations~\eqref{eq.6.1} we have
\begin{equation}\label{eq.6.17}
\begin{split}
&\omega_{(g,x)}\big(\big(\xi_1^l(g),t_1\tfrac{\partial}{\partial x}\big),
\big((a X+b Y+c Z)^l(g),t_2\tfrac{\partial}{\partial x}\big)\big)\\
&=t_1\Bigl(f''(x)a-\tfrac{c_Z\sinh x}{\cosh^2 x}c\Bigr)
+\bigl\langle
\bigl(b\tfrac{c_Z}{\cosh x}-f''(x)t_2\bigr)X, \xi_1 \bigr\rangle \\
&\quad +\bigl\langle
\bigl(c f'(x) - a\tfrac{c_Z}{\cosh x}\bigr)Y
+\bigl(\tfrac{c_Z\sinh x}{\cosh^2 x}t_2-b f'(x)\bigr)Z
, \xi_1 \bigr\rangle.
\end{split}
\end{equation}
Since the vector field  $U$ is ${\mathbf g}$-orthogonal to
the subbundle $V\subset T(G\times W^+)$ generated by the vector
fields $(\xi_1^l,0)$, $\xi_1\in\fg$,
then $U$ is $\omega$-orthogonal to the subbundle $J^K_c(V)$
generated by  $(\xi_1^l,t_1 \tfrac{\partial}{\partial x})$,
$\xi_1\in\fg$, $\langle \xi_1, X\rangle=0$, $t_1\in\bbR$,
because by~\eqref{eq.4.27},
\begin{equation}\label{eq.6.18}
J_c^K(X^l,0)=(0,\,\tfrac{\partial}{\partial x})
\quad\text{and}\quad
J_c^K(Y^l,\,0)= (-\tfrac{\cosh x}{\sinh x}Z^l,0)).
\end{equation}

Putting $U=\big((a X+b Y+c Z)^l,\,\tau  \tfrac{\partial}{\partial
x}\big)$, where $a,b,c,\tau$ are functions of $x$,
we obtain the following orthogonality conditions
$$
a f''-c\, \tfrac{c_Z\sinh x}{\cosh^2 x}=0,
\quad
c f' - a\, \tfrac{c_Z}{\cosh x}=0,
\quad
\tau\, \tfrac{c_Z\sinh x}{\cosh^2 x}-b f'=0,
$$
with the solution: $a=0$, $c=0$ and
$b=\tau \tfrac{c_Z\sinh x}{f'\cosh^2 x}$.  Thus
$U= \tau\,\left(\tfrac{c_Z\sinh x}{f'\cosh^2 x} Y^l,\,
\tfrac{\partial}{\partial x}\right)$.
Since $\|U\|\eqdef\omega(J^K_c(U),U)\equiv 1$,
then by~\eqref{eq.6.17} and~(\ref{eq.6.18})
$$
\tau^2 \omega\Bigl(\bigl(
-\tfrac{c_Z}{f'\cosh x} Z^l-X^l,
0\bigr),\bigl(\tfrac{c_Z\sinh x}{f'\cosh^2 x} Y^l,
\tfrac{\partial}{\partial x}\bigr)\Bigr)
= \tau^2 \bigl(f''-\tfrac{c_Z^2\sinh x}{f'\cosh^3 x}\bigr)\equiv 1.
$$
Thus $ U(g,x)= \left(\tfrac{f'(x) \cosh^3 x}{f''(x) f'(x)
\cosh^3 x -c_Z^2\sinh x}\right)^{1/2}
\left(\tfrac{c_Z\sinh x}{f'(x)\cosh^2 x} Y^l(g),\,
\tfrac{\partial}{\partial x}\right)$.

To prove the inequality in the statement let us find
the Hamiltonian vector field $\mathtt{H^x}$ of the
($G$-invariant) function $x$.
Putting $\mathtt{H^x}=\big((a_0 X+b_0 Y+c_0 Z)^l, \,\tau_0
\tfrac{\partial}{\partial x}\big)$, where $a_0,b_0,c_0,\tau_0$ are functions of $x$,
we obtain the following relation
$$
t_1
{=}\rd x \big(\xi_1^l, \,t_1 \tfrac{\partial}{\partial x}\big)
{\eqdef}\omega\big((\xi_1^l,t_1 \tfrac{\partial}{\partial x}), \,\mathtt{H^x}\big)
{=}\omega\big((\xi_1^l, \,t_1 \tfrac{\partial}{\partial x}),
\big((a_0 X+b_0 Y+c_0 Z)^l,\,\tau_0
\tfrac{\partial}{\partial x}\big) \big),
$$
for arbitrary $t_1\in\bbR$, $\xi_1\in\fg$.
Using~\eqref{eq.6.17} again we obtain the following equations:
$$
a_0 f''-c_0\, \tfrac{c_Z\sinh x}{\cosh^2 x}=1, \quad
b_0\, \tfrac{c_Z}{\cosh x}=\tau_0 f'', \quad
c_0 f'= a_0\, \tfrac{c_Z}{\cosh x}, \quad
\tau_0\, \tfrac{c_Z\sinh x}{\cosh^2 x}=b_0 f'
$$
with the following solution: $b_0=0$, $\tau_0=0$ and
\begin{equation}\label{eq.6.19}
a_0=c_0\cdot\tfrac{f'\cosh x}{c_Z},
\qquad
c_0=\tfrac{c_Z\cosh^2 x}{f''f'\cosh^3 x-c_Z^2\sinh x}.
\end{equation}
Thus $\mathtt{H^x}=\big((a_0 X+c_0 Z)^l, 0\big)$.
Since $J^K_c(\mathtt{H^x})=\big(c_0\tfrac{\sinh x}{\cosh x}Y^l,
a_0 \tfrac{\partial}{\partial x}\big)$
we have that
\begin{equation*}
\begin{split}
\|\mathtt{H^x}\|^2\eqdef &
\omega\bigl((c_0\tfrac{\sinh x}{\cosh x}Y^l,
a_0 \tfrac{\partial}{\partial x}),((a_0 X+c_0 Z)^l, 0)\bigr) \\
= & c_0^2 f'\tfrac{\sinh x}{\cosh x}- 2a_0c_0
\tfrac{c_Z\sinh x}{\cosh^2 x} + a_0^2f''.
\end{split}
\end{equation*}
Taking into account~\eqref{eq.6.19} we obtain that
$\|\mathtt{H^x}\|^2
=\tfrac{f' \cosh^3 x}{f'' f' \cosh^3 x -c_Z^2\sinh x}$.
Hence $\|\mathtt{H^x}\|=f_U$.
Now, by the Cauchy-Schwarz inequality for metrics one has
\begin{equation*}
\begin{split}
\big|\rd x(\xi_1^l,\,t_1\tfrac{\partial}{\partial x})\big|
= & \;
\big|\omega((\xi_1^l,\,t_1 \tfrac{\partial}{\partial x}), \mathtt{H^x})\big|
=\big|{\mathbf g}((\xi_1^l,\,t_1 \tfrac{\partial}{\partial x}),
J^K_c(\mathtt{H^x}))\big| \\ \noalign{\smallskip}
\; \leqslant & \;
\big\|J^K_c (\mathtt{H^x})\big\|\cdot \big\|(\xi^l(g),\,
t_1 \tfrac{\partial}{\partial x})\big\|
=\|\mathtt{H^x}\|\cdot \big\|(\xi^l(g),\, t_1
\tfrac{\partial}{\partial x})\big\|,
\end{split}
\end{equation*}
that is, we obtain~\eqref{eq.6.15}.
\end{proof}

Using now the vector field $U$ we shall
calculate the distance between the level sets $G\times\{b\}$ and
$G\times\{c\}$ in $G\times W^+$
with respect to the metric ${\mathbf g}$.
Let $\gamma(t)=(\widehat g(t),\widehat x(t))$, $t\in[0,T]$, be the
integral curve of the vector field $U$ with
initial point $p_b$ in $G\times\{b\}$, that is, $\widehat x(0)=b$.
There exists a function
$h$ on  ${\mathbb R}^+$ such that
the function $h(\widehat x(t))$ is linear in $t$.
It is easy to verify that
$\displaystyle h(x)=\int_{b}^{x}\frac{1}{f_U(s)}\, \rd s$, because
\begin{equation*}
\frac{\rd}{\rd t}h\big(\widehat x(t)\big)
=  h'\big(\widehat x(t)\big)\cdot \widehat x'(t)
=h'\big(\widehat x(t)\big)\cdot \rd x\big(\gamma'(t)\big)
= h'\big(\widehat x(t)\big)\cdot
\big(f_U(\widehat x(t) )\big) =1.
\end{equation*}
Suppose that $p_c\in G\times\{c\}$, where $p_c=\gamma(t_c)$,
$t_c\in[0, T]$.
Since the curve $\gamma$ is a geodesic, the length of
the curve $\gamma(t)$, $t\in[0,t_c]$, from
$p_b$ to $p_c$ is $t_c=h(x(p_c))-h \big(x(p_b)\big)=h(c)-h(b)$.
Thus $\mathrm{dist}(b,c)\geqslant h(c)-h(b)$.

For any other curve $\gamma_1(t)=\big(\widehat g_1(t),
\widehat x_1(t)\big)$,
with $\|\gamma_1'(t)\|=1$, starting at the point $p_b$,  and ending
at a point $p^1_c\in G\times\{c\}$, $p^1_c= \gamma_1(t_c^1)$
(of length $t^1_c$), we have by Lemma~\ref{le.6.2}
\begin{equation*}
\frac{\rd}{\rd t}h\big(\widehat x_1(t)\big)
=h'\big(\widehat x_1(t)\big)\cdot \rd x\big(\gamma_1'(t)\big)
\leqslant \frac{1}{f_U\big(\widehat x_1(t)\big)}\cdot
f_U(\widehat x_1(t))\big) \cdot \|\gamma'_1(t)\|=1.
\end{equation*}

Thus $h(c)-h(b)\leqslant t^1_c$ and the length $t^1_c$ of
the curve $\gamma_1$  from $p_b$ to $p^1_c$ is not
less than the length of the curve $\gamma(t)$, $t\in[0,t_c]$.
So the distance between the level surfaces $G\times\{b\}$ and
$G\times\{c\}$ is  $|h(c)-h(b)|$.

Now, since  by~\eqref{eq.6.2}
and~\eqref{eq.6.7} for $C_1=0$
\begin{equation}\label{eq.6.20}
\begin{split}
f'(x)&=\sqrt{C\sinh^2 x+c_Z^2\sinh^2 x\cosh^{-2}x}, \\
f''(x)&={\bigl(C\sinh x\cosh x
+c_Z^2 \sinh x\cosh^{-3}x\bigr)}/{f'(x)},
\end{split}
\end{equation}
we see that $f'(x)\sim \sqrt{C}\sinh x$,
$f''(x)\sim \sqrt{C}\sinh x$ and, by~(\ref{eq.6.14}),
$\frac{1}{f_U(x)}\sim\left(\sqrt{C}\sinh x\right)^{1/2} $ as
$x\to\infty$. Therefore
$\lim_{x\to\infty} h(x)=\infty$. Hence the metric
${\mathbf g}_0={\mathbf g}_0(C,c_Z, 0)$
(that is, for $C_1=0$)  on the tangent bundle
$T(G/K)\cong G\times_K\fm$ is complete for any $C>0$, $c_Z\in\bbR$.
\end{proof}

The proof of Theorem~\ref{th.6.1} above
and Corollary~\ref{co.5.4} immediately imply the following
\begin{corollary}\label{co.6.3}
Let $G/K = \mathrm{SO}(3)/\mathrm{SO}(2)=\mathbb{S}^2$. A $2$-form
$\omega$ on the punctured tangent bundle $G \times W^+
\cong T^+\mathbb{S}^2$ of $\mathbb{S}^2$
determines a $G$-invariant
K\"ahler structure, associated to the canonical complex structure
$J^K_c$  if and only if $\omega$ on
$G \times W^+$ is expressed as
$\omega = \rd \tilde\theta^{\mathbf a}$,
for the vector-function
${\mathbf a}(x) = f'(x)X + \tfrac{c_Z}{\cosh x}Z+\tfrac{c_Y}{\sinh x}Y$,
where $c_Z, c_Y\in\bbR$, $f\in C^\infty(\bbR^+,\bbR)$ and
$$
f''(x)>0,\ \forall x\in\bbR^+, \quad
\frac{f''(x)f'(x)}{\cosh x\sinh x}
-\frac{c_Z^2}{\cosh^4 x}-\frac{c_Y^2}{\sinh^4 x}>0,
\ \forall x\in\bbR^+.
$$
In particular, if $c_Z=c_Y=0$ then the function $(g,x)\mapsto 2f(x)$
on $G\times W^+$ is a potential function
of the K\"ahler structure $(\omega,J^K_c)$.
\end{corollary}

Finally, we relate the metrics ${\mathbf g}(C,c_{Z},C_{1})$
of Theorem \ref{th.6.1} with the Eguchi-Hanson and Stenzel
metrics. We will show that our metrics coincide with the
well-known (hyper-K\"ahler) Eguchi-Hanson metrics if $C_ {1} = 0$
and $c_{Z} =0$. To prove it, let us rewrite the metrics
${\mathbf g}(C,c_{Z},C_{1})$ in terms of the left $G$-invariant
forms $\theta^{X}$, $\theta^{Y}$ and $\theta^{Z}$ on the Lie
algebra $G = \mathrm{SO}(3)$. Indeed, taking into account the commutation
relations~\eqref{eq.6.1}
for any $\xi_1,\xi_2\in\fg=T_eG$, we have
\begin{equation*}
[\xi_1,\xi_2]=
-\bigl(\theta^Y_e\land \theta^Z_e\bigr)(\xi_1,\xi_2)\cdot X
+\bigl(\theta^X_e\land \theta^Z_e\bigr)(\xi_1,\xi_2)\cdot Y
-\bigl(\theta^X_e\land \theta^Y_e\bigr)(\xi_1,\xi_2)\cdot Z.
\end{equation*}
Taking into account expression~\eqref{eq.6.9}
for the K\"ahler form $\omega=\tio$,
we obtain that
\begin{equation*}\label{eq.6}
\omega
=f'(x) \theta^Y\land \theta^Z
+\tfrac{c_Z}{\cosh x} \theta^X\land \theta^Y
+ f''(x) \rd x\land \theta^X
-\tfrac{c_Z\sinh x}{\cosh^2 x} \rd x\land \theta^Z.
\end{equation*}

But $ {\mathbf g} (\cdot, \cdot)=\omega(J^K_c\cdot, \cdot)$
and therefore by~\eqref{eq.6.18},
\begin{equation*}
\begin{split}
\rd s^2 & = f''(x) (\rd x^2) + f'(x) \tfrac{\sinh x}{\cosh x} (\theta^Z)^2
+ f'(x) \tfrac{\cosh x}{\sinh x} (\theta^Y)^2  + f''(x) (\theta^X)^2 \\
\noalign{\smallskip}
& \quad - c_Z\Big( \tfrac{\sinh x}{\cosh^2 x}(\theta^X\theta^Z +
\theta^Z\theta^X)
+ \tfrac{1}{\cosh x}(\rd x \,\theta^Y + \theta^Y \rd x)\Big),
\end{split}
\end{equation*}
where the functions $f'(x)$ and
$f''(x)$ are described by expressions~\eqref{eq.6.20}.
Putting $c_Z=0$ (then $f'(x)=\sqrt C \sinh x$) we obtain
the ``diagonal'' Stenzel metric
\begin{equation*}
(1/\sqrt C)\rd s^2
=\cosh x (\rd x)^2
+\sinh x\tanh x (\theta^Z)^2
+\cosh x (\theta^Y)^2
+\cosh x(\theta^X)^2,
\end{equation*}
which for $C=1$ after the change of variable $\cosh x=(t/\ell)^2$ becomes
the Eguchi-Hanson metric with parameter $\ell$
(see Gibbons and Pope \cite[(4.17)]{GP}).

\end{document}